\documentclass[reqno,12pt,letterpaper]{amsart}
\usepackage[proof]{macros}
\usepackage{enumerate}
\usepackage{subfig,caption}
\usepackage{graphicx}				
\usepackage{amssymb}
\usepackage{amsmath}
\usepackage{amsthm,amsfonts,bm,bbm}
\usepackage{mathtools}
\mathtoolsset{showonlyrefs=true}
\usepackage{mathrsfs}
\usepackage{xcolor}
\usepackage{enumerate,enumitem}
\usepackage{leftidx}
\numberwithin{equation}{section}
\usepackage{marginnote}
\title[Backward uniqueness and polynomial stability]{Optimal backward uniqueness and polynomial stability of second order equations with unbounded damping}
\author{Perry Kleinhenz and Ruoyu P. T. Wang}
\date{}

\newcommand{\Dc}{\mathcal{D}}

\newcommand{\M}{M}

\newcommand{\ra}{\rightarrow}

\newcommand{\<}{\left\langle}
\renewcommand{\>}{\right\rangle}
\newcommand{\e}{\epsilon}

\newcommand{\nm}[1]{\| #1 \|} 

\renewcommand{\supp}{\text{supp }}

\newcommand{\Ac}{\mathcal{A}}
\newcommand{\Lc}{\mathcal{L}}

\newcommand{\Hc}{\mathcal{H}}

\newcommand{\abs}[1]{\left|#1\right|}

\newcommand{\cim}{\operatorname{Im}}
\newcommand{\cre}{\operatorname{Re}}
\newcommand{\id}{\operatorname{Id}}

\newcommand{\Acd}{\dot{\Ac}}
\newcommand{\Hcd}{\dot{\Hc}}
\newcommand{\Dcd}{\dot{\Dc}}
\renewcommand{\ker}{\operatorname{Ker}}

\newcommand{\fpow}{\alpha}

\newcommand{\spec}{\operatorname{Spec}}
\newcommand{\bigo}{\mathcal{O}}

\begin{document}

\begin{abstract}
For general second order evolution equations, we prove an optimal condition on the degree of unboundedness of the damping, that rules out finite-time extinction. We show that control estimates give energy decay rates that explicitly depend on the degree of unboundedness, and establish a dilation method to turn existing control estimates for one propagator into those for another in the functional calculus. As corollaries, we prove Schrödinger observability gives decay for unbounded damping, weak monotonicity in damping, and quantitative unique continuation and optimal propagation for fractional Laplacians. As applications, we establish a variety of novel and explicit energy decay results to systems with unbounded damping, including singular damping, linearised gravity water waves and Euler--Bernoulli plates.
\end{abstract}

\maketitle

\section{Introduction}
\subsection{Introduction}
Let $H=H_0$ be a Hilbert space and $P:H_1\rightarrow H$ be a nonnegative self-adjoint operator with compact resolvent, defined on $H_1$, a dense subspace of $H$. This implies $P$ is Fredholm and $\ker P$ is finite dimensional. The operator $P$ admits a spectral resolution
\begin{equation}
P u=\int_0^\infty\rho^2\ dE_\rho(u),
\end{equation}
where $E_\rho$ is a projection-valued measure on $H$ and $\supp E_\rho\subset [0, \infty)$. For $s\in\mathbb{R}$, define the scaling operators and the interpolation spaces via
\begin{equation}\label{Lambdadef}
\Lambda^{s}u=\int_0^\infty (1+\rho^2)^{s}\ dE_\rho(u), \ H_s=\Lambda^{-s}(H_0).
\end{equation}
Those operators $\Lambda^{-s}: H\rightarrow H_s$ are bounded from above and below, and they commute with $P$. For $s>0$, $H_{-s}$ is isomorphic to $\mathcal{L}(H_s, H)$, the dual space of $H_s$ with respect to $H$. 

Let the observation space $Y$ be a Hilbert space. We will consider damping of the form $Q^*Q$, where the control operator $Q^*\in \mathcal{L}(Y, H_{-\gamma})$ and the observation operator $Q\in \mathcal{L}(H_{\gamma}, Y)$ for $\gamma\in[0,\frac{1}{2}]$. Note that $Q^*Q$ is not necessarily a bounded operator on $H$. We consider an abstract damped second-order evolution equation:
\begin{equation}
(\partial_t^2+P+Q^*Q\partial_t)u=0.
\end{equation}
It can be written as a first-order evolution system:
\begin{equation}
\partial_t\begin{pmatrix}
u\\\partial_t u
\end{pmatrix}=\mathcal{A}\begin{pmatrix}
u\\\partial_t u
\end{pmatrix}, \ \mathcal{A}=\begin{pmatrix}
0 & 1\\
-P & -Q^*Q
\end{pmatrix}. 
\end{equation}
Here $\Ac$ generates a strongly continuous semigroup $e^{t\Ac}$ on $\mathcal{H}=H_{1/2}\times H$. We want to understand the stability of $e^{t\Ac}$ on the energy space $\Hcd=\Hc/\ker\Ac$, that is, to estimate $(\mathcal{A}+i\lambda)^{-1}$ on $\Hcd$ uniformly in large $\abs{\lambda}$ for $\lambda\in\mathbb{R}$. The stability estimate will give us energy decay results $e^{t\Ac}$ from $\mathcal{D}\rightarrow \Hcd$: see further in \S \ref{s3-3}. 

\subsection{Main results}
Our results are trifold. The first result is on the optimal backward uniqueness of the damped wave semigroup $e^{t\Ac}$. The authors proved in \cite{kw22} that when $Q\in \mathcal{L}(H_{\gamma}, Y)$ for $\gamma<\frac{1}{4}$, $e^{t\Ac}$ is backward unique. Here we make an improvement to show the backward uniqueness still holds for the limit case $\gamma=\frac{1}{4}$. 
\begin{theorem}[Optimal backward uniqueness]\label{thmbackward}
Let $Q\in \mathcal{L}(H_\frac{1}{4}, Y)$. Then $e^{t\Ac}$ is backward unique, that is, if $e^{t\Ac}u=0$ for some $t>0$, then $u=0$. 
\end{theorem}
The limit case $\gamma=\frac{1}{4}$ is optimal: in \cite{cc01}, an example of $Q$ such that $Q\in \mathcal{L}(H_\gamma, Y)$ for all $\gamma>\frac{1}{4}$ was constructed. For such an example, $e^{t\Ac}\equiv 0$ for any $t\ge 2$: finite time extinction happens and the backward uniqueness fails. 

Our second result is to connect control estimates to resolvent estimates that imply decay results, even when $Q^*Q$ fails to be bounded on $H$. Such connection allows us to separate the analysis of damped resolvent into that of the propagator $P$, and that of the control operator $Q$. 
\begin{definition}[Control estimates]
Fix $\gamma\in [0,\frac{1}{2}]$ and $\mu\ge 0$. Let $Q\in \mathcal{L}(H_{\gamma}, Y)$. We say the control estimate holds for $P$ and $Q$, if there are $\lambda_0\ge 0$, $N>0$ such that
\begin{equation}\label{1l1}
\|u\|_{H}\le M(\lambda)\langle \lambda\rangle^{-1}\|(P-\lambda^2)\Lambda^{\mu-\gamma} u\|_{H}+m(\lambda)\|Q \Lambda^{\mu} u\|_{Y}+C\|\Lambda^{-N} u\|_{H}
\end{equation}
uniformly for any $u\in H_{1+\mu-\gamma}$ and all $\lambda\ge \lambda_0$. Here $M(\lambda), m(\lambda)$ are positive continuous functions on $[0, \infty)$, where there is $C>0$ such that $M(\lambda)\ge C$, $m(\lambda)\ge C$ for all $\lambda\ge 0$. 
\end{definition}
We have a flexible parameter $\mu$. We mostly choose $\mu=\gamma$, or sometimes $\mu=0$. 

\begin{theorem}[Control and stability]\label{thmcontrol}The following are true:
\begin{enumerate}[wide]
\item Let the control estimate \eqref{1l1} hold for $P$ and $Q$. Then there exists $\lambda_0'> 0,$ such that
\begin{equation}
\|(\mathcal{A}+i\lambda)^{-1}\|_{\mathcal{L}(\mathcal{H})}\le C\abs{\lambda}^{4\mu}(M(\abs{\lambda})^2+m(\abs{\lambda})^2),
\end{equation}
for all $\lambda\in \mathbb{R}$ with $\abs{\lambda}\ge\lambda_0'$.

\item Suppose there is $\lambda_0\ge0$ such that $\|(\mathcal{A}+i\lambda)^{-1}\|\le K(\abs{\lambda})$ for all $\lambda\in\mathbb{R}$ with $\abs{\lambda}\ge \lambda_0$, where $K(\abs{\lambda})\ge C'>0$. Then there exists $C>0$ such that the control estimate holds for $P$ and $Q$, that is for any $\mu \in [0, \frac{1}{2}+\gamma]$, there exists $C>0,$ such that
\begin{equation}
\|u\|_{H}\le C\langle \lambda\rangle^{2(\gamma-\mu)}K(\abs{\lambda}) \left(\langle \lambda\rangle^{-1}\|(P-\lambda^2)\Lambda^{\mu-\gamma}u\|_{H}+ \|Q\Lambda^{\mu}u\|_Y\right),
\end{equation}
for all $\lambda\ge \lambda_0$ and $u\in H_{\frac{1}{2}+\mu}$.
\end{enumerate}
\end{theorem}
Note that estimates \eqref{1l1} are not the non-uniform Hautus test studied in \cite{mil12,cpsst19}. Indeed, Theorem \ref{thmcontrol} outperforms results based on the non-uniform Hautus test in many cases where $Q^*Q$ is unbounded on $H$: see \S \ref{s2-1}. Moreover, it implies that we can sometimes choose a small $\mu$ to improve the decay rates. Theorem \ref{thmcontrol} also accommodates compact error terms $\|\Lambda^{-N}u\|_H$ in \eqref{1l1}: this allows us to use microlocal analysis, as in \S \ref{s2}.  

Our third result turns control estimates for $P$ into control estimates for $P^\alpha$. Here we define
\begin{equation}
P^\alpha=\int_0^\infty \rho^{2\alpha}\ dE_\rho, \ \text{and}\ \Lambda_{(\alpha)}^s=\int_0^\infty (1+\rho^{2\alpha})^{s}\ dE_\rho
\end{equation}
are scaling operators for $H_s^{(\alpha)}=\Lambda^{-s}_{(\alpha)}(H)$ with respect to $P^\alpha$. Let $e(\lambda)$ be a nonnegative function. 
\begin{theorem}[Dilation and contraction]\label{thmdilate}
Let $\gamma\in[0, \frac{1}{2}]$, $\mu\ge 0$. Assume the control estimate \eqref{1l1} holds for $P$ and $Q$, that is there exists $N>0$ and for $\lambda>\lambda_0$, 
\begin{equation}
\|u\|_{H}\le M(\lambda)\lambda^{-1}\|(P-\lambda^2)\Lambda^{\mu-\gamma} u\|_{H}+m(\lambda)\|Q\Lambda^{\mu} u\|_{Y}+e(\lambda)\|\Lambda^{-N} u\|_{H},
\end{equation}
uniformly for all $u\in H_{1+\mu-\gamma}$. Fix positive $\fpow\ge 2\gamma$. Then there exists $C, \lambda_0'$ such that the control estimate holds for $P^\alpha$ and $Q\in \mathcal{L}(H^{(\alpha)}_{\gamma/\alpha}, Y)$. That is for $\lambda > \lambda_0'$, 
\begin{multline}\label{1l2}
\|u\|\le C\left(M(\lambda^{\frac{1}{\fpow}})\lambda^{\frac{1}{\alpha}-1}+m(\lambda^{\frac{1}{\fpow}})\lambda^{1+(-2+2(2\gamma+1-\fpow)_+)/\alpha}\right)\langle \lambda\rangle^{-1}\|(P^{\fpow}-\lambda^2)\Lambda_{\alpha}^{(\mu-\gamma)/\fpow} u\|_H\\
+m(\lambda^{\frac{1}{\fpow}})\|Q\Lambda_\fpow^{\mu/\fpow}u\|_Y+Ce(\lambda^{\frac{1}{\alpha}})\|\Lambda_\alpha^{-N/\fpow}u\|_H,
\end{multline}
uniformly for all $u\in H_{1+(\mu-\gamma)/\alpha}^{(\alpha)}$, where $(2\gamma+1-\fpow)_+=\max\{2\gamma+1-\fpow, 0\}$.
\end{theorem}
Theorem \ref{thmdilate} is designed synergistically with Theorem \ref{thmcontrol}: one can immediately apply Theorem \ref{thmcontrol} to \eqref{1l2}, as the control estimates \eqref{1l1} for $P^\alpha$ and $Q\in \mathcal{L}(H^{(\alpha)}_{\gamma/\alpha}, Y)$. Note that the case of $\alpha>1$ ($\lambda^{\alpha}$ is a convex function on $\mathbb{R}_{>0}$) was known in \cite[Theorem 3.5]{mil12} for $P$ positive. Here Theorem \ref{thmdilate} works for all $\alpha>0$ ($\lambda^{\alpha}$ can be either a convex or concave function on $\mathbb{R}_{>0}$) and $P$ nonnegative ($\ker P$ does not have to be trivial). We remark here that cutting out the $\ker P$ out of $\mathcal{A}$ when $Q^*Q$ is not relatively compact (when $\gamma=\frac{1}{2}$) is not trivial: see Remark \ref{3t10}.

The paper is organised in the following manner: immediately in Section \ref{s2} we present corollaries and new and novel decay results obtained as straightforward applications of Theorems \ref{thmcontrol} and \ref{thmdilate}. In Section \ref{s3}, we establish facts about semigroup $e^{t\Ac}$, prove Theorems \ref{thmbackward}, \ref{thmcontrol}, \ref{thmdilate} and auxiliary results used in Section \ref{s2}. 

\subsection{Acknowledgement}
The authors thank Jeffrey Galkowski, Jeremy Marzuola, David Seifert, Jian Wang and Jared Wunsch for discussions around the results. The authors are grateful to Jeffrey Galkowski for pointing out an improved isolation of pole at 0 when the damping is not relatively compact. RPTW is partially supported by NSF grant DMS-2054424 and EPSRC grant EP/V001760/1.

\section{Applications}\label{s2}
In this section, we discuss consequences of Theorems \ref{thmcontrol} and \ref{thmdilate}. In \S \ref{s2-1}, we show the Schrödinger observability gives explicit $\gamma$-dependent energy decay rates for unbounded damping. In \S\ref{s2-2}, we show a larger damping will always decay at least as fast as a known slower rate. In \S\ref{s2-3}, we prove a global Carleman estimate, and an optimal propagation estimate for fractional Laplacians $\Delta^\alpha$ without potential ($\alpha>0$). In \S\ref{s2-4}, we obtain energy decay results for wave equations with $L^p$-damping on manifolds. In \S\ref{s2-5}, we obtain energy decay results for linearised gravity water wave equations with damping on manifolds, where both the propagator and the damping may be non-local. In \S\ref{s2-6}, we obtain energy decay results for Euler--Bernoulli plate equations with both viscous and structural damping, the later of which is differential and unbounded. 

\subsection{Schrödinger observability gives decay}\label{s2-1}
The abstract Schr\"odinger equation for $P$, 
\begin{equation}
	(i\partial_t+P)u=0, \ u|_{t=0}=u_0\in H,
\end{equation}
is uniquely solved by $u=e^{itP}u_0$, where $e^{itP}$ is the unitary Schr\"odinger operator group on $H$ generated by the anti-self-adjoint operator $iP: H_1\rightarrow H$. We say that the Schr\"odinger equation is exactly observable via $Q\in \mathcal{L}(H_{\gamma}, Y)$ if there exists $T>0, C_T>0$ such that for any $u_0 \in H_1$ we have 
\begin{equation}
\|u_0\|_{H}\le C_T\int_{0}^{T} \|Qe^{itP} u_0\|_{Y}\ dt.
\end{equation}
In \cite{mil05,mil12} (see also \cite{mil05,jl20,ll21}), an equivalent resolvent condition was given. The Schr\"odinger equation is exactly observable via $Q\in \mathcal{L}(H_\gamma, Y)$ if and only if there is $C>0$ such that
\begin{equation}\label{2l19}
\|u\|_{H}\le C\|(P-\lambda^2)u\|_{H}+C\|Qu\|_{Y},
\end{equation}
for every $\lambda\ge 0$ and $u\in H_1$. When $\Omega$ is an open set and the Schrödinger equation is exactly observable via $\mathbbm{1}_{\Omega}$, we say the Schr\"odinger equation is exactly observable from $\Omega$. The Schrödinger observability has been investigated in many systems: see further in \cite{jaf90,mil05,bz12,am14,alm16,dj18,djn22}. 
\begin{example}
The Schrödinger equation is observable from any open set, on $\mathbb{T}^d$ or compact hyperbolic surfaces. 
\end{example}
Notably, in \cite{aln14}, it was proved that the Schrödinger observability gives an energy decay rate $\|e^{t\Ac}\|_{\mathcal{D}\rightarrow \Hcd}\le C\langle t\rangle^{-1/2}$ in abstract systems when $\gamma=0$, that is, when $Q$ is bounded on $H$. 

We are interested in investigating the case when $Q$ is no longer bounded on $H$. We first discuss a recent development in semigroup theory \cite{cpsst19}, where it was shown that the non-uniform Hautus test gives rise to stability estimates without a priori dependence on the degree of unboundedness. Following our remarks in \cite[Remark 2.29]{kw22}, this can be improved to give stability estimates that explicitly depend on the degree of unboundedness $\gamma$:

\begin{theorem}[Non-uniform Hautus test, \cite{cpsst19}]\label{t4}
Let $Q\in \mathcal{L}(H_{\gamma}, Y)$ for $\gamma\in [0,\frac{1}{2}]$. Assume there exists $C>0$ such that
\begin{equation}\label{2l1}
\|u\|_{H}\le M(\lambda)\langle \lambda\rangle^{-1}\|(P-\lambda^2)u\|_{H}+m(\lambda)\|Q u\|_{Y}
\end{equation}
uniformly for any $u\in H_1$ and all $\lambda\ge\lambda_0>0$. Then
\begin{equation}
\|(\mathcal{A}+i\lambda)^{-1}\|_{\mathcal{L}(\mathcal{H})}\le C\abs{\lambda}^{8\gamma}M(\abs{\lambda})^2m(\abs{\lambda})^2,
\end{equation}
for all $\lambda\in \mathbb{R}$ with $\abs{\lambda}\ge\lambda_0'$ for some $\lambda_0'>0$.
\end{theorem}
\begin{remark}
We claim our main theorem, Theorem \ref{thmcontrol} of control, is in general better than Theorem \ref{t4}: (1) the control estimate in \eqref{1l1} accommodates compact errors in Theorem \ref{thmcontrol}, and enables the use of classical microlocal analysis. (2) The $4\mu$-loss in the asymptotics in Theorem \ref{thmcontrol} is in general only half of the $8\gamma$-loss in this theorem, as illustrated by the next two propositions. (3) Choosing a small flexible parameter $\mu$ in Theorem \ref{thmcontrol} allows further improvement over the $\gamma$-loss, as in \S\ref{s2-6}. 
\end{remark}
To keep the manuscript self-contained, we give a complementary proof of Theorem \ref{t4} in \S \ref{s3-7} using semiclassical methods. As a direct corollary, the energy decay induced by the Schrödinger observability in \cite{aln14} can be generalised to the case $\gamma>0$:

\begin{proposition}[Observatory stability via $Q$, \cite{cpsst19}]\label{t1}
Assume the Schr\"odinger equation is exactly observable via $Q\in \mathcal{L}(H_{\gamma}, Y)$ for $\gamma\in[0, \frac{1}{2}]$. Then there exists $C>0$ such that
\begin{equation}
\|e^{t\Ac}\|_{\mathcal{D}\rightarrow \Hcd}\le C\langle t\rangle^{-\frac{1}{2+8\gamma}}.
\end{equation}
\end{proposition}

In \cite[Remark 2.29]{kw22}, the authors obtained better decay estimates under the same geometric assumptions on the support of $W$, without additional regularity assumptions. In this manuscript, we further investigate such improvement, and find it comes from another observability condition. Indeed, observing the system using $Q\Lambda^\gamma$ gives an improvement:
\begin{proposition}[Observatory stability via $Q\Lambda^\gamma$]\label{t2}
Assume the Schr\"odinger equation is exactly observable via $Q\Lambda^{\gamma}\in \mathcal{L}(H_{2\gamma}, Y)$ for $\gamma\in[0, \frac{1}{2}]$ and $N> 0$. Then there exists $C>0$ such that
\begin{equation}
\|e^{t\Ac}\|_{\mathcal{D}\rightarrow \Hcd}\le C\langle t\rangle^{-\frac{1}{2+4\gamma}}.
\end{equation}
\end{proposition}
Note the $4\gamma$-improvement in the decay order. We remark here that the observability via $Q\Lambda^\gamma$ is not artificial. Many systems are observable in such way and decay faster: see for example, the rest of \S \ref{s2}. In the bounded case of $\gamma=0$, both observability results reduce back to \cite{aln14} and give the same $\langle t\rangle^{-1/2}$-decay.
\begin{proof}[Proof of Propositions \ref{t1} and \ref{t2}]
We prove Proposition \ref{t2}, and the proof of Proposition \ref{t1} is similar. As in \eqref{2l19}, the Schrödinger observability via $Q\Lambda^{\gamma}$ is equivalent to
\begin{equation}
\|u\|_{H}\le C\|(P-\lambda^2)u\|_{H}+C\|Q \Lambda^{\gamma} u\|_{Y},
\end{equation}
for any $u\in H_{1}$. Apply Theorem \ref{thmcontrol}(1) with $M=C\langle \lambda\rangle$, $m=C$, $\mu=\gamma$ to obtain
\begin{equation}
\|(\mathcal{A}+i\lambda)^{-1}\|_{\mathcal{L}(\mathcal{H})}\le C\abs{\lambda}^{2+4\gamma},
\end{equation}
uniformly for large $\abs{\lambda}$, $\lambda\in\mathbb{R}$. Lemma \ref{uniqelemma}(1) implies $\operatorname{UCP}_{P, Q\Lambda^{\gamma}}$ holds. See the definition of the UCP in Definition \ref{3t6}. Lemma \ref{uniqelemma}(4) implies $\operatorname{UCP}_{P, \Lambda^{-\gamma}}$ holds. Lemma \ref{uniqelemma}(5) then implies $\operatorname{UCP}_{P, Q}$ holds. Apply Lemma \ref{3t7} to conclude the energy decay. 
\end{proof}

\subsection{Weak Monotonicity}\label{s2-2}
One peculiar feature of the damped wave equation is the overdamping phenomenon: strictly larger damping may not lead to faster decay. The phenomenon, still mysterious, has been intensely investigated in \cite{aln14,sta17,kle19b,dk20,sun22,kw22}. In other words, the decay rate is not monotonic with respect to the size of damping. However we claim the decay rate is still weakly monotonic: larger damping will decay at least as fast as a known slower rate. This is characterised in the next proposition:

\begin{proposition}\label{2l15}
Let $\gamma\in[0,\frac{1}{2}]$. Consider $Q_1 \in L(H_{\gamma}, Y_1)$, $Q_2\in L(H_{\gamma}, Y_2)$ and assume there is $C_0>0$ such that $\|Q_1 u\|_{Y_1}  \leq C_0\|Q_2 u\|_{Y_2}$ for all $u \in H_1/\ker(P)$. Then the following are true:
\begin{enumerate}
	\item If $\gamma=0$ and $\|e^{t\Ac_{Q_1}}\|_{\mathcal{D}\rightarrow \Hcd}\le e^{Ct}$, then $\|e^{t\Ac_{Q_2}}\|_{\mathcal{D}\rightarrow \Hcd}\le e^{C't}$.
	\item If $\gamma>0$ and $\|e^{t\Ac_{Q_1}}\|_{\mathcal{D}\rightarrow \Hcd}\le e^{Ct}$, then $\|e^{t\Ac_{Q_2}}\|_{\mathcal{D}\rightarrow \Hcd}\le C'\langle t\rangle^{-\frac{1}{4\gamma}}$.
	\item If $\|e^{t\Ac_{Q_1}}\|_{\mathcal{D}\rightarrow \Hcd}\le C\langle t\rangle^{-\alpha}$ for $\alpha> 0$, then $\|e^{t\Ac_{Q_2}}\|_{\mathcal{D}\rightarrow \Hcd}\le C'\langle t\rangle^{-\frac{\alpha}{2+4\gamma\alpha}}$. 
	\item If $\|e^{t\Ac_{Q_1}}\|_{\mathcal{D}\rightarrow \Hcd}\le C/\log(2+t)$, then $\|e^{t\Ac_{Q_2}}\|_{\mathcal{D}\rightarrow \Hcd}\le C'/\log(2+t)$.
\end{enumerate}
\end{proposition}

\begin{corollary}\label{2t15}
Let $P=\Delta$. Consider nonnegative $W_1\in C^\infty(\M)$, $W_2\in L^\infty(\M)$, with $W_1\le W_2$. If $\|e^{t\Ac(\sqrt{W_1})}\|_{\mathcal{D}\rightarrow \Hcd}\le C\langle t\rangle^{-\alpha}$ for $\alpha>0$, then $\|e^{t\Ac(\sqrt{W_2})}\|_{\mathcal{D}\rightarrow \Hcd}\le C'\langle t\rangle^{-\frac{\alpha}{2}}$.
\end{corollary}
The immediate corollary allows us to obtain decay information for any $L^{\infty}$ damping by studying a smaller smooth damping. Part (1) of the proposition is a well known tool used to prove that the geometric control condition implies exponential energy decay for $C^0$ damping from the same result for $C^{\infty}$ damping. 
\begin{example}
On $\M=\mathbb{T}^2$, given any nonempty open set $\Omega$, the authors of \cite[Theorem 2.6]{aln14} showed for every $\epsilon>0$ there exists a smooth bounded nonnegative function $W_\epsilon$ supported in $\Omega$ such that $\|e^{t\Ac(\sqrt{W_\epsilon})}\|_{\mathcal{D}\rightarrow \Hcd}\le C\langle t\rangle^{-1+\epsilon}$. Note $W_\epsilon\le C \mathbbm{1}_\Omega$ and Corollary \ref{2t15} implies $\|e^{t\Ac(\sqrt{W_\epsilon})}\|_{\mathcal{D}\rightarrow \Hcd}\le C_\epsilon\langle t\rangle^{-1/2+\epsilon/2}$ for every $\epsilon>0$. This is consistent with the $\langle t\rangle^{-1/2}$-decay governed by the Schrödinger observability in \cite[Theorem 2.3]{aln14}. 
\end{example}

\begin{proof}[Proof of Proposition \ref{2l15}]
We will prove (3), and the other three are similar. Apply Lemma \ref{3t7} to see $\operatorname{UCP}_{P,Q_1}$ holds and $\|(\Ac_{Q_1}+i\lambda)^{-1}\|_{\mathcal{L}(\Hc)}\le C\abs{\lambda}^{1/\alpha}$ for large $\abs{\lambda}$, $\lambda\in\mathbb{R}$. See the definition of $\operatorname{UCP}$ in Definition \ref{3t6}. Apply Theorem \ref{thmcontrol}(2), with $\mu=0$, to see for every $u\in H_1$, 
\begin{equation}
\|u\|_{H}\le C\langle \lambda\rangle^{2\gamma}\abs{\lambda}^{\frac{1}{\alpha}} \left(\langle \lambda\rangle^{-1}\|(P-\lambda^2)\Lambda^{-\gamma}u\|_{H}+ \|Q_1 u\|_{Y_1}\right),
\end{equation}
for large $\lambda$. By the assumption, the last term can be replaced by $\|Q_2 u\|_{Y_2}$. Apply Theorem \ref{thmcontrol}(1) to see for large $\abs{\lambda}$, $\lambda\in\mathbb{R}$, $\|(\Ac_{Q_2}+i\lambda)^{-1}\|_{\mathcal{L}(\Hc)}\le C\abs{\lambda}^{4\gamma+2/\alpha}$. By Lemma \ref{uniqelemma}(2), we know $\operatorname{UCP}_{P,Q_2}$ holds. Apply Lemma \ref{3t7}, and $\|e^{t\Ac_{Q_2}}\|_{\mathcal{D}\rightarrow \Hcd}\le C'\langle t\rangle^{-\frac{\alpha}{2+4\gamma\alpha}}$ as desired.  
\end{proof}

\subsection{Quantitative unique continuation and propagation for fractional Laplacian}\label{s2-3}
Eigenfunctions of Laplacian can be uniquely continued from any open set $\Omega$ in compact manifolds $\M$. Indeed, for $\lambda\in\mathbb{R}$, 
\begin{equation}
(\Delta-\lambda^2)u=0, \ u\neq 0\ \Rightarrow \ \mathbbm{1}_\Omega u\neq0.
\end{equation}
This is called the unique continuation principle, first observed by Carleman \cite{car39} in 1939. Its quantitative form, called the Carleman estimate, is the following: uniformly for $\lambda\ge \lambda_0>0$,
\begin{equation}
\|u\|_{H^{2}}\le e^{C\lambda}(\|(\Delta-\lambda^2)u\|_{L^2}+\|\mathbbm{1}_\Omega u\|_{L^2}).
\end{equation}
Carleman estimates quantitatively describe the tunnelling phenomenons in quantum physics. Many results \cite{hor4,leb93,lr95,lr97,bur98,bn02,mil05,bdl07,lr10,bj16,wan20} on the quantitative unique continuation for Laplacian are known, yet little is known about that for fractional Laplacians: see \cite{mil06,mil12,rul17} and the references there within. Here we present a global Carleman estimate for fractional Laplacians without potentials, as a simple corollary of Theorem \ref{thmdilate}.
\begin{proposition}[Global Carleman estimates]\label{2t17}
Let $\M$ be compact manifold without boundary, and $\Omega$ be a non-empty open set. Let $\alpha>0$. Then there exist $C, \lambda_0>0$ such that
\begin{equation}
\|u\|_{H^{\alpha+(\alpha-1)_{+}}}\le e^{C\lambda^{1/\alpha}}(\|(\Delta^\alpha-\lambda^2)u\|_{H^{-\alpha+(\alpha-1)_+}}+\|\mathbbm{1}_\Omega u\|_{L^2})
\end{equation}
for every $\lambda\ge \lambda_0$ and $u\in H^{\alpha+(\alpha-1)_{+}}(\M)$. 
\end{proposition}
\begin{remark}\label{2t20}
This result is non-trivial, even assuming known results in semiclassical analysis. Indeed, $(h^{2}\Delta)^\alpha$ is not a semiclassical pseudodifferential operator when $\alpha \notin\mathbb{N}$, and standard semiclassical tools, such as elliptic parametrices and the Gårding inequality, fail to yield the desired Carleman estimates without careful separation of the singularity at $0$ of the semiclassical symbol $\abs{\xi}^{2\alpha}$. When $\alpha>1$, a possible alternative direction of proof may follow from the wavepacket conditions in \cite{mil12}. 
\end{remark}
\begin{proof}
1. Consider the standard Carleman estimate from \cite[Proposition 3.2]{bur20} (see also \cite{leb93,bur98,bj16,wan20}): for each $u\in H^1$, 
\begin{equation}\label{2l22}
\|u\|_{L^2}\le e^{C\lambda}(\|(\Delta-\lambda^2)u\|_{H^{-1}}+\|\mathbbm{1}_\Omega u\|_{L^2}). 
\end{equation}
When $\alpha\ge 1$, apply Theorem \ref{thmdilate} of dilation with $\mu=0$, $\gamma=\frac{1}{2}$, $e(\lambda)=0$, $M=m=e^{C\lambda}$ to see $\alpha\ge 2\gamma$ and
\begin{equation}
\|u\|_{L^2}\le e^{C\lambda^{1/\alpha}}(\|(\Delta^\alpha-\lambda^2)u\|_{H^{-1}}+\|\mathbbm{1}_\Omega u\|_{L^2}).
\end{equation}
Note we used that polynomial growth is bounded by exponential growth. Now combine this with
\begin{equation}
\|u\|_{H^{2\alpha-1}}\le C\|\Delta^{\alpha}u\|_{H^{-1}}+C\|u\|_{L^2}\le C\|(\Delta^{\alpha}-\lambda^2)u\|_{H^{-1}}+(1+C\lambda^2)\|u\|_{L^2},
\end{equation}
to observe
\begin{equation}
\|u\|_{H^{2\alpha-1}}\le e^{C\lambda^{1/\alpha}}(\|(\Delta^\alpha-\lambda^2)u\|_{H^{-1}}+\|\mathbbm{1}_\Omega u\|_{L^2}).
\end{equation}

2. When $\alpha\in (0,1)$, note \eqref{2l22} implies
\begin{equation}
\|u\|_{L^2}\le e^{C\lambda}(\|(\Delta-\lambda^2)u\|_{H^{-\alpha}}+\|\mathbbm{1}_\Omega u\|_{L^2}). 
\end{equation}
Now apply Theorem \ref{thmdilate} of dilation with $\mu=0$, $\gamma=\frac{\alpha}{2}$, $e(\lambda)=0$, $M=m=e^{C\lambda}$. Note the key fact $\alpha=2\gamma$ allows us the use of Theorem \ref{thmdilate}. We now have
\begin{equation}
\|u\|_{L^2}\le e^{C\lambda^{1/\alpha}}(\|(\Delta^\alpha-\lambda^2)u\|_{H^{-\alpha}}+\|\mathbbm{1}_\Omega u\|_{L^2}).
\end{equation}
Combining this with 
\begin{equation}
\|u\|_{H^{\alpha}}\le C\|\Delta^{\alpha}u\|_{H^{-\alpha}}+C\|u\|_{L^2}\le C\|(\Delta^{\alpha}-\lambda^2)u\|_{H^{-\alpha}}+(1+C\lambda^2)\|u\|_{L^2},
\end{equation}
gives the desired inequality.
\end{proof}

On another hand, the wave equation has solutions concentrating near and propagating along Hamiltonian flows (geodesic flows in $T^*\M$). To be able to observe such solutions, we must observe at least one point of every geodesic. This leads to the definition of the geometric control condition:
\begin{definition}[Geometric control condition]
We say an open set $\Omega\subset \M$ satisfies the geometric control condition (for $\Delta$), if there is $T>0$ such that every unit-speed geodesic of length $T>0$ has a nonempty intersection with $\Omega$. 
\end{definition}
In \cite{rt74,blr92,bg97}, it was shown that the exact observability of the wave equation is equivalent to the geometric control condition. The exact observability of the wave equation from $\Omega$ open is equivalent to the propagation estimate, a stationary control estimate of type
\begin{equation}
\|u\|_{L^2}\le C\lambda^{-1}\|(\Delta-\lambda^2)u\|_{L^2}+C\|\mathbbm{1}_\Omega u\|_{L^2}. 
\end{equation}
For observability from time dependent $\Omega$ and its consequences for energy decay see \cite{lrltt17, kle22b, kle23}. It is natural to ask if the propagation estimate still holds for systems other than waves, for example, linearised water waves. This is to say whether we can replace the Laplacian $\Delta$ by the fractional Laplacian $\Delta^\alpha$ for $\alpha>0$. When $\alpha\notin\mathbb{N}$, $\Delta^\alpha$ is not a local operator, and will move the information of $u$ in and out of the support of $u$: this is the complicated part of the problem, as discussed in Remark \ref{2t20}. Here we present an optimal propagation estimate for fractional Laplacian:
\begin{proposition}[Optimal geometric propagation]\label{2t16}
Let $\M$ be a smooth manifold without boundary. Let $\Omega$ satisfy the geometric control condition (for $\Delta$). Then for $\alpha\in (0,\infty)$, there exists $C>0$ such that uniformly for $u\in H^{2\alpha}(\M)$ and large real $\lambda$
\begin{equation}\label{2l16}
\|u\|_{L^2}\le C\lambda^{-2+\frac{1}{\alpha}}\|(\Delta^\alpha-\lambda^2)u\|_{L^2}+C\|\mathbbm{1}_\Omega u\|_{L^2}.
\end{equation}
This estimate is optimal in the asymptotics of both terms when $\Omega$ is not dense in $\M$. 
\end{proposition}
With this propagation estimate, we recover a Schrödinger observability result for fractional Laplacian without potential. Note a stronger version is proved in \cite[Theorem 1]{mac21} using semiclassical defect measures. 
\begin{corollary}[\cite{mac21}]
When $\alpha\in [\frac{1}{2}, \infty)$, the Schrödinger group $e^{it\Delta^\alpha}$ is exactly observable from $\Omega$, an open set that satisfies the geometric control condition (for $\Delta$). When $\alpha\in (0, \frac{1}{2})$, the Schrödinger group $e^{it\Delta^\alpha}$ is never exactly observable from any $\Omega$ that is not dense in $\M$.
\end{corollary}
\begin{proof}
When $\alpha\ge \frac{1}{2}$, we have $-2+\frac{1}{\alpha}\le 0$. Thus for large real $\lambda$, 
\begin{equation}
\|u\|_{L^2}\le C\|(\Delta^\alpha-\lambda^2)u\|+C\|\mathbbm{1}_\Omega u\|. 
\end{equation}
On the other hand, Lemma \ref{uniqelemma}(1) and (3) implies $\operatorname{UCP}_{\Delta^\alpha, \mathbbm{1}_{\Omega}}$ holds. Now apply \cite[Theorem 2.4]{mil12} to conclude the desired observability. When $0<\alpha<\frac{1}{2}$, the optimality of \eqref{2l16} implies the failure of observability.
\end{proof}

\begin{proof}[Proof of Proposition \ref{2t16}]
1. We began with the classical propagation of singularity result
\begin{equation}
\|u\|_{L^2}\le C\lambda^{-1}\|(\Delta-\lambda^2)u\|_{L^2}+C\|\mathbbm{1}_\Omega u\|_{L^2}
\end{equation}
uniformly for large real $\lambda$. See further in \cite{rt74,blr92} and \cite[Theorem E.47]{dz19}. When $\alpha\le 1$, we apply Theorem \ref{thmdilate} with $M=m=C$, $\gamma=\mu=0$ to obtain \eqref{2l16}. When $\alpha>1$, we apply \cite[Theorem 3.5]{mil12}.

2. We now show the estimate is optimal. Let $\M=\mathbb{S}^1=[0,2\pi]_x$ and $\Omega=(0, 1)$. Consider eigenfunctions $e_\lambda=\exp(i\lambda^{1/\alpha} x)$ to the Laplacian with $\Delta e_\lambda=\lambda^{2/\alpha}e_\lambda$, $\lambda^{1/\alpha}\in\mathbb{N}$. By the functional calculus, $\Delta^\alpha e_\lambda=\lambda^2 e_\lambda$. Moreover, note $\|u\|_{L^2}=\sqrt{2\pi}$ and $\|\mathbbm{1}_\Omega u\|_{L^2}\equiv 1$. This implies the $C\|\mathbbm{1}_\Omega u\|_{L^2}$ is optimal. On another hand, fix $\chi\in C^\infty$ such that $\chi \mathbbm{1}_\Omega\equiv 0$. Now estimate
\begin{equation}
\Delta^\alpha (\chi e_\lambda)=\chi\Delta^\alpha e_\lambda+[\Delta^\alpha, \chi]e_\lambda. 
\end{equation}
Note that $\Delta^\alpha$ is a classical pseudodifferential operator of order $2\alpha$ and $[\Delta^\alpha,\chi]\in \Psi^{2\alpha-1}$. This implies
\begin{equation}\label{2l17}
\|(\Delta^\alpha-\lambda^2)(\chi e_\lambda)\|_{L^2}\le C\|[\Delta^\alpha, \chi]e_\lambda\|_{L^2}\le C\|e_\lambda\|_{H^{2\alpha-1}}=C\lambda^{2-\frac{1}{\alpha}}\|e_\lambda\|_{L^2}\le C2\pi \lambda^{2-\frac{1}{\alpha}}\|\chi e_\lambda\|_{L^2}.
\end{equation}
Plugging $\chi e_{\lambda}$ into \eqref{2l16} gives
\begin{equation}
\|\chi e_\lambda\|_{L^2}\le C\lambda^{-(2-\frac{1}{\alpha})}\|(\Delta^\alpha-\lambda^2)(\chi e_\lambda)\|_{L^2}.
\end{equation}
Compare it with \eqref{2l17} to observe the optimality of $C\lambda^{-(2-\frac{1}{\alpha})}\|(\Delta^\alpha-\lambda^2)u\|_{L^2}$. This argument generalises the case $\alpha=\frac{1}{2}$ in the remark after \cite[Theorem 2]{amw23}. This argument works on any compact manifold without boundary, as long as $\Omega$ is not dense in $\M$, which implies there is a cutoff $\chi,$ such that $\chi\mathbbm{1}_\Omega\equiv 0$.
\end{proof}

\subsection{Singular damping on manifolds}\label{s2-4}
In this section, let $\M$ be a $d$-dimensional compact manifold. Let $P=\Delta\ge 0$ and $H=L^2(\M)$. Then $H_s=H^{2s}(\M)$, the $L^2$-Sobolev space with $2s$ weak derivatives. We will consider the damping of the form $Q^*Q=W$, where $Qu=\sqrt{W}u$ for some potential $W\in L^p(\M)$ for $p\in (d,\infty]$. Then the semigroup $e^{t\Ac}$ is naturally the solution operator where $(u, \partial_t u)=e^{t\Ac}(u_0,u_1)$ solves the initial value problem for the damped wave equation
\begin{equation}
(\partial_t^2+W(x)\partial_t+\Delta)u(x,t)=0, \ u(x,0)=u_0(x), \ \partial_t u(x, 0)=u_1(x),
\end{equation}
where the damping $W\in L^p$. When $W\in L^\infty$, there are many results known under different regularity, geometric and dynamical assumptions: see \cite{rt74,bg97,bur98,bh07,aln14,csvw14,bc15,kle19b,dk20,jin20,sun22, kle22}, and \cite{blr92,cv02,nis13,wan21,wan21b} for the boundary cases. However when $W\notin L^\infty$, very few results are known in \cite{cc01,fst18,fhs20,kw22}. Here we use Theorem \ref{thmcontrol} to give some new decay results when $W\notin L^\infty$. 
\begin{proposition}[$L^p$-damping]\label{lpdecay}
Let $\M$ be a $d$-dimensional compact manifold without boundary. Fix $p\in [\frac{d}{2},\infty]$ such that $p>1$. Let $W$ be $L^p(\M)$, then $Q\in\mathcal{L}(H^{\frac{d}{2p}}, L^2)$ is defined via $Qu=\sqrt{W}u$. Let $\Omega_{\e}$ be the interior of $\{W\ge\epsilon\}$. Then the following are true: if for some $\epsilon>0$, 
\begin{enumerate}[wide]
\item $\Omega_{\e}$ satisfies the geometric control condition, then 
\begin{equation}
\|e^{t\Ac}\|_{\mathcal{D}\rightarrow \Hcd}\le C\langle t\rangle^{-\frac{p}{d}}.
\end{equation}
\item the Schr\"odinger equation is exactly observable from $\Omega_{\e}$, then
\begin{equation}
\|e^{t\Ac}\|_{\mathcal{D}\rightarrow \Hcd}\leq C\<t\>^{-\frac{1}{2+\frac{d}{p}}}.
\end{equation}
\item $\Omega_{\e}$ is open and nonempty, then
 \begin{equation}
\|e^{t\Ac}\|_{\mathcal{D}\rightarrow \Hcd} \leq \frac{C}{\log(2+t)}.
\end{equation}
\end{enumerate}
When $d=1$ and $W$ is $L^1$, the above decay rates hold with $\frac{p}{d}$ replaced by $1-0$.
\end{proposition}

Before proving this result, we introduce normally-$L^p$ damping, a class of functions whose singularities structured near hypersurfaces, and state the analogous results that are comparable to those for $L^p$-damping. 

\begin{definition}[Normally $L^p$-functions]\label{2t14}
Assume the function $W(z)\ge 0$ on $\M$ and is $L^\infty$ on any compact subset of $\M\setminus N$. For $p\in [1,\infty]$, we say $W(z)$ is normally $L^p$ (with respect to $N$) if
\begin{equation}
w(\rho)=\operatorname{esssup}\{W(q,\rho): q\in N\}\in L^{p}(-\delta, \delta).
\end{equation}

\end{definition}
The name comes from the fact that $W$ blows up near $N=\rho^{-1}(0)$ like $w(\rho)$, a function $L^p$-integrable along $\rho$, the fiber variable of the normal bundle to $N$. For $1\le p<q\le \infty$, any normally $L^q$-damping is also normally $L^p$. The class of normally $L^\infty$-damping coincide with $L^\infty(\M)$. 
Note that by \cite[Lemma 2.3]{kw22} if $W$ is normally $L^p$ for $p \in (1,\infty)$ then the multiplier $\sqrt{W}$ is a bounded map from $H^{\frac{1}{2p}}(\M)$ to $L^2(\M)$, and if $W$ is normally $L^1$, then $\sqrt{W}$ is a bounded map from $H^{\frac{1}{2}+}(\M)$ to $L^2(\M)$. 

\begin{proposition}[Normally $L^p$-damping, \cite{kw22}]\label{normallpdecay}
Let $\M$ be a compact manifold without boundary. Fix $p\in (1,\infty)$. Let $W$ be normally $L^p$, then $Q\in\mathcal{L}(H^{\frac{1}{2p}}, L^2)$ is defined via $Qu=\sqrt{W}u$. Let $\Omega_{\e}$ be the interior of $\{W\ge\epsilon\}$. Then the following are true:
\begin{enumerate}[wide]
\item If for some $\epsilon>0$, $\Omega_{\e}$ geometrically controls $\M$, then 
\begin{equation}
\|e^{t\Ac}\|_{\mathcal{D}\rightarrow \Hcd}\le C\<t\>^{-p}.
\end{equation}
\item Assume the Schr\"odinger equation is exactly observable from $\Omega_{\e}$, then
\begin{equation}\label{2l10}
\|e^{t\Ac}\|_{\mathcal{D}\rightarrow \Hcd}\le C\langle t\rangle^{-\frac{1}{2+\frac{1}{p}}}.
\end{equation}
This is exactly \cite[Theorem 3]{kw22}.
\item If for some $\e>0$, $\Omega_{\e}$ is open and nonempty, then
 \begin{equation}
\|e^{t\Ac_Q}\|_{\Dc \ra \Hc} \leq  \frac{C}{\log(2+t)}.
\end{equation}
\end{enumerate}
When $W$ is normally $L^1$, then $\sqrt{W} \in \Lc(H^{1/2+0}, L^2)$ and the above decay rates hold with $p$ replaced by $1-0$. 
\end{proposition}
\begin{remark}
Note, the numerology of exponents here is exactly that of Proposition \ref{lpdecay} with $d=1$. This is because normally singular damping is only singular along the normal direction, which is one dimension.
\end{remark}
Here is a lemma establishing the multiplier properties of $L^{p}$-functions. 
\begin{lemma}\label{prop_lpmultiplier}
Let $\M$ be a $d$-dimensional compact manifold without boundary. Then for each $W\in L^{p}(\M)$:
\begin{enumerate}[wide]
\item When $p>1$, $\sqrt{W}$ maps $H^{\frac{d}{2p}}$ to $L^2$ as a multiplier. 
\item When $p=1$, $\sqrt{W}$ maps $H^{\frac{d}{2}+0}$ to $L^2$ as a multiplier. 
\end{enumerate}
\end{lemma}
\begin{proof}
When $p>1$, the Sobolev inequality gives $H^{\frac{d}{2p}}\hookrightarrow L^{\frac{2p}{p-1}}$, on which $\sqrt{W}\in L^{2p}$ is a multiplier to $L^2$. When $p=1$, $H^{\frac{d}{2p}+0}\hookrightarrow L^{\infty}$. 
\end{proof}

\begin{proof}[Proof of Propositions \ref{lpdecay} and \ref{normallpdecay}]
1. We begin with the assumption that $W\in L^p$. Let $\gamma=\frac{d}{4p}\le \frac{1}{2}$, and note Lemma \ref{prop_lpmultiplier} implies $Q\in\mathcal{L}(H_{\gamma}, L^2)$, given by $Qu=\sqrt{W}u$, where $H_{\gamma}=H^{2\gamma}(\M)=H^{\frac{d}{2p}}(\M)$. 

2. In this step, we aim to obtain the stability estimates for $\Ac$.

2a. When $\Omega_{\epsilon}$ geometrically controls $\M$, from the propagation estimates \cite[Theorem E.47]{dz19}, there exists $C, h_0>0$ such that uniformly for $h\in (0,h_0)$ we have
\begin{equation}
\|u\|_{L^2}\le Ch^{-1}\|(h^2\Delta-1)u\|_{H^{-1}_h}+C\|\mathbbm{1}_{\Omega_\epsilon}u\|_{L^2}+\bigo(h^{N})\|u\|_{H^{-N}_h},
\end{equation}
the last term of which can be absorbed by the left. Note $\|\mathbbm{1}_{\Omega_\epsilon} u\|_{L^2}\le \e^{-\frac{1}{2}} \|\sqrt{W} u\|_{L^2}$ and
\begin{equation}
\|(h^2\Delta-1)u\|_{H^{-1}_h}\le C\|(h^2\Delta-1)u\|_{H^{-2\gamma}_h}\le Ch^{-2\gamma}\|(h^2\Delta-1)u\|_{H^{-2\gamma}}.
\end{equation}
Let $h=\lambda^{-1}$. Then for $\lambda>h_0^{-1}$, we have the control estimate \eqref{1l1}:
\begin{equation}\label{lphautus}
\|u\|_{L^2}\le C\lambda^{-1+2\gamma}\|(\Delta-\lambda^2)u\|_{H^{-2\gamma}}+C\|\sqrt{W} u\|_{L^2}. 
\end{equation}
Then Theorem \ref{thmcontrol}(1), with $\mu=0$, $M=C\langle \lambda\rangle^{2\gamma}$, $m=C$ for some constant $C$, implies 
\begin{equation}\label{lpgccresolve}
\|(\mathcal{A}+i\lambda)^{-1}\|_{\mathcal{L}(\mathcal{H})}\le C\abs{\lambda}^{4\gamma}=C\abs{\lambda}^{\frac{d}{p}}.
\end{equation}

2b. When the Schr\"odinger equation is exactly observable from $\Omega_\epsilon$, there is $\Omega'$ compactly supported in $\Omega_\epsilon$ from which the Schr\"odinger equation is also exactly observable. Then \eqref{2l19} implies
\begin{equation}\label{2l28}
\|u\|_{L^2}\le C\|(\Delta-\lambda^2)u\|_{L^2}+C\|\mathbbm{1}_{\Omega'}u\|_{L^2},
\end{equation}
for all $\lambda\ge 0$. Now consider two cutoff functions $\chi,\chi_1\in C^\infty(\M)$ such that $\chi\equiv 1$ on $\Omega'$ and $\chi\chi_1\equiv\chi$ on $\M$, $\supp{\chi_1}\subset \Omega_\epsilon$. Since $\Lambda^{\gamma}=\langle D\rangle^{2\gamma}$ is classically elliptic on $\WF(\chi)$, apply the elliptic estimate \cite[Theorem E.33]{dz19} to obtain
\begin{equation}
\|\mathbbm{1}_\Omega' u\|_{L^2}\le \|\chi u\|_{L^2}\le C\|\chi_1 \langle D\rangle^{2\gamma} u\|_{L^2}+C\|u\|_{H^{-N}}\le C{\epsilon}^{-\frac{1}{2}}\|\sqrt{W} \Lambda^\gamma u\|_{L^2}+C\|u\|_{H^{-N}},
\end{equation}
for arbitrarily large $N$. 
Thus we have the control estimate \eqref{1l1}:
\begin{equation}\label{lpschro}
\|u\|_{L^2}\le C\abs{\lambda}\abs{\lambda}^{-1}\|(\Delta-\lambda^2)u\|_{L^2}+C\|\sqrt{W} \Lambda^\gamma u\|_{L^2}+C\|u\|_{H^{-N}}.
\end{equation}
Then Theorem \ref{thmcontrol}(1), with $\mu=\gamma$, $M=C\abs{\lambda}$, $m=C$ implies 
\begin{equation}\label{lpschroresolve}
\|(\mathcal{A}+i\lambda)^{-1}\|_{\mathcal{L}(\mathcal{H})}\le C\abs{\lambda}^{2+4\gamma}=\abs{\lambda}^{2+\frac{d}{p}}.
\end{equation}

2c. When $\Omega_\epsilon$ is open and nonempty, by a Carleman estimate (see \cite{leb93, bur98,bj16,wan20}), we have uniformly for all $v\in H^{1}_h$ and for all $h\in (0,h_0)$ that
\begin{equation}
\|v\|_{H^1_h}\le e^{C/h}\left(\|(h^2\Delta-1)v\|_{L^2}+\|\mathbbm{1}_{\Omega_\epsilon}v\|_{L^2}\right)\le e^{C/h}\left(\|(h^2\Delta-1)v\|_{L^2}+\epsilon^{-\frac{1}{2}}\|\sqrt{W}v\|_{L^2}\right).
\end{equation}
Let $v=\langle hD\rangle^{2\gamma}u$ and we have
\begin{multline}
\|u\|_{L^2}\le \|u\|_{H^{1-2\gamma}_h}\le e^{C/h}\left(\|(h^2\Delta-1)u\|_{H^{-2\gamma}_h}+\|\sqrt{W}\langle hD\rangle^{2\gamma}u\|_{L^2}\right)\\
\le e^{C/h}\left(\|(h^2\Delta-1)u\|_{L^2}+\|\sqrt{W}\langle D\rangle^{2\gamma}u\|_{L^2}\right).
\end{multline}
Now let $\lambda=h^{-1}$ and for $\lambda \geq 1/h_0$ we have
\begin{equation}
\|u\|_{L^2}\le e^{C\lambda}\left(\lambda^{-2}\|(\Delta-\lambda^2)u\|_{L^2}+\|\sqrt{W}\Lambda^{\gamma}u\|_{L^2}\right),
\end{equation}
where the powers of $\lambda$ can be absorbed into the exponential term, after potentially changing the constant $C$. This yields the control estimate \eqref{1l1}:
\begin{equation}
\|u\|_{L^2}\le e^{C\lambda}\abs{\lambda}^{-1}\|(\Delta-\lambda^2)u\|_{L^2}+e^{C\lambda}\|\sqrt{W}\Lambda^{\gamma}u\|_{L^2},
\end{equation}
uniformly for all $\lambda\ge \lambda_0$ for some $\lambda_0>0$. Then Theorem \ref{thmcontrol} with $\mu=\gamma$, $M=m=e^{C\lambda}$ implies
\begin{equation}\label{lplogresolve}
\|(\Ac+i\lambda)^{-1}\|_{\mathcal{L}(\mathcal{H})}\le Ce^{2C\lambda}\abs{\lambda}^{4\gamma} \le Ce^{C\lambda}.
\end{equation}

3. We have obtained desired stability estimates and now we verify the unique continuation assumptions. Let $(\Delta-\lambda^2)u=0$ for $\lambda>0$. By the classical unique continuation principle (see for example, \cite[Theorem 8.6.5]{hor1}), $u$ is not identically $0$ on $\Omega_\epsilon$ and thus $Wu\neq 0$. Thus $\operatorname{UCP_{\Delta, W}}$ holds. Apply Lemma \ref{3t7} to \eqref{lpgccresolve}, \eqref{lpschroresolve} and \eqref{lplogresolve} to conclude the desired decay results. 

4. When $d=1, p=1$, then let $\gamma=\frac{d}{4}+0$, so $Q \in \Lc(H_{\gamma},L^2)$ by Lemma \ref{prop_lpmultiplier}. Then the statement follows from an analogous proof as in 2. and 3. Note, when $d=2, p=1$, we have $Q \in \Lc(H_{\gamma}, L^2)$ only for $\gamma>1/2$, and Theorem \ref{thmcontrol} cannot be applied. 

5. The proof of Proposition \ref{normallpdecay} is similar. If $W$ is normally $L^p$ then $Q=\sqrt{W} \in \Lc(H^{2\gamma}, L^2)$ for $\gamma=\frac{1}{4p}$ when $p>1$ and $\gamma=\frac{1}{4}+0$ when $p=1$. See further in \cite[Lemma 2.3]{kw22}.
\end{proof}

\subsection{Damped gravity water waves}\label{s2-5}
Water waves describe the displacement of the free surface of water after disturbances: for example, the oscillation of the surface of an ocean under the wind. Specifically, using the paradifferential formulation of the water waves developed in \cite{abz11}, we obtain the leading order linear model for damped gravity water waves:
\begin{equation}\label{2l18}
(\partial_t^2+W(x)\partial_t +\abs{D})u(t, x)=0. 
\end{equation}
Here $\abs{D}=\Delta^{\frac{1}{2}}$ is a non-local operator, whose non-local nature breaks many known semiclassical technologies. Recently in \cite{amw23}, it was shown that on $\mathbb{S}^1$, if $W\in C^{\frac{1}{2}}$ and $W>0$ somewhere, then the energy decays by $\langle t\rangle^{-1}$, and this is optimal. This is the first energy decay result on damped water waves. In this section, we look at  damped water wave equations of the form
\begin{equation}
(\partial_t^2+\abs{D}^{s}W(x)\abs{D}^{s}\partial_t +\abs{D})u(t, x)=0,
\end{equation}
for $s\in[0,\frac{1}{2}]$ under general geometric assumptions. When $s=0$, it reduces to \eqref{2l18}. The case $s=\frac{1}{2}$ is motivated by the open problem formulated in \cite{ala18,amw23} from a similar paradifferential formulation that leads to nonlinear damping. The analysis of such systems are known to be difficult due to the non-local nature of $\abs{D}$ contributed to the failure of semiclassical techniques. We now prove a range of new decay results to demonstrate the power of our theorems. 
\begin{proposition}
Let $\M$ be a compact manifold without boundary, $P=\abs{D}$, $Q=\sqrt{W(x)}\abs{D}^s\in\mathcal{L}(H^s, L^2)$, where $W\in L^\infty$ is non-negative and $s\in[0,\frac{1}{2}]$. Let $\Omega_{\e}$ be the interior of $\{W\ge\epsilon\}$. Then the following are true: if for some $\epsilon>0$, 
\begin{enumerate}[wide]
\item $\Omega_\epsilon$ satisfies the geometric control condition (for $\Delta$), then
\begin{equation}
\|e^{t\Ac}\|_{\mathcal{D}\rightarrow \Hcd}\le C\langle t\rangle^{-\frac{1}{2+4s}}.
\end{equation}
\item the Schr\"odinger equation (for $\Delta$) is exactly observable from $\Omega_{\e}$, then
\begin{equation}
\|e^{t\Ac}\|_{\mathcal{D}\rightarrow \Hcd}\le C\langle t\rangle^{-\frac{1}{6+4s}}.
\end{equation}
\item $\Omega_{\e}$ is nonempty, then
	\begin{equation}
	\|e^{t\Ac}\|_{\mathcal{D}\rightarrow \Hcd}\le \frac{C}{\log^{\frac{1}{2}}(2+t)}.
	\end{equation}
\end{enumerate}
\end{proposition}
\begin{proof}
1. When $\Omega_\epsilon$ satisfies the geometric control condition, there is an open subset $\Omega'$ compact supported in $\Omega_\epsilon$ that also satisfies the geometric control condition. We have from \eqref{lphautus} that
\begin{equation}
\|u\|_{L^2}\le C\lambda^{-1+s}\|(\Delta-\lambda^2)u\|_{H^{-s}}+C\|\mathbbm{1}_{\Omega'}u\|_{L^2}+C\|u\|_{H^{-N}}. 
\end{equation}
Note $\abs{D}=\Delta^{\frac{1}{2}}$. Apply Theorem \ref{thmdilate} of dilation with $\alpha=\frac{1}{2}$, $\mu=0$, $\gamma=s$, $M(\lambda)=C\lambda^{s}, m(\lambda)=C$ to see 
\begin{equation}
\|u\|_{L^2}\le C(\lambda^{1+2s}+\lambda^{8s-1})\langle \lambda\rangle^{-1}\|(\abs{D}-\lambda^2) u\|_{H^{-s}}+C\|\mathbbm{1}_{\Omega'}u\|_L^2+C\|u\|_{H^{-N}}.
\end{equation}
Note that $1+2s>8s-1$, and this reduces to
\begin{equation}
\|u\|_{L^2}\le C\lambda^{1+2s}\langle \lambda\rangle^{-1}\|(\abs{D}-\lambda^2) u\|_{H^{-s}}+C\|\mathbbm{1}_{\Omega'}u\|_{L^2}+C\|u\|_{H^{-N}}.
\end{equation}
Now consider two cutoff functions $\chi,\chi_1\in C^\infty(\M)$ such that $\chi\equiv 1$ on $\Omega'$ and $\chi\chi_1\equiv\chi$ on $\M$, $\supp{\chi_1}\subset \Omega_\epsilon$. Now since $\chi_1\abs{D}^s$ is classically elliptic on $\WF\chi$, apply the elliptic estimate, \cite[Theorem E.33]{dz19}, to see
\begin{equation}
\|\mathbbm{1}_{\Omega'}u\|_{L^2}\le \|\chi u\|_{L^2}\le C\|\chi_1 \abs{D}^s u\|_{L^2}+C\|u\|_{H^{-N}}\le C\|\sqrt{W} \abs{D}^s u\|_{L^2}+C\|u\|_{H^{-N}}.
\end{equation}
This implies
\begin{equation}
\|u\|_{L^2}\le C\lambda^{1+2s}\langle \lambda\rangle^{-1}\|(\abs{D}-\lambda^2) u\|_{H^{-s}}+C\|\sqrt{W} \abs{D}^s u\|_{L^2}+C\|u\|_{H^{-N}}.
\end{equation}
Apply Theorem \ref{thmcontrol} to turn the control estimate with $M=C\lambda^{1+2s}$, $m=C$, $\mu=0$, $\gamma=s$ into 
\begin{equation}
\|(\Ac+i\lambda)^{-1}\|_{\mathcal{L}(\Hc)}\le C\abs{\lambda}^{2+4s},
\end{equation}
for large real $\lambda$. 

2. Assume there Schr\"odinger equation is exactly observable from $\Omega_\epsilon$. There is $\Omega'$ compactly supported in $\Omega_\epsilon$ from which the Schr\"odinger equation is also exactly observable. As in \eqref{2l19}, 
\begin{equation}
\|u\|_{L^2}\le C\langle \lambda\rangle\langle \lambda\rangle^{-1}\|(\Delta-\lambda^2)u\|_{L^2}+C\|\mathbbm{1}_{\Omega'} u\|_{L^2}. 
\end{equation}
Apply Theorem \ref{thmdilate} with $\alpha=\frac{1}{2}$, $\mu=\gamma=0$, $M=C \langle \lambda\rangle$, $m=C$ to see for large real $\lambda$,
\begin{equation}
\|u\|_{L^2}\le C(\lambda^{3}+\lambda^{-1})\langle \lambda\rangle^{-1}\|(\abs{D}-\lambda^2)u\|_{L^2}+C\|\mathbbm{1}_{\Omega'}u\|_{L^2},
\end{equation}
which simplifies to 
\begin{equation}
\|u\|_{L^2}\le C\lambda^3\langle \lambda\rangle^{-1}\|(\abs{D}-\lambda^2)u\|_{L^2}+C\|\mathbbm{1}_{\Omega'}u\|_{L^2}.
\end{equation}
Now consider two cutoff functions $\chi,\chi_1\in C^\infty(\M)$ such that $\chi\equiv 1$ on $\Omega'$ and $\chi\chi_1\equiv\chi$ on $\M$, $\supp{\chi_1}\subset \Omega_\epsilon$. Now $\chi_1\abs{D}^s\langle D\rangle^{s}$ is elliptic on $\WF\chi$. By the elliptic estimate, \cite[Theorem E.33]{dz19}, we have
\begin{multline}
\|\mathbbm{1}_\Omega' u\|_{L^2}\le \|\chi u\|_{L^2}\le C\|\chi_1\abs{D}^s\langle D\rangle^{s} u\|_{L^2}+C\|u\|_{H^{-N}}\\
\le C\|\sqrt{W} \abs{D}^s\Lambda^s u\|_{L^2}+C\|u\|_{H^{-N}}.
\end{multline}
This implies
\begin{equation}
\|u\|_{L^2}\le C\lambda^3\langle \lambda\rangle^{-1}\|(\abs{D}-\lambda^2)u\|_{L^2}+C\|\sqrt{W} \abs{D}^s\Lambda^s u\|_{L^2}+C\|u\|_{H^{-N}}.
\end{equation}
Apply Theorem \ref{thmcontrol} with $\mu=\gamma=s$, $M=C\lambda^3$, $m=C$, 
\begin{equation}
\|(\mathcal{A}+i\lambda)^{-1}\|_{\mathcal{L}(\mathcal{H})}\le C\abs{\lambda}^{6+4s},
\end{equation}
for real $\lambda$ with $\abs{\lambda}$ large. 

3. We now address the logarithmic decay. Note that Theorem \ref{thmcontrol} (or Theorem \ref{t4}) does not permit compact errors of size $e^{C\lambda^2}\|u\|_{H^{-N}}$, and this implies we cannot using any microlocal arguments that would leave us an unabsorbed error. We remind the reader that $E_\rho$ is the spectral measure with respect to $P=\abs{D}$. First we see for $\lambda\ge \lambda_0>0$ and $u\in H^2$, we have uniformly
\begin{equation}\label{2l23}
\|u\|^2=\int_0^\infty \ d\langle E_\rho u, u\rangle\le C\int_0^\infty ((\rho^2-\lambda^2)^2+\rho^{4s}) \ d\langle E_\rho u, u\rangle\le C\|(\abs{D}-\lambda^2) u\|^2+C\|\abs{D}^{s}u\|^2,
\end{equation}
where the constants do not depend on $\lambda$. Note Proposition \ref{2t17} with $\alpha=\frac{1}{2}$ implies for any $v\in H^{1/2}$, 
\begin{equation}
\| v\|_{H^{\frac{1}{2}}}\le e^{C\lambda^{2}}(\|(\abs{D}-\lambda^2) v\|_{H^{-\frac{1}{2}}}+\|\mathbbm{1}_{\Omega_\epsilon}  v\|_{L^2}).
\end{equation}
Specifically take $v=\abs{D}^s u$ and we have
\begin{equation}
\|\abs{D}^{s} u\|_{H^{\frac{1}{2}}}\le e^{C\lambda^{2}}(\|(\abs{D}-\lambda^2) u\|_{H^{s-\frac{1}{2}}}+\|\mathbbm{1}_{\Omega_\epsilon} \abs{D}^{s} u\|_{L^2}).
\end{equation}
Note $s-\frac{1}{2}\le 0$ and revisit \eqref{2l23} to see
\begin{equation}
\|u\|_{L^2}^2\le e^{C\lambda^{2}}(\|(\Delta-\lambda^2) u\|_{L^2}^2+\|\mathbbm{1}_{\Omega_\epsilon}\abs{D}^{s}u\|_{L^2}^2)
\end{equation}
uniformly for $\lambda\ge \lambda_0$ and $u\in H^2$. Apply Theorem \ref{t4} with $M=m=e^{C\lambda^2}$ to see 
\begin{equation}
\|(\mathcal{A}+i\lambda)^{-1}\|_{\mathcal{L}(\mathcal{H})}\le e^{C\lambda^2}. 
\end{equation}
We remark here that if we used Theorem \ref{thmcontrol} here instead of Theorem \ref{t4}, we would would only get the desired decay up to $s\in [0,\frac{1}{4}]$, compared to the full range $s\in [0,\frac{1}{2}]$ we obtained. 

4. We now prove the unique continuation holds and show the energy decay. Note that $\operatorname{UCP}_{\Delta, \mathbbm{1}_{\Omega_\epsilon}}$ holds for any ${\Omega_\epsilon}$ open on compact manifolds. By Lemma \ref{uniqelemma}(3), this implies $\operatorname{UCP}_{\abs{D}, \mathbbm{1}_{\Omega_\epsilon}}$, and thus $\operatorname{UCP}_{\abs{D}, \sqrt{W}}$ holds since $\mathbbm{1}_{\Omega_\epsilon}\le W$. And naturally $\operatorname{UCP}_{\abs{D}, \abs{D}^s}$ holds, and $[\abs{D},\abs{D}^s]=0$ implies $\operatorname{UCP}_{\abs{D}, \sqrt{W}\abs{D}^s}$, by Lemma \ref{uniqelemma}(5). Apply Lemma \ref{3t7} to conclude the energy decay.  
\end{proof}

\subsection{Damped plates}\label{s2-6}
Let $\M$ be a compact manifold without boundary. 
In the Euler--Bernoulli beam theory, vibrating beams/plates with damping are modelled by the damped beam equation
\begin{equation}
(\partial_t^2+\Delta^2+W_v(x)\partial_t+\nabla^* W_s(x)\nabla \partial_t)u(t, x)=0,
\end{equation}
where $W_v, W_s\ge 0$ represents two types of damping. The viscous damping $W_v$ models the dissipation of energy due to viscous drag, for example, from the air or fluids that surround the plate. The structural damping $\nabla^* W_s\nabla$ models the dissipation of energy due to the parts of the plate made of viscoelastic materials. There are many stabilisation results in similar systems: see \cite{jaf90,las92,las02,teb09,dk18,ahr23,lrz23} and the references therein. 

Our semigroup setting naturally accommodates the damped plate equation with a mixture of viscous and viscoelastic structural damping, by letting $P=\Delta^2$, the bi-Laplacian operator, and $Qu=(\sqrt{W_v}u, \sqrt{W_s}\nabla u)$, $Y=L^2(\M)\times L^2(T\M)$. We obtained several new results on the energy decay for the damped plate equations. 
\begin{proposition}\label{2t18}
Let $\M$ be a compact manifold without boundary and $P=\Delta^2$. Let $W_v,W_s\in L^\infty(\M)$ be non-negative, and let $Q\in \mathcal{L}(H^1, Y)$ be defined via $Qu=(\sqrt{W_v}u, \sqrt{W_s}\nabla u)$, $Y=L^2(\M)\times L^2(T\M)$. Let $\Omega_\epsilon$ be the interior of $\{W_v\ge \epsilon\}\cup\{W_s\ge \epsilon\}$. Then the following are true: if for some $\epsilon>0$, 
\begin{enumerate}[wide]
\item $\Omega_\epsilon$ satisfies the geometric control condition (for $\Delta$), then 
\begin{equation}
\|e^{t\Ac}\|_{\Dc\rightarrow\Hcd}\le e^{-Ct}. 
\end{equation}
\item the Schr\"odinger equation (for $\Delta$) is exactly observable from $\Omega_{\e}$, then
\begin{equation}
\|e^{t\Ac}\|_{\Dc\rightarrow\Hcd}\le C\langle t\rangle^{-1}. 
\end{equation}
If we further assume $W_s\equiv 0$, then the same assumption implies $\|e^{t\Ac}\|_{\Dc\rightarrow\Hcd}\le e^{-Ct}$. 
\item $\Omega_\epsilon$ is nonempty, then
\begin{equation}
\|e^{t\Ac}\|_{\mathcal{D}\rightarrow \Hcd}\le \frac{C}{\log^2(2+t)}.
\end{equation}
\end{enumerate}
\end{proposition}
\begin{remark}
(1) and (3) generalises \cite{dk18} and \cite{ahr23} to arbitrary compact manifolds. The exponential decay in (2) when $W_s\equiv 0$ (fully viscous) recovers \cite{jaf90}. 
\end{remark}
\begin{lemma}\label{2t19}
Let $\psi,\tilde\psi$ be smooth cutoffs such that $\psi\tilde\psi=\psi$. Then 
\begin{equation}
\|\psi u\|_{L^2}\le C\lambda^{-1}\|(\Delta^2-\lambda^2)u\|_{H^{-2}}+C\lambda^{-\frac{1}{2}}\|\tilde\psi \nabla u\|_{L^2},
\end{equation}
uniformly in large real $\lambda$. 
\end{lemma}
\begin{proof}
We mainly follow the idea in \cite[Proposition 3.2]{bur20}. Assume without loss of generality $\lambda>0$. Let $(\Delta^2-\lambda^2)u=f$. This implies $(\Delta-\lambda)u=(\Delta+\lambda)^{-1}f=\tilde f$. Pair $(\Delta-\lambda)u$ with $\psi^2 u$ to see
\begin{equation}
\langle (\Delta-\lambda)u, \psi^2 u\rangle=\|\psi\nabla u\|^2-\lambda\|\psi u\|^2+\langle 2(\nabla \psi) \cdot \nabla u, \psi u\rangle.
\end{equation}
Estimate
\begin{gather}
\abs{\langle 2(\nabla \psi)\cdot\nabla u, \psi u\rangle}\le \epsilon \lambda\|\psi u\|^2+C\epsilon^{-1}\lambda^{-1}\|\tilde\psi \nabla u\|^2\\
\abs{\langle (\Delta-\lambda)u, \psi^2 u\rangle} \le \epsilon\lambda\|\psi u\|^2+C\epsilon^{-1}\lambda^{-1}\|\tilde f\|^2. 
\end{gather}
After absorption of terms we have as desired
\begin{equation}
\|\psi u\| \le C\lambda^{-\frac{1}{2}}\|\tilde\psi \nabla u\|+C\lambda^{-1}\|\tilde f\|\le C\lambda^{-\frac{1}{2}}\|\tilde\psi \nabla u\|+C\lambda^{-1}\|f\|_{H^{-2}},
\end{equation}
where we used $\|(\Delta+\lambda)\|_{L^2\rightarrow H^{-2}}\ge C^{-1}$ uniformly in large real $\lambda$. 
\end{proof}

\begin{proof}[Proof of Proposition \ref{2t18}]
1. We claim that for every smooth cutoff $\chi_1$ compactly supported in $\Omega_\epsilon$, 
\begin{equation}\label{2l27}
\|\chi_1 u\|_{L^2}\le C\lambda^{-1}\|(\Delta^2-\lambda^2)u\|_{H^{-2}}+\|\sqrt{W_v}u\|_{L^2}+\|\sqrt{W_s}\nabla u\|_{L^2(T\M)}. 
\end{equation}
To see  this let $\chi_2$ be a smooth cutoff compactly supported in $\operatorname{Int}\{W_s\ge\epsilon\}$, and identically 1 on a neighbourhood of $K=(\supp\chi_1\cap\overline{\{W_s \geq \e\}}) \setminus \operatorname{Int}\{W_v \geq \e\}$ (if such set is empty, we can choose any cutoff $\chi_2$ supported in $\operatorname{Int}(W_s\ge \epsilon)$). Note
\begin{equation}
K\cup\operatorname{Int}\{W_v \geq \e\}\subset \supp\chi_1\subset \Omega_\epsilon=\operatorname{Int}\{W_s \geq \e\}\cup \operatorname{Int}\{W_v \geq \e\}
\end{equation}
implies $K\subset \operatorname{Int}\{W_s\ge\epsilon\}$. To show the existence of such cutoff $\chi_2$, it suffices to show $\operatorname{Int}\{W_s\ge \epsilon\}\setminus K\neq \emptyset$. If $\operatorname{Int}\{W_s\ge \epsilon\}\cap\operatorname{Int}\{W_v\ge \epsilon\}=\emptyset$, then $K\subset \supp\chi_1\setminus \operatorname{Int}\{W_v\ge \epsilon\}$ is compactly supported in $\operatorname{Int}\{W_s\ge \epsilon\}$. If $\operatorname{Int}\{W_s\ge \epsilon\}\cap\operatorname{Int}\{W_v\ge \epsilon\}\neq\emptyset$, we have
\begin{equation}
\operatorname{Int}\{W_s\ge \epsilon\}\setminus K=(\operatorname{Int}\{W_s\ge \epsilon\}\setminus \supp\chi_1)\cup(\operatorname{Int}\{W_s\ge \epsilon\}\cap\operatorname{Int}\{W_v\ge \epsilon\})\neq\emptyset.
\end{equation}
Thus we have $\operatorname{Int}\{W_s\ge \epsilon\}\setminus K\neq \emptyset$ and the existence of $\chi_2$. Now $\supp(\chi_1\chi_2)\subset \operatorname{Int}\{W_s\ge\epsilon\}$ and there exists $\chi_3$ compactly supported in $\operatorname{Int}\{W_s>\epsilon\}$ with $\chi_1\chi_2\chi_3=\chi_1\chi_2$. Apply Lemma \ref{2t19} to see
\begin{equation}\label{2l26}
\|\chi_1\chi_2 u\|_{L^2}\le C\lambda^{-1}\|(\Delta^2-\lambda^2)u\|_{H^{-2}}+C\lambda^{-\frac{1}{2}}\|\chi_3 \nabla u\|_{L^2}.
\end{equation}
On another hand, note $\chi_2\equiv 1$ on $\supp\chi_1\cap \operatorname{Int}\{W_s>\epsilon\}$. This implies $\supp(1-\chi_2)$ is disjoint from $\supp\chi_1\cap \operatorname{Int}\{W_s>\epsilon\}$. Thus
\begin{equation}
\supp(\chi_1(1-\chi_2))\cap\operatorname{Int}\{W_s>\epsilon\}\subset \supp(1-\chi_2)\cap \supp\chi_1\cap \operatorname{Int}\{W_s>\epsilon\}=\emptyset. 
\end{equation}
With $\supp\chi_1\subset \operatorname{Int}\{W_v>\epsilon\}\cup \operatorname{Int}\{W_s>\epsilon\}$, we know $\supp(\chi_1(1-\chi_2))\subset \operatorname{Int}\{W_v>\epsilon\}$. This implies $\|\chi_1(1-\chi_2)u\|_{L^2}\le C\|\sqrt{W_v} u\|_{L^2}$. Now note $\supp\chi_3\subset \operatorname{Int}\{W_s>\epsilon\}$ and \eqref{2l26} implies \eqref{2l27} as desired.

2. When $\Omega_\epsilon$ satisfies the geometric control condition, there is an open subset $\Omega'$ compactly supported in $\Omega_\epsilon$ that also satisfies the geometric control condition. Apply Proposition \ref{2t16} with $\alpha=2$ to see uniformly for all $v\in H^{4}$ we have
\begin{equation}
\|v\|_{L^2}\le C\lambda^{-\frac{3}{2}}\|(\Delta^2-\lambda^2)v\|_{L^2}+C\|\mathbbm{1}_{\Omega'} v\|_{L^2},
\end{equation}
for large real $\lambda$. consider two cutoff functions $\chi,\chi_1\in C^\infty(\M)$ such that $\chi\equiv 1$ on $\Omega'$ and $\chi\chi_1\equiv\chi$ on $\M$, $\supp{\chi_1}\subset \Omega_\epsilon$. Semiclassicalise via $h=\lambda^{-1/2}$ to see
\begin{equation}
\|v\|_{L^2}\le Ch^{-1}\|(h^4\Delta^2-1)v\|_{L^2}+C\|\chi v\|_{L^2},
\end{equation}
uniformly for $h$ small. Let $v=\langle hD\rangle^{-1}u$, and for $u\in H^3$ we have
\begin{equation}\label{2l25}
\|u\|_{H^{-1}_h}\le Ch^{-1}\|(h^4\Delta^2-1)u\|_{H^{-1}_h}+C\|\chi\langle hD\rangle^{-1} u\|_{L^2}.
\end{equation}
Consider that $h^{4}\Delta+1$ is semiclassically elliptic: this implies
\begin{equation}\label{2l24}
\|u\|_{L^2}\le C\|(h^4\Delta^2+1)u\|_{H^{-1}_h} +\bigo(h^{\infty})\|u\|_{H_h^{-N}}\le C\|(h^4\Delta^2-1)u\|_{H^{-1}_h}+C\|u\|_{H^{-1}_h}
\end{equation}
uniformly in $h$ small, after the absorption of the small term. Note $\WF_h(\chi\langle hD\rangle^{-1})\subset \Ell_h\chi_1$. This implies
\begin{equation}
\|\chi\langle hD\rangle^{-1}u\|_{L^2}\le C\|\chi_1 u\|_{L^2}+\bigo(h^{\infty})\|u\|_{H^{-N}_h},
\end{equation}
for arbitrary $N$. Bring \eqref{2l25} into \eqref{2l24} to see
\begin{equation}
\|u\|_{L^2}\le Ch^{-1}\|(h^4\Delta^2-1)u\|_{H^{-1}_h}+C\|\chi_1 u\|_{L^2}+\bigo(h^{\infty})\|u\|_{H^{-N}_h}. 
\end{equation}
Absorb the small error. Now note $\|v\|_{H^{-1}_h}\le Ch^{-1}\|v\|_{H^{-1}}$ uniformly for $v\in H^{-1}$ and $h$ small. Then it reduces to an uniform estimate for all $u\in H^3$ that
\begin{equation}
\|u\|_{L^2}\le C\lambda^{-1}\|(\Delta^2-\lambda^2)u\|_{H^{-1}}+C\|\chi_1 u\|_{L^2},
\end{equation}
uniformly for real $\lambda$ large. Now \eqref{2l27} gives
\begin{equation}
\|u\|_{L^2}\le C\lambda^{-1}\|(\Delta^2-\lambda^2)u\|_{H^{-1}}+C\|Q u\|_{Y}.
\end{equation}
Now apply Theorem \ref{thmcontrol} with $\mu=0$, $\gamma=\frac{1}{4}$, $M=m=C$ to see
\begin{equation}
\|(\Ac+i\lambda)^{-1}\|_{\mathcal{L}(\Hc)}\le C,
\end{equation}
for $\lambda$ real and $\abs{\lambda}$ large. 

3a. Assume the Schr\"odinger equation is exactly observable from $\Omega_\epsilon$. There is $\Omega'$ compactly supported in $\Omega_\epsilon$ from which the Schr\"odinger equation is also exactly observable. We then have \eqref{2l28}:
\begin{equation}
\|u\|_{L^2}\le C\|(\Delta-\lambda)u\|_{L^2}+C\|\mathbbm{1}_{\Omega'}u\|_{L^2},
\end{equation}
uniformly for real $\lambda$ with $\abs{\lambda}$ large. After semiclassicalisation with $h=\lambda^{-1/2}$, it becomes
\begin{equation}\label{2l29}
\|u\|_{L^2}\le Ch^{-2}\|(h^2\Delta-1)u\|_{L^2}+C\|\mathbbm{1}_{\Omega'}u\|_{L^2}.
\end{equation}
Now note $(h^2\Delta-1)u=(h^2\Delta+1)^{-1}(h^4\Delta^2-1)u$, where $h^2\Delta+1$ is semiclassically elliptic and invertible from $L^2$ to $H^{-2}_h$. Thus \eqref{2l29} becomes
\begin{equation}\label{2l30}
\|u\|_{L^2}\le Ch^{-2}\|(h^4\Delta^2-1)u\|_{H^{-2}_h}+C\|\mathbbm{1}_{\Omega'}u\|_{L^2}.
\end{equation}
Replace $H^{-2}_h$ by $H^{-1}_h$ and it reduces to 
\begin{equation}
\|u\|_{L^2}\le C\lambda^{-\frac{1}{2}}\|(\Delta^2-\lambda^2)u\|_{H^{-1}}+C\|\mathbbm{1}_{\Omega'}u\|_{L^2},
\end{equation}
for large real $\lambda$. Now consider a cutoff function $\chi\in C^\infty(\M)$ such that $\chi\equiv 1$ on $\Omega'$ and $\supp{\chi}\subset \Omega_\epsilon$. We have from \eqref{2l27} that
\begin{equation}
\|\mathbbm{1}_{\Omega'}u\|_{L^2}\le \|\chi u\|_{L^2}\le C\lambda^{-1}\|(\Delta^2-\lambda^2)u\|_{H^{-2}}+\|\sqrt{W_v}u\|_{L^2}+\|\sqrt{W_s}\nabla u\|_{L^2(T\M)}. 
\end{equation}
We end up with
\begin{equation}
\|u\|_{L^2}\le C\lambda^{-\frac{1}{2}}\|(\Delta^2-\lambda^2)u\|_{H^{-1}}+C\|Q u\|_{Y},
\end{equation}
for large real $\lambda$. Apply Theorem \ref{thmcontrol} with $\mu=0$, $\gamma=\frac{1}{4}$, $M=\abs{\lambda}^{\frac{1}{2}}$, $m=C$ to see
\begin{equation}
\|(\Ac+i\lambda)^{-1}\|_{\mathcal{L}(\Hc)}\le C\abs{\lambda}, 
\end{equation}
uniformly for real $\lambda$ with $\abs{\lambda}$ large.

3b. If we further assume $W_s\equiv 0$, then $Q\in \mathcal{L}(L^2, Y)$ is bounded with $\gamma=0$. In \eqref{2l30}, by replacing $H^{-2}_h$ by $L^2$ we have
\begin{equation}
\|u\|_{L^2}\le C\lambda^{-1}\|(\Delta^2-\lambda^2)u\|_{L^2}+C\|\mathbbm{1}_{\Omega'}u\|_{L^2}. 
\end{equation}
Proceed similarly as in Step 3a, this time with $\gamma=0$ and $M(\lambda)=C$ to obtain 
\begin{equation}
\|(\Ac+i\lambda)^{-1}\|_{\mathcal{L}(\Hc)}\le C,
\end{equation}
when $W_s\equiv 0$. 

4. By Proposition \ref{2t17} with $\alpha=2$, we have
\begin{equation}
\|u\|_{H^3}\le e^{C\lambda^{1/2}}(\|(\Delta^2-\lambda^2)u\|_{H^{-1}}+\|\chi u\|_{L^2}),
\end{equation}
uniformly for large real $\lambda>0$, where $\chi$ is a non-trivial cutoff compactly supported in $\Omega_\epsilon$. Then \eqref{2l27} implies
\begin{equation}
\|u\|_{L^2}\le e^{C\lambda^{1/2}}(\|(\Delta^2-\lambda^2)u\|_{H^{-1}}+\|Q u\|_{Y}).
\end{equation}
Apply Theorem \ref{thmcontrol} with $\mu=0$, $\gamma=\frac{1}{4}$, $M=m=e^{C\lambda^{1/2}}$ to see
\begin{equation}
\|(\Ac+i\lambda)^{-1}\|_{\mathcal{L}(\Hc)}\le e^{C\lambda^{1/2}},
\end{equation}
uniformly for real $\lambda$ with $\abs{\lambda}$ large. 

5. We now check the unique continuation. Note $\operatorname{UCP}_{\Delta, \mathbbm{1}_{\Omega'}}$ for any open nonempty $\Omega'$ compactly supported in $\Omega_\epsilon$ on compact manifolds. By Lemma \ref{uniqelemma}(3), we have $\operatorname{UCP}_{\Delta^2, \mathbbm{1}_{\Omega'}}$. Then by \eqref{2l27}, $\operatorname{UCP}_{\Delta^2, Q}$ holds. Now apply Lemma \ref{3t7} to conclude the energy decay. 
\end{proof}

\section{Proofs of main theorems}\label{s3}
In this section, we prove Theorems \ref{thmbackward}, \ref{thmcontrol}, \ref{thmdilate} and \ref{t4}. In \S \ref{s3-1}, we show an interpolation inequality and some unique continuation properties in the functional calculus of $P$. In \S \ref{s3-2}, we show the 0-frequencies of the evolution semigroup $e^{t\Ac}$ can be isolated out even when the perturbation is not relatively compact, and obtain a semigroup decomposition result. In \S \ref{s3-3}, we show the energy decay of $e^{t\Ac}$, the stability of $\Ac$, and the resolvent estimate of $P_\lambda$ are equivalent. In \S \ref{s3-4}, we prove Theorem \ref{thmcontrol} of control and stability. In \S \ref{s3-5}, we prove Theorem \ref{thmbackward} of optimal backward uniqueness. In \S \ref{s3-6}, we prove Theorem \ref{thmdilate} of dilation and contraction. In \S \ref{s3-7}, we prove Theorem \ref{t4} of non-uniform Hautus test. 
\subsection{Preliminaries}\label{s3-1}
We remind readers that the non-negative self-adjoint operator $P:H_{1}\rightarrow H_0=H$ admits a spectral resolution
\begin{equation}
P u=\int_0^\infty\rho^2\ dE_\rho(u),
\end{equation}
where $E_\rho$ is a projection-valued measure on $H$ and $\supp E_\rho\subset [0, \infty)$. For $s\in\mathbb{R}$, the scaling operators were defined as
\begin{equation}
\Lambda^{s}u=\int_0^\infty (1+\rho^2)^{s}\ dE_\rho(u).
\end{equation}
We define the interpolation spaces $H_s=\Lambda^{-s}(H)$, equipped with the norm $\|u\|_{s}=\|\Lambda^{s}u\|_{H}$. Those operators $\Lambda^{-s}: H\rightarrow H^s$ are unitary, and they commute with $P$. The operator $P$ extends to a self-adjoint operator in $\mathcal{L}(H_{1/2}, H_{-1/2})$. Moreover $H_{-s}$ is isomorphic to $\mathcal{L}(H_s, H)$, the dual space of $H_s$ with respect to $H$. Note also, for $s<t$ we have $H_t \subset H_s$ compactly, and for $u \in H_t$ we have $\|u\|_s \leq C \|u\|_t$. Among those spaces we have an interpolation inequality:
\begin{lemma}[Interpolation]\label{interpolation}
Let $s < t$. For every $r\in (s,t)$, for every $u\in H_t$, and every $\alpha>0$, 
\begin{equation}\label{3l8}
\|u\|_{r}\le \left( \alpha\|u\|_{t}+\alpha^{\frac{s-r}{t-r}}\|u\|_s \right). 
\end{equation}
\end{lemma}
\begin{proof}
Consider the algebraic inequality
\begin{multline}
(1+\rho^2)^r=(1+\rho^2)^t\alpha^{\frac{r-s}{t-s}}\left(\alpha^{\frac{s-r}{t-r}}(1+\rho^2)^{s-t}\right)^{\frac{t-r}{t-s}}\\
\le (1+\rho^2)^t\left(\frac{r-s}{t-s}\alpha+\frac{t-r}{t-s}\alpha^{\frac{s-r}{t-r}}(1+\rho^2)^{s-t}\right)
\le \alpha(1+\rho^2)^t+\alpha^{\frac{s-r}{t-r}}(1+\rho^2)^s,
\end{multline}
where we used the generalised AM-GM inequality in the first inequality. Now 
\begin{equation}
\|u\|_r\le \left\|\int_0^\infty \alpha(1+\rho^2)^t\ dE_\rho(u)\right\| +  \left\|\int_0^\infty \alpha^{\frac{s-r}{t-r}}(1+\rho^2)^s\ dE_\rho(u)\right\| 
 \le \left(\alpha\|u\|_{t} +\alpha^{\frac{s-r}{t-r}}\|u\|_s \right),
\end{equation}
as desired.
\end{proof}
Let $f$ be a real-valued Borel function on $\mathbb{R}$. We can define $f(P)$ via
\begin{equation}
f(P) u=\int_0^\infty f(\rho^2)\ dE_\rho(u),
\end{equation}
where the domain of $f(P)$ is the set of $u$ that $f(P)u\in H$. As a specific example, $P^{\frac{1}{2}}\in\mathcal{L}(H_{1/2},H)$ is essentially self-adjoint on $H$ and $P=(P^{\frac{1}{2}})^{2}$. 

We later will study the stability of the semigroup $e^{t\mathcal{A}}$ generated by $\mathcal{A}$. In some cases, there will be some oscillating modes which do not see the damping $Q^*Q$ at low frequencies. An effective damping would intercept oscillating modes at all positive frequencies, and we want a condition that ensures this. These are the unique continuation assumptions on $P$ and $Q$: 
\begin{definition}[Unique continuation]\label{3t6}
We say the unique continuation principle for $P$ and observation operator $Q\in\mathcal{L}(H_\gamma, Y)$, denoted $\operatorname{UCP}_{P,Q}$, holds if for all $\lambda>0$, all $u\in H_{1/2}$ and $u\neq 0$,
\begin{equation}
(P-\lambda^2)u=0\ \Rightarrow\ Q u\neq 0.
\end{equation}
\end{definition}
The following lemma gives practical conditions to check whether the unique continuation principle holds. 
\begin{lemma}[Unique continuation for functional calculus]\label{uniqelemma}
Let $\gamma\in[0, \frac{1}{2}]$, $Q\in \mathcal{L}(H_\gamma, H)$. Then the following are true:
\begin{enumerate}
\item If the Schr\"odinger equation $(i\partial_t +P)u=0$ is exactly observable via $Q$, then $\operatorname{UCP}_{P, Q}$ holds. 
\item Let $Q_1\in\mathcal{L}(H_{\gamma}, Y_1)$, $Q_2\in\mathcal{L}(H_{\gamma}, Y_2)$, and there is $C>0$ such that $\|Q_1 u\|_{Y_1}\le C\|Q_2 u\|_{Y_2}$ for all $u\in H_1/\ker P$. Then if $\operatorname{UCP}_{P, Q_1}$ holds, then $\operatorname{UCP}_{P, Q_2}$ holds. 
\item Let $f$ be a continuous function injectively mapping $[0,\infty)$ into $[0,\infty)$, and $f(0)=0$. If $\operatorname{UCP}_{P, Q}$ holds, then $\operatorname{UCP}_{f(P), Q}$ holds. 
\item Let $g$ be a continuous function mapping $[0,\infty)$ into $[0,\infty)$, and $g>0$ on $(0,\infty)$. Then $\operatorname{UCP}_{P, g(P)}$ holds. 
\item Assume $V\in\mathcal{L}(H_{1/2}, H_\gamma)$ commutes with $P$. If $\operatorname{UCP}_{P, Q}$ and $\operatorname{UCP}_{P, V}$ hold, then $\operatorname{UCP}_{P, QV}$ holds. 
\item If $\operatorname{UCP}_{P, Q}$ holds, then $\ker (P-i\lambda Q^*Q-\lambda^2)=\{0\}$ for all $\lambda>0$. 
\end{enumerate}
\end{lemma}
\begin{proof}
1. 
Let $(P-\lambda^2)u=0$ for $u\in H_{1/2}$, $u\neq 0$. This implies $u\in H_{s}$ for every $s\in \mathbb{R}$. Specifically, $u\in H_{1}$. Now assume the Schrödinger equation $(i\partial_t +P)u=0$ is exactly observable via $Q$. By \cite{mil12} we have
\begin{equation}
	0<\|u\|\le C\|(P-\lambda^2)u\|+C\|Q u\|_Y=C\|Q u\|_Y
\end{equation} 
This implies $Qu\neq 0$.

2. Assume $\|Q_1 u\|_{Y_1}\le C\|Q_2 u\|_{Y_2}$ for all $u\in H_1/\ker P$. Then $\|Q_1 \Pi_0^\perp u\|_{Y_1}\le C\|Q_2 \Pi_0^\perp u\|_{Y_2}$ for all $u\in H_1$, where $\Pi_0^\perp$ is the orthogonal projection onto $(\ker P)^{\perp}$, defined in \eqref{3l34}. Assume $(P-\lambda^2)u=0$ for some $\lambda>0$, $u\neq 0,$ then $(P-\lambda^2)\Pi_0^\perp u=0$.  $\operatorname{UCP}_{P, Q_1}$ then implies $Q_1 \Pi_0^\perp u\neq 0$. Then $Q_2\Pi_0^\perp u\neq 0$ and thus $Q_2 u\neq 0$. 

3. Let $(f(P)-\lambda^2)u=0$ for $\lambda\in (0,\infty)$, so $\lambda\neq \sqrt{f(0)}$. This implies 
\begin{equation}
\|(f(P)-\lambda^2)u\|^2=\int_0^\infty (f(\rho^2)-\lambda^2)^2d\langle E_\rho u, u\rangle=0. 
\end{equation}
This implies $\langle E_\rho u, u\rangle=0$ away from $\rho_0=\sqrt{f^{-1}(\lambda^2)}\neq 0$. Then $u=E_{\rho_0} u$ and
\begin{equation}
(P-\rho_0^2)u=\int_{0}^\infty (\rho^2-\rho_0^2)\ dE_\rho(u)=0. 
\end{equation}
Thus $(P-\rho_0^2)u=0$ and our assumption implies $Qu\neq0$.

4. Assume $(P-\lambda^2)u=0, u \neq 0$. Then
\begin{multline}
0<\|u\|^2=\int_0^\infty \ d\langle E_\rho u, u\rangle\le C\int_0^\infty (\rho^2-\lambda^2)^2+g(\rho^2)^2 \ d\langle E_\rho u, u\rangle\\
\le C\|(P-\lambda^2) u\|^2+C\|g(P)u\|^2= C \|g(P)u\|^2.
\end{multline}
So $g(P)u \neq 0$.

5. Assume $(P-\lambda^2)u=0$, $QVu=0$. The commutativity implies $(P-\lambda^2)Vu=0$. Then $\operatorname{UCP}_{P, Q}$ implies $Vu=0$. Now we have $(P-\lambda^2)Vu=0$, $Vu=0$, and $\operatorname{UCP}_{P, V}$ gives $u=0$.

6. Suppose $(P-i\lambda Q^* Q - \lambda^2)u=0$. Pairing with $u$
\begin{equation}
\<Pu,u\> -\lambda^2 \<u,u\> - i \lambda \|Qu\|^2 = 0.
\end{equation} 
Taking real and imaginary parts gives $(P-\lambda^2)u=0$ and $Qu=0$, respectively. Thus by $\operatorname{UCP}_{P,Q}$ we have $u=0$.
\end{proof}

\subsection{Semigroup decomposition}\label{s3-2}
Fix $\gamma\in[0,\frac{1}{2}]$ and an observation operator $Q\in\mathcal{L}(H_{\gamma}, Y)$, where the observation space $Y$ is a Hilbert space. We call its adjoint $Q^*\in\mathcal{L}(Y, H_{-\gamma})$ the corresponding control operator. We will consider damping of the form $Q^*Q\in \mathcal{L}(H_\gamma, H_{-\gamma})$ and the evolution semigroup for the operator family $P_\lambda=P-i\lambda Q^*Q-\lambda^2$. Let the state space $\mathcal{H}=H_{1/2}\times H$. Let 
\begin{equation}
\mathcal{D}=\{(u,v)\in H_{1/2}\times H_{1/2}: P u+Q^*Q v\in H\}.
\end{equation}
Note that those spaces are equipped with the norms
\begin{equation}
\|(u,v)\|_{\mathcal{H}}^2=\|\Lambda^\frac{1}{2}u\|_H^2+\|v\|_{H}^2, \ \|(u,v)\|_{\mathcal{D}}^2=\|\Lambda^\frac{1}{2}u\|_H^2+\|\Lambda^\frac{1}{2}v\|_{H}^2+\|Pu+Q^*Qv\|_H^2. 
\end{equation}
We show in Lemma \ref{semigrouplemma} that the damped semigroup generator 
\begin{equation}
\mathcal{\Ac}=\mathcal{\Ac}_{P,Q}=
\begin{pmatrix}
0 & \id\\
-P & -Q^*Q
\end{pmatrix}
:\mathcal{D}\rightarrow \mathcal{H}
\end{equation}
generates a strongly continuous semigroup $e^{t\Ac}$ on $\mathcal{H}$ that satisfies $\partial_t e^{t\Ac}|_{\mathcal{D}}=\Ac$. 

Note $\ker \Ac=\ker P\times \{0\}$ is not trivial if $\ker P$ is not so. The positive dimension of $\ker P$ prohibits $P$ from being a positive (rather than non-negative) self-adjoint operator on $H$: $P$ being positive is the key assumption in \cite[\S 2.2]{cpsst19} that allows the generation of a contraction semigroup $e^{t\Ac}$. There are many physical systems in which $\ker P$ is not trivial, for example, $P=\Delta$ when studying waves on compact manifolds without boundary. 

In order to analyse the energy decay, we will decompose the semigroup $e^{t
\Ac}$ into a stationary part and $e^{t\Acd}$, evolution outside $\ker\Ac$. The strategy is the following. Firstly we show in Lemma \ref{lemmainvertibility} that if $\ker\Ac$ is not trivial, $0$ must be an isolated eigenvalue. Then we can define the Riesz projector of $\Ac$ that projects $\mathcal{H}$ onto $\ker\Ac$, given by
\begin{equation}
\Pi=\frac{1}{2\pi i}\int_\gamma (z\id-\mathcal{A})^{-1}~dz,
\end{equation}
where $\gamma$ is a small circle around $0$ in $\mathbb{C}$ containing $0$ as the only eigenvalue of $\Ac$ in its interior. Consider the range of $\Ac$, $\Hcd=\mathcal{A}(\mathcal{D})$. Define the complement Riesz projector $\Pi_\bullet=\id-\Pi$, then $\Pi_\bullet$ projects $\mathcal{H}$ onto $\Hcd$. The Riesz projectors $\Pi_\bullet$, $\Pi$ non-orthogonally decompose $\mathcal{H}$ into $\Hcd\oplus \ker\Ac$. Note
\begin{equation}
\Pi\Pi_\bullet=\Pi_\bullet\Pi=0, \ \Hcd=\ker \Pi, \ \Pi \mathcal{H}=\ker\mathcal{A}=\ker P\times \{0\}.
\end{equation}
Let $\Dcd=\Pi_\bullet \mathcal{D}=\mathcal{D}\cap \Hcd$. Equip $\Hcd, \Dcd$ with the energy norms
\begin{equation}
\|(u,v)\|_{\Hcd}^2=\|P^\frac{1}{2}u\|_H^2+\|v\|_{H}^2, \ \|(u,v)\|_{\Dcd}^2=\|P^\frac{1}{2}u\|_H^2+\|\Lambda^\frac{1}{2}v\|_{H}^2+\|Pu+Q^*Qv\|_H^2. 
\end{equation}
The energy norm ${\|\cdot\|}_{\Hcd}$ represents the total mechanical energy of the given state $(u, v)\in \Hcd$, and is equivalent to the norm ${\|\cdot\|}_{\Hc}$ since $P$ is positive on $H/\ker P$. ${\|\cdot\|}_{\Hcd}$ is a seminorm on $\Hcd$ vanishing on $\ker\Ac$: it totally ignores the evolution of $e^{t\Ac}$ on $\ker\mathcal{A}$, the zero-energy space. Then the Riesz projectors decompose $\mathcal{D}=\Dcd\oplus\ker\Ac$. Let
\begin{equation}
\Acd=\Ac|_{\Dcd}=\begin{pmatrix}
0 & \id\\
-P & -Q^*Q
\end{pmatrix}
:\Dcd\rightarrow \Hcd. 
\end{equation}
The nature of Riesz projectors implies
\begin{equation}
\Pi_\bullet \mathcal{A}=\mathcal{A} \Pi_\bullet=\Acd \Pi_\bullet, \ \Pi \mathcal{A}=\mathcal{A}\Pi=0, \ \mathcal{A}=\Acd\Pi_\bullet. 
\end{equation}
We show in Lemma \ref{semigrouplemma} that $\Acd$ generates a contraction semigroup $e^{t\Acd}$ on $\Hcd$, and we can decompose
\begin{equation}
e^{t\Ac}=e^{t\Acd}\Pi_\bullet+\Pi. 
\end{equation}
This implies
\begin{equation}\label{3l29}
\Pi_\bullet e^{t\Ac}=e^{t\Acd}\Pi_\bullet, \ \|e^{t\Ac} (u,v)\|_{\Hcd}=\|e^{t\Acd}\Pi_\bullet (u,v)\|_{\Hcd}.
\end{equation}
by noting $\|\Pi_\bullet (u,v)\|_{\Hcd}=\|(u,v)\|_{\Hcd}$ for any $(u,v)\in \mathcal{H}$. It is interpreted that the mechanical energy of the state $e^{t\Ac}(u,v)$ is completely represented by that of the state $e^{t\Acd}(u,v)$, and it suffices to look at only $e^{t\Acd}$ (rather than $e^{t\Ac}$) to understand the energy decay. 

Lemmata \ref{lemmainvertibility}, \ref{semigrouplemma} are devoted to establishing the semigroup properties claimed above. See Figure \ref{f2} for a spectral illustration. 
\begin{figure}
\includegraphics[page=2]{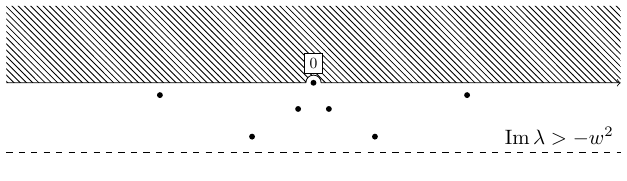}
\caption{Spectral illustration of $P_\lambda$. In the shaded region, $P_\lambda$ is invertible. At $0$, $P_\lambda$ may have a isolated pole of finite degree. In the region above the dotted line, $P_\lambda$ is Fredholm of index $0$.}\label{f2}
\end{figure}
\begin{lemma}[Meromorphic continuation of $P_\lambda^{-1}$]\label{lemmainvertibility}
Let $Q:H_{1/2}\rightarrow Y$ be bounded, and $w=\|Q\|_{H_{1/2}\rightarrow Y}>0$. Then in $\{\cim \lambda> -w^{-2}\}$, $P_{\lambda}=P-i\lambda Q^*Q-\lambda^2: H_{1/2}\rightarrow H_{-1/2}$ is an analytic family of Fredholm operators of index 0, and $P_\lambda^{-1}: H_{-1/2}\rightarrow H_{1/2}$ is a meromorphic family of Fredholm operators with poles of finite rank. 
\end{lemma}
\begin{remark}\label{3t10}
\begin{enumerate}[wide]
\item When $Q:H_\gamma\rightarrow H$ for some $\gamma<\frac{1}{2}$, $Q^*Q:H_\gamma\rightarrow H_\gamma$ is relatively compact with respect to $P: H_{1/2}\rightarrow H_{-1/2}$, and one can show $P_\lambda^{-1}$ is a meromorphic family of Fredholm operators with poles of finite order in the entire $\mathbb{C}$. When $\gamma=\frac{1}{2}$, $Q^*Q$ is no longer relatively compact with respect to $P$, and the loss of compactness in the perturbative term is not trivial. Here we carefully address the limit case and prove that we can still meromorphically continue the resolvent into the lower half plane until the line $\{\cim\lambda>-w^{-2}\}$. 

\item Here $\{\cim\lambda>-w^{-2}\}$ is sharp in general. Consider $Q^*Q=P$ and $P_{-i}=1: H_{1/2}\rightarrow H_{-1/2}$ is not Fredholm at $\lambda=-i\in \{\cim\lambda=-w^{-2}\}$. 
\end{enumerate}

\end{remark}
\begin{proof}
The major tools we use here are perturbative arguments and analytic Fredholm theory. See \cite[Appendix C]{dz19} for further details. 

1. We first show $P_\lambda:H_{1/2}\rightarrow H_{-1/2}$ is invertible in $\{\cim\lambda>0\}$. Consider the bilinear form
\begin{equation}
\langle P_\lambda u, v\rangle=\langle P^{\frac{1}{2}}u, P^{\frac{1}{2}}v\rangle-i\lambda\langle Qu, Qv\rangle_Y-\lambda^2\langle u, v\rangle
\end{equation} 
for $u, v\in H_{1/2}$. It is a bounded bilinear form on $H_{1/2}$. Moreover, it is coercive:
\begin{equation}
\langle P_\lambda u, u\rangle=\|P^{\frac{1}{2}}u\|^2+(\abs{\cim\lambda}^2-\abs{\cre\lambda}^2)\|u\|^2+\cim\lambda\|Qu\|_Y^2-i\cre\lambda(2\cim\lambda\|u\|^2 +\|Qu\|_Y^2).
\end{equation}
When $\cre\lambda=0$, $\cim\lambda>0$, 
\begin{equation}
\cre{\langle P_\lambda u, u\rangle}\ge \|P^{\frac{1}{2}}u\|^2+\abs{\cim\lambda}^2\|u\|^2\ge \min\{1, (\cim \lambda)^2\}\|u\|_{\frac{1}{2}}^2.
\end{equation}
For every fixed $\delta>0$, $\lambda\in\{\cim\lambda>\delta\abs{\cre\lambda}>0\}$, we have
\begin{equation}
\cre{\langle P_\lambda u, u\rangle}\ge \|P^{\frac{1}{2}}u\|^2-\abs{\cre\lambda}^2\|u\|^2, \ \cim{\langle P_\lambda u, u\rangle}\ge 2\delta\abs{\cre{\lambda}}^2\|u\|^2,
\end{equation}
and 
\begin{align}
\abs{\langle P_\lambda u, u\rangle}^2\ge &\left(\sqrt{\frac{2}{2+4\delta^2}}\|P^{\frac{1}{2}}u\|^2-\sqrt{\frac{2+4\delta^2}{2}}\abs{\cre\lambda}^2\|u\|^2\right)^2\\
&+\frac{4\delta^2}{2+4\delta^2}\|P^{\frac{1}{2}}u\|^{4}+2\delta^2\abs{\cre\lambda}^{4}\|u\|^4\ge c_{\delta, \cre\lambda} \|u\|_{\frac{1}{2}}^4. 
\end{align}
for some small $c_{\delta, \cre\lambda}>0$. By the Lax--Milgram theorem, $P_\lambda$ is invertible on $\{\cim\lambda>0\}$.

2. We now show $P_\lambda$ is Fredholm on $\{\cim\lambda> -w^{-2}\}$. Let $\rho_0>0$ be chosen later. Consider the spectral projector to frequencies $[0, \rho_0)$ of $P$, 
\begin{equation}\label{3l34}
\Pi_0=\int_0^{\rho_0} dE_\rho, \ \Pi_0^*=\Pi_0, \ \Pi_0^2=\Pi_0.
\end{equation}
Note that $\Pi_0: H_{1/2}\rightarrow H_{-1/2}$ is compact and depends on parameter $\rho_0>0$. Denote $\Pi_0^{\perp}=\id-\Pi_0$. Our strategy is the following decomposition: for fixed $\rho_0>0$, $\lambda\in \mathbb{C}$,
\begin{equation}\label{3l36}
P_\lambda=F_{\lambda, \rho_0}+(\Pi_0 P_\lambda \Pi_0+\Pi_0 P_\lambda \Pi_0^\perp+\Pi_0^\perp P_\lambda \Pi_0-\langle \rho_0\rangle^2\Pi_0), \ F_{\lambda, \rho_0}=\Pi_0^\perp P_\lambda \Pi_0^\perp+\langle \rho_0\rangle^2\Pi_0
\end{equation}
Note that $P_\lambda-F_{\lambda, \rho_0}$ is compact from $H_{1/2}\rightarrow H_{-1/2}$. It suffices to show for each $\lambda\in \{\cim\lambda> -w^{-2} \}$, there is $\rho_0>0$ that $F_{\lambda,\rho_0}$ is Fredholm. 

3. We show show at each $\lambda\in\mathbb{R}$, for all $\rho_0>\abs{\lambda}$, $F_{\lambda, \rho_0}$ is invertible from $H_{1/2}$ to $H_{-1/2}$. Consider
\begin{multline}
\cre(\langle \Pi_0^\perp P_\lambda \Pi_0^\perp u, u\rangle)=\langle (P-\lambda^2)\Pi_0^\perp u, \Pi_0^\perp u\rangle=\int_{\rho_0}^\infty \frac{\rho^2-\lambda^2}{1+\rho^2}\langle \rho\rangle^2 ~d\langle E_\rho u, u \rangle \\
\ge \int_{\rho_0}^\infty \left(1-\frac{1+\lambda^2}{1+\rho_0^2}\right)\langle \rho\rangle^2 ~d\langle E_\rho u, u \rangle\ge \left(1-\frac{1+\lambda^2}{1+\rho_0^2}\right) \|\Pi_0^\perp u\|_{\frac{1}{2}}^2. 
\end{multline} 
On another hand, 
\begin{equation}
\|\Pi_0 u\|^2 \ge \langle \rho_0\rangle^{-2}\int_0^{\rho_0}\langle \rho\rangle^2~d\langle E_\rho u,u\rangle=\langle \rho_0\rangle^{-2}\|\Pi_0 u\|_{\frac{1}{2}}^2.
\end{equation}
This implies via the orthogonality of $\Pi_0^{\perp}$ and $\Pi_0$:
\begin{multline}
\cre{\langle F_{\lambda, \rho_0} u, u\rangle}=\cre{\langle \Pi_0^\perp P_\lambda \Pi_0^\perp u, u\rangle}+\langle \rho\rangle^{-2}\|\Pi_0 u\|^{2}_{\frac{1}{2}} \ge \left(1-\frac{1+\lambda^2}{1+\rho_0^2}\right)\|\Pi_0^\perp u\|_{\frac{1}{2}}^2\\
+\|\Pi_0u\|_{\frac{1}{2}}^2\ge \left(1-\frac{1+\lambda^2}{1+\rho_0^2}\right)\|u\|_{\frac{1}{2}}^2. 
\end{multline}
By the Lax--Milgram theorem, $F_{\lambda, \rho_0}$ is invertible and $P_\lambda$ is Fredholm on $\lambda\in\mathbb{R}$. Moreover, for fixed $\lambda\in\mathbb{R}$ and $\rho_0>\abs{\lambda}$, we have
\begin{equation}
\|F_{\lambda,\rho_0}^{-1}\|_{H_{-1/2}\rightarrow H_{1/2}}\le \left(1-\frac{1+\lambda^2}{1+\rho_0^2}\right)^{-1}.
\end{equation}
Note that the right hand side can be as close as possible to $1$ by choosing a large $\rho_0$. 

4. Fix $\lambda_0\in \mathbb{R}$ and $\rho_0>\abs{\lambda_0}$, we know that $F_{\lambda_0, \rho_0}$ is invertible. We now show that for $\lambda\in\mathbb{C}$ close to $\lambda_0$, $F_{\lambda, \rho_0}$ is also invertible. Note
\begin{equation}
F_{\lambda, \rho_0}=F_{\lambda_0, \rho_0}-i(\lambda-\lambda_0)\Pi_0^\perp Q^*Q\Pi_0^\perp-(\lambda^2-\lambda_0^2)\Pi_0^\perp. 
\end{equation}
This implies
\begin{equation}\label{3l35}
F_{\lambda_0, \rho_0}^{-1} F_{\lambda, \rho_0}=\id-i(\lambda-\lambda_0)F_{\lambda_0, \rho_0}^{-1}\Pi_0^\perp Q^*Q\Pi_0^\perp-(\lambda^2-\lambda_0^2)F_{\lambda_0, \rho_0}^{-1}\Pi_0^\perp
\end{equation}
We estimate terms on the right: 
\begin{equation}
\|F_{\lambda_0, \rho_0}^{-1}\Pi_0^\perp Q^*Q\Pi_0^\perp u\|_{\frac{1}{2}}\le w^2\left(1-\frac{1+\lambda^2}{1+\rho_0^2}\right)^{-1}\|u\|_{\frac{1}{2}}.
\end{equation}
On another hand, 
\begin{equation}
\|\Pi_0^{\perp}u\|_{-\frac{1}{2}}^2=\int_{\rho_0}^\infty \langle \rho\rangle^{-2}\langle \rho\rangle~d\langle E_\rho u, u\rangle\le \langle \rho_0\rangle^{-2}\|u\|_{\frac{1}{2}}^2 
\end{equation}
implies
\begin{equation}
\|F_{\lambda_0, \rho_0}^{-1}\Pi_0^\perp u\|_{\frac{1}{2}}\le \langle \rho_0\rangle^{-1}\left(1-\frac{1+\lambda^2}{1+\rho_0^2}\right)^{-1}\|u\|_{\frac{1}{2}}. 
\end{equation}
Thus
\begin{multline}
\|i(\lambda-\lambda_0)\Pi_0^\perp Q^*Q\Pi_0^\perp+(\lambda^2-\lambda_0^2)\Pi_0^\perp\|_{H_{1/2}\rightarrow H_{1/2}}\\
\le \abs{(\lambda-\lambda_0)(w^2+\langle \rho_0\rangle^{-1}(\lambda+\lambda_0))}\left(1-\frac{1+\lambda^2}{1+\rho_0^2}\right)^{-1}
\end{multline}
is strictly bounded by $1$ when
\begin{equation}
\abs{\lambda-\lambda_0}< w^{-2}\left(1-\frac{1+\lambda^2}{1+\rho_0^2}\right)\left(1+\langle \rho_0\rangle^{-2}(2\abs{\lambda_0}+w^{-2})w^{-2}\right)^{-1}=w^{-2}R_{\rho_0}. 
\end{equation}
Indeed, for $\abs{\lambda-\lambda_0}<w^{-2}R_{\rho_0}$, we have
\begin{equation}
\|i(\lambda-\lambda_0)\Pi_0^\perp Q^*Q\Pi_0^\perp+(\lambda^2-\lambda_0^2)\Pi_0^\perp\|_{H_{1/2}\rightarrow H_{1/2}}< \frac{1+\langle \rho_0\rangle^{-1}\abs{\lambda+\lambda_0}w^{-2}}{1+\langle \rho_0\rangle^{-1}(2\abs{\lambda_0}+w^{-2})w^{-2}}<1,
\end{equation}
the last step of which we used that $R_{\rho_0}<1$. Indeed, $R_{\rho_0}$ can be made arbitrarily close to $1$ by choosing large $\rho_0$. Now the Neumann series
\begin{equation}
\left(\id-i(\lambda-\lambda_0)F_{\lambda_0, \rho_0}^{-1}\Pi_0^\perp Q^*Q\Pi_0^\perp-(\lambda^2-\lambda_0^2)F_{\lambda_0, \rho_0}^{-1}\Pi_0^\perp\right)^{-1}=\sum_{k=0}^\infty (i(\lambda-\lambda_0)\Pi_0^\perp Q^*Q\Pi_0^\perp+(\lambda^2-\lambda_0^2)\Pi_0^\perp)^k
\end{equation}
converges on $\{\abs{\lambda-\lambda_0}<w^{-2}R_{\rho_0}\}$. Thus $F_{\lambda_0, \rho_0}^{-1} F_{\lambda, \rho_0}$ in \eqref{3l35} is invertible on $H_{1/2}$. This implies $F_{\lambda,\rho_0}:H_{1/2}\rightarrow H_{-1/2}$ is invertible at all $\{\abs{\lambda-\lambda_0}<w^{-2}R_{\rho_0}\}$. Now note $R_{\rho_0}\rightarrow 1$ as $\rho_0\rightarrow \infty$. This implies for any $\lambda\in\{0\ge\cim\lambda>-w^{-2}\}$, there exists $\rho_0>\abs{\cre\lambda}$ such that $\abs{\cim\lambda}<w^{-2} R_{\rho_0}$ and thus $F_{\lambda, \rho_0}$ is invertible. 

5. We now prove that $P_\lambda:H_{1/2}\rightarrow H_{-1/2}$ is a holomorphic family of Fredholm operators in $\{\cim\lambda>-w^{-2}\}$. It is straightforward to check $P_\lambda$ is an analytic family of operators by
\begin{equation}
\langle P_\lambda u, v\rangle=\langle P u, v\rangle-i\lambda\langle Q^*Qu,v\rangle-\lambda^2\langle u, v\rangle,
\end{equation} 
for any $u, v$ in $H_{1/2}$, in the pairing between $H_{-1/2}$ and $H_{1/2}$. When $\lambda=0$, $P_0=P=P^*$, and $\ker P=\ker P^*$ are finite dimensional. This means $P_{\lambda}$ is Fredholm at $0$. In $\{0\ge \cim\lambda>-w^{-2}\}$, in Step 4 we have shown there exists $\rho_0>\abs{\lambda}$ such that $F_{\lambda,\rho_0}$ is invertible. Now \eqref{3l36} implies that $P_\lambda$, as a compact perturbation of $F_{\lambda,\rho_0}$, is Fredholm. In $\{\cim\lambda>0\}$, $P_\lambda$ is invertible and hence Fredholm. 

5. We conclude $P_{\lambda}^{-1}:H_{-1/2}\rightarrow H_{1/2}$ is a meromorphic family of Fredholm operators with poles of finite rank in $\{\cim\lambda>-w^{-2}\}$. Note that Step 1 implies $P_{i}$ is invertible. Apply \cite[Theorem C.8]{dz19} to conclude that $P_\lambda^{-1}$ is a meromorphic family of Fredholm operators with poles of finite rank. Moreover, the poles cannot accumulate at any interior point at $0$. There is then $\epsilon>0$ such that $P_\lambda$ is invertible on $\{\cim\lambda>0\}\cup\{\abs{\lambda}< \epsilon\}$. 
\end{proof}
 
\begin{lemma}[Semigroup decomposition]\label{semigrouplemma}
Let $Q: H_{1/2}\rightarrow H$. The following are true:
\begin{enumerate}[wide]
\item For $\lambda\in \mathbb{C}$, $\mathcal{A}+i\lambda$ is invertible if $P_\lambda$ is invertible. 
\item When $\ker P\neq \{0\}$, $0$ is an isolated eigenvalue of $\mathcal{A}$. 
\item $\Acd$ generates a contraction semigroup $e^{t\Acd}$ on $\Hcd$. 
\item $\mathcal{A}$ generates a strongly continuous bounded semigroup $e^{t\Ac}$ on $\mathcal{H}$. Moreover, we have the following semigroup decomposition
\begin{equation}
e^{t\Ac}=e^{t\Acd}\Pi_\bullet+\Pi.
\end{equation}
\item $\Acd$ is invertible. For any $\lambda\in\mathbb{C}$, $\lambda\neq 0$, $(\Acd+i\lambda)$ is invertible if and only $(\Ac+i\lambda)$ is invertible. Moreover, there exists $C>0$ such that for all $\lambda$,
\begin{equation}
	C^{-1} \|(\Acd+i\lambda)^{-1}\|_{\Lc(\Hcd)} \leq \nm{(\Ac+i\lambda)^{-1}}_{\Lc(\Hc)} \leq C  (\|(\Acd+i\lambda)^{-1}\|_{\Lc(\Hcd)} + \abs{\lambda}^{-1} ).
\end{equation}
\end{enumerate}
\end{lemma}
\begin{proof}
1. We claim that if $P_\lambda: H_{1/2}\rightarrow H_{-1/2}$ is invertible, then $-i\lambda\notin \spec(\mathcal{A})$ (the converse direction can be found in \cite[Lemma 2.11]{kw22}). If this claim holds then Lemma \ref{lemmainvertibility} implies that $\mathcal{A}$ has an isolated eigenvalue at $0$. To prove the claim, consider for any $(f,g)\in \mathcal{H}=H_{1/2}\times H$, 
\begin{equation}
\begin{pmatrix}
u\\v
\end{pmatrix}=(\mathcal{A}+i\lambda)^{-1}\begin{pmatrix}
f\\g
\end{pmatrix}=\begin{pmatrix}
-i\lambda^{-1}(\id-P_\lambda^{-1}P) & -P_{\lambda}^{-1}\\
P_\lambda^{-1}P& i\lambda P_{\lambda}^{-1}
\end{pmatrix}\begin{pmatrix}
f\\g
\end{pmatrix}
\end{equation}
is an element in $H_{1/2}\times H_{1/2}$ that satisfies
\begin{equation}
(\mathcal{A}+i\lambda)\begin{pmatrix}
u\\v
\end{pmatrix}=\begin{pmatrix}
i\lambda u+v\\
-P u-(Q^*Q-i\lambda)v
\end{pmatrix}=\begin{pmatrix}
f\\g
\end{pmatrix}. 
\end{equation}
Then $Pu+Q^*Qv=-g+i\lambda v\in H$. Thus $(u,v)\in \mathcal{D}$ and $(\mathcal{A}+i\lambda)(u,v)=(f,g)$. Moreover $(\mathcal{A}+i\lambda)^{-1}(0,0)=(0,0)$ since $P_\lambda$ is injective. Thus $\mathcal{A}+i\lambda$ is bijective if $P_\lambda$ is bijective, as claimed. By Lemma \ref{lemmainvertibility}, $P_\lambda$ is invertible in $\{0<\abs{\lambda}<\epsilon\}$. When $\ker P\neq \{0\}$, $\mathcal{A}$ has an isolated eigenvalue at $0$.  

2. Without loss of generality, let $0$ be a isolated eigenvalue of $\mathcal{A}$. Then the Riesz projector $\Pi$ discussed at the beginning of \S \ref{s3-2} is well-defined. Now we show $\Acd:\Dcd\rightarrow \Hcd$ is maximally dissipative and hence generates a contraction semigroup $e^{t\Ac}$ on $\Hcd$. 

2a. We claim $\Acd-\id$ is invertible from $\Dcd$ to $\Hcd$. Note that the part (1) and Lemma \ref{lemmainvertibility} implies $\mathcal{A}-\id$ is invertible from $\mathcal{D}=\Dcd\oplus \ker\Ac$ to $\mathcal{H}=\Hcd\oplus \ker\Ac$. Note $\Ac=0$ on $\ker\Ac$ and $\mathcal{A}-\id$ maps $\ker\Ac$ to $\ker\Ac$ in an invertible manner. Thus $\Ac-\id$ is invertible from $\Dcd$ to $\Hcd$. Eventually observe that $\Acd-\id=\mathcal{A}|_{\Hcd}-\id$ is invertible on $\Hcd$ as claimed. 

2b. We verify that $\Acd$ is dissipative. This is clear from 
\begin{equation}
\cre\langle \Acd(u,v), (u,v)\rangle_{\Hcd}=-\|Qv\|_{Y}^2\le 0. 
\end{equation}
Now $\Acd$ is maximally dissipative, and we invoke \cite[Corollary II.3.20, p.86]{en99} to see $\Dcd$ is dense in $\Hcd$ and $\Acd$ generates a contraction semigroup on $\Hcd$. 

3. We now show $\Ac$ generates a strongly continuous semigroup on $\mathcal{H}$. Note since $\Dcd$ is dense in $\Hcd$, $\mathcal{D}=\Dcd\oplus \ker\Ac$ is dense in $\mathcal{H}=\Hcd\oplus \ker\Ac$ and $\Ac$ is densely defined. From Step 1 we know $\spec\mathcal{A}\neq \mathbb{C}$ and thus $\Ac$ is closed. Now compute
\begin{equation}
\cre\langle \mathcal{A}(u,v), (u,v)\rangle_{\mathcal{H}}=-\|Qv\|_{Y}^2+\cre\langle v,u\rangle_{H}\le \frac{1}{2}\|v\|_{H}^2+\frac{1}{2}\|u\|_{H}^2\le \frac{1}{2}\|(u,v)\|_{\mathcal{H}}^2. 
\end{equation}
We know from the part (1) and Lemma \ref{lemmainvertibility} that $\mathcal{A}-\lambda:\mathcal{D}\rightarrow \mathcal{H}$ is bijective for every $\lambda>\frac{1}{2}$. Moreover
\begin{equation}
\|(\mathcal{A}-\lambda)(u,v)\|_{\mathcal{H}}\|(u,v)\|_{\mathcal{H}}\ge \cre\langle (\lambda-\mathcal{A})(u,v), (u,v)\rangle_{\mathcal{H}}\ge (\lambda-\frac{1}{2})\|(u,v)\|_{\mathcal{H}}^2
\end{equation}
implies that $\|(\mathcal{A}-\lambda)^{-1}\|_{\mathcal{L}(\mathcal{H})}\le (\lambda-\frac{1}{2})^{-1}$ for all $\lambda>\frac{1}{2}$. Apply the Hille--Yoshida theorem as in \cite[Corollary II.3.6, p. 76]{en99} to conclude $\mathcal{A}$ generates a strongly continuous semigroup $e^{t\Ac}$ on $\mathcal{H}$. 

4. We now decompose $e^{t\Ac}$. We begin with the initial decomposition
\begin{equation}
e^{t\Ac}=e^{t\Ac}\Pi_\bullet+e^{t\Ac}\Pi. 
\end{equation}
Consider on $\Hcd$, 
\begin{equation}
\partial_t e^{t\mathcal{A}}|_{\Dcd}=\mathcal{A}|_{\Dcd}=\Acd, \ e^{t\Ac}|_{t=0}=\id.
\end{equation}
Thus $\Acd$ generates both $e^{t\Ac}|_{\Hcd}$ and $e^{t\Acd}$ on $\Hcd$. By the uniqueness of generators we know $e^{t\Ac}|_{\Hcd}=e^{t\Acd}$, and thus $e^{t\Ac}\Pi_\bullet=e^{t\Acd}\Pi_\bullet$. On the other hand, on $\Pi\mathcal{H}=\ker \Ac$, we have
\begin{equation}
\partial_t e^{t\Ac}|_{\ker\Ac}=\Ac|_{\ker\Ac}=0, \ e^{t\Ac}|_{t=0}=\id. 
\end{equation}
Thus $e^{t\Ac}\Pi=\Pi$. The above observation gives desired decomposition. 

5a. Note $\Hcd=\mathcal{A}(\mathcal{D})$. This implies $\mathcal{A}:\Dcd\oplus \ker\Ac\rightarrow \Hcd$ is surjective, and then $\Acd: \Dcd\rightarrow \Hcd$ is bijective. 

5b. Let $\Acd+i\lambda: \Dcd \ra \Hcd$ be invertible. Consider 
\begin{equation}
\Ac+i\lambda = (\Pi_\bullet + \Pi) (\Ac+i\lambda) = (\Acd+i\lambda)\Pi_\bullet + \Pi(\Ac+i\lambda), 
\end{equation}
since $\Pi_\bullet=\Acd \Pi_\bullet$. For any $V \in \Hc$, since $\Acd+i\lambda$ is invertible, there exists $\dot{U} \in \Acd$ such that $(\Acd+i\lambda) \dot{U} = \Pi_\bullet V$. Then 
\begin{equation}
(\Ac+i\lambda)(\dot{U} - i\lambda^{-1} \Pi V) = (\Acd+i\lambda) \dot{U} - \Pi (\Ac+i\lambda) i\lambda^{-1} \Pi V = \Pi_\bullet V + \Pi V = V,
\end{equation}
where we used $\Ac \Pi=0$. Thus $\Ac+i\lambda$ is surjective. To see injectivity, assume $(\Ac+i\lambda)U=0$ for some $U \in \Dc$. Then 
\begin{equation}
0 = \Pi_\bullet (\Ac+i\lambda) U = (\Acd+i\lambda) \Pi_\bullet U.
\end{equation}
Since $\Acd+i\lambda$ is injective, $\Pi_\bullet U=0$ so $U=\Pi U$. Then
\begin{equation}
0 = \Pi (\Ac+i\lambda) U =i\lambda U,
\end{equation}
and since $\lambda \neq 0$, then $U=0$ and $\Ac+i\lambda$ is bijective. 

5c. Assume $\Ac+i\lambda:\Dc\ra \Hc$ is invertible, that is, bijective on $\Dcd\oplus\ker\Ac \ra \Hcd\oplus\ker\Ac$. Note $\Ac=0$ on $\ker\Ac$ and thus $A+i\lambda$ maps $\ker\Ac$ to $\ker\Ac$ bijectively. Then $\Ac+i\lambda$ is bijective from $\Dcd$ to $\Hcd$. Since $\Acd+i\lambda=(\Ac+i\lambda)$ on $\Dcd$, this means $\Acd+i\lambda$ is bijective from $\Dcd$ to $\Hcd$.

5d. Fix $\dot{V} \in \Dcd$ and let $\dot{U}$ be the unique element in $\Hcd$ such that $(\Acd+i\lambda)\dot{U} = \dot{V}$. Since $\dot{V} \in \Dc$ there exists a unique $U \in \Hc$ such that $(\Ac+i\lambda) U = \dot{V}$. Furthermore 
\begin{equation}
(\Acd+i\lambda) \Pi_{\bullet} U = \Pi_\bullet (\Ac+i\lambda) U = \Pi_\bullet \dot{V} = \dot{V}.
\end{equation}
Thus $\dot{U}=\Pi_\bullet U$. And so 
\begin{equation}
\|\dot{U}\|_{\Hcd} \leq \|\dot{U}\|_{\Hc} \leq C \nm{U}_{\Hc} \leq C \nm{(\Ac+i\lambda)^{-1}}_{\Lc(\Hc)} \nm{\dot{V}}_{\Hc} = C \nm{(\Ac+i\lambda)^{-1}}_{\Lc(\Hc)} \nm{\dot{V}}_{\Hcd},
\end{equation}
the last equality of which comes from $\Pi \dot V=0$. Therefore
\begin{equation}
\nm{(\Acd+i\lambda)^{-1}}_{\Lc(\Hcd)} \leq C \nm{(\Ac+i\lambda)^{-1}}_{\Lc(\Hc)}.
\end{equation}

5e. On the other hand, let $(\Ac+i\lambda) U =V$ for $U \in \Dc, V \in \Hc$. Then 
\begin{equation}
(\Acd+i\lambda)\Pi_\bullet U = \Pi_\bullet(\Ac+i\lambda) U = \Pi_\bullet V,
\end{equation}
implies $\nm{\Pi_\bullet U}_{\Hc} \leq \nm{(\Acd+i\lambda)^{-1}}_{\Lc(\Hcd)} \nm{\Pi_\bullet V}_{\Hc}$. Meanwhile
\begin{equation}
\Pi V = \Pi (\Ac+i\lambda) U = i\lambda \Pi U, 
\end{equation}
since $\Pi \Ac=0$. Thus $|\lambda| \nm{\Pi U}_{\Hc} \leq \nm{\Pi V}_{\Hc}$ and 
\begin{equation}
\nm{U}_\Hc \leq C \left(\nm{(\Acd+i\lambda)^{-1}}_{\Lc(\Hcd)} + |\lambda|^{-1}\right) \nm{v}_{\Hc},
\end{equation}
as claimed. 
\end{proof}
\subsection{Semigroup stability}\label{s3-3}
In this paper, we study the stability of the semigroup generated by $\Ac$ via quantitatively studying the resolvent estimates of the operator family $P_\lambda: H_{1/2}\rightarrow H_{-1/2}$ along the real line. As a first step, we prove the equivalence between estimates for $P_\lambda^{-1}$ and $(\mathcal{A}+i\lambda)^{-1}$. 
\begin{lemma}[Resolvent equivalence on the real line]\label{3t5}
Let $\lambda_0>0$ and $K(\lambda)\ge C$ uniformly for $\abs{\lambda}\ge \lambda_0$. The following are equivalent:
\begin{enumerate}[wide]
\item There exists $C_1>0$ such that for all $\lambda\in\mathbb{R}$ with $\abs{\lambda}\ge \lambda_0$ we have
\begin{equation}
\|P_{\lambda}^{-1}\|_{\mathcal{L}(H)}\le C_1 K(\abs{\lambda})/\abs{\lambda}. 
\end{equation}
\item There exists $C_2>0$ such that for all $\lambda\in\mathbb{R}$ with $\abs{\lambda}\ge \lambda_0$ we have
\begin{equation}\label{3l26}
\|(\mathcal{A}+i\lambda)^{-1}\|_{\mathcal{L}(\mathcal{H})}\le C_2K(\abs{\lambda}). 
\end{equation}
\end{enumerate}
\end{lemma}
\begin{remark}
We call estimates of type \eqref{3l26} stability estimates. They are by nature high-frequency estimates that holds uniformly for large $\abs{\lambda}$. The rest of Section \ref{s3-3} aims to explicate the following relations:
\begin{align*}
\text{High Frequencies: }&& \text{Estimates for } P_{\lambda}^{-1} &&\Leftrightarrow & \text{\ \ \ Estimates for }(\mathcal{A}+i\lambda)^{-1},\\
\text{All Frequencies: }&& \text{Decay of } e^{t\mathcal{A}} &&\Leftrightarrow & \text{\ \ \ Estimates for }(\mathcal{A}+i\lambda)^{-1}+\operatorname{UCP}_{P,Q},
\end{align*}
where the unique continuation assumption $\operatorname{UCP}_{P,Q}$ was introduced earlier in Definition \ref{3t6}. The possible failure of $\operatorname{UCP}_{P,Q}$ indicates eigenfunctions of $P$ that cannot be observed via $Q$ and hence cannot be damped by $Q^*Q$. Such eigenfunctions oscillate at frequency $\lambda$ and do not decay, which would prevent a decay rate of $e^{t\Ac}$ from being established. 
\end{remark}

\begin{proof}
Lemma \ref{3t5} follows immediately from the next lemma. 
\end{proof}

\begin{lemma}\label{pencilsemilemma}
Let $\lambda_0>0$ and $\beta\in [0, \frac{1}{2}]$. The following are true uniformly for all $\lambda\in\mathbb{R}$ with $\abs{\lambda}\ge \lambda_0$:
\begin{equation}\label{3l1}
\|P_\lambda^{-1}\|_{H\rightarrow H_{\beta}}\le C\langle \lambda\rangle^{2\beta}\left(\langle \lambda\rangle^{-1} +\|P_{\lambda}^{-1}\|_{\mathcal{L}(H)}\right),
\end{equation}
\begin{equation}\label{3l2}
\|P_{\lambda}^{-1}\|_{H\rightarrow H_{1/2}}\le \|(\Ac+i\lambda)^{-1}\|_{\mathcal{L}(\mathcal{H})}, \ \|P_{\lambda}^{-1}\|_{\mathcal{L}(H)}\le \abs{\lambda}^{-1}\|(\Ac+i\lambda)^{-1}\|_{\mathcal{L}(\mathcal{H})},
\end{equation}
\begin{equation}\label{3l3}
\|(\Ac+i\lambda)^{-1}\|_{\mathcal{L}(\mathcal{H})}\le C\abs{\lambda}\|P_{\lambda}^{-1}\|_{\mathcal{L}(H)}+C\abs{\lambda} \|P_{-\lambda}^{-1}\|_{\mathcal{L}(H)}+C.
\end{equation}
\end{lemma}
\begin{proof}[Proof of Lemma \ref{pencilsemilemma}]
1. For $u \in H_{\frac{1}{2}}$, consider the pairing
\begin{equation}
\langle P_\lambda u, u\rangle=\|P^{\frac{1}{2}}u\|^2-\lambda^2\|u\|^2+i\lambda \|Qu\|^{2}.
\end{equation}
Rearrange the real part to see
\begin{equation}
\|P^{\frac{1}{2}}u\|^2\le C\epsilon^{-1}\|P_{\lambda}u\|_{-\frac{1}{2}}^2+\langle \lambda\rangle^2\|u\|^2+\epsilon \|u\|_{\frac{1}{2}}^2.
\end{equation}
Note $\|u\|_{\frac{1}{2}}^2\le C\|u\|^{2}+C\|P^{\frac{1}{2}}u\|^2$ and we have
\begin{equation}\label{3l4}
\|u\|_{\frac{1}{2}}^2\le C\|P_{\lambda}u\|_{-\frac{1}{2}}^2+C\langle \lambda\rangle^2\|u\|^2.
\end{equation}
after absorption of the $\epsilon$-term. Applying Lemma \ref{interpolation} to this and $\|u\|$ with $\alpha= \<\lambda\>^{2\beta-1}$ gives \eqref{3l1}. 

2. For $u \in H_{\frac{1}{2}}$, note that $(\Ac+i\lambda)(-u, i\lambda u)=(0, P_\lambda u)$. This implies
\begin{equation}
\|u\|_{H_{1/2}}\le \|(\Ac+i\lambda)^{-1}\|_{\mathcal{L}(\mathcal{H})}\|P_\lambda u\|, \ \|u\|\le \abs{\lambda}^{-1}\|(\Ac+i\lambda)^{-1}\|_{\mathcal{L}(\mathcal{H})}\|P_\lambda u\|,
\end{equation}
and thus \eqref{3l2}.

3. Note that $P^*_\lambda=P_{-\lambda}$. Then we have
\begin{equation}
\left\|P^{-1}_\lambda\right\|_{H_{-1/2}\rightarrow H}=\left\|\left(P^{-1}_\lambda\right)^*\right\|_{H\rightarrow H_{1/2}}=\left\|\left(P^*_\lambda\right)^{-1}\right\|_{H\rightarrow H_{1/2}}=\left\|P_{-\lambda}^{-1}\right\|_{H\rightarrow H_{1/2}}. 
\end{equation}
Now for $(f,g) \in \Hc$, let $(u,v) \in \Hc$ be $(u,v)=(\Ac+i\lambda)^{-1}(f,g)$. Then
\begin{equation}
u=-i\lambda^{-1}f+i\lambda^{-1}P_\lambda^{-1}P f-P_\lambda^{-1}g, \ v=P_\lambda^{-1}P f+i\lambda P_\lambda^{-1}g.
\end{equation}
Note
\begin{equation}\label{3l14}
\|P_\lambda^{-1}P f\|\le \|P_{\lambda}^{-1}\|_{H_{-1/2}\rightarrow H}\|f\|_{\frac{1}{2}}=\|P_{-\lambda}^{-1}\|_{H\rightarrow H_{1/2}}\|f\|_{\frac{1}{2}}. 
\end{equation}
Thus
\begin{equation}
\|v\|\le \left(\|P_{-\lambda}^{-1}\|_{H\rightarrow H_{1/2}}+\abs{\lambda}\|P_{\lambda}^{-1}\|_{\mathcal{L}(H)}\right)\|(f,g)\|_{\mathcal{H}}.
\end{equation}
On another hand, apply \eqref{3l4} and \eqref{3l14} to see
\begin{equation}
\|P_\lambda^{-1}P f\|_{\frac{1}{2}}\le C\langle \lambda\rangle \|P_\lambda^{-1}P f\|+C\|P f\|_{-\frac{1}{2}}\le C\left(\langle \lambda\rangle \|P_{-\lambda}^{-1}\|_{H\rightarrow H_{1/2}}+1\right)\|f\|_{\frac{1}{2}}. 
\end{equation}
Then 
\begin{equation}
\|u\|_{\frac{1}{2}}\le \left(\|P_\lambda^{-1}\|_{H\rightarrow H_{1/2}}+C\abs{\lambda}^{-1}\langle \lambda\rangle \|P_{-\lambda}^{-1}\|_{H\rightarrow H_{1/2}}+C\abs{\lambda}^{-1}\right)\|(f,g)\|_{\mathcal{H}}. 
\end{equation}
Bring together the estimates for $u$ and $v$ to see 
\begin{equation}
\|(\Ac+i\lambda)^{-1}\|_{\mathcal{L}(\mathcal{H})}\le \abs{\lambda}\|P_{\lambda}^{-1}\|_{\mathcal{L}(H)}+\|P_\lambda^{-1}\|_{H\rightarrow H_{1/2}}+C \|P_{-\lambda}^{-1}\|_{H\rightarrow H_{1/2}}+C\abs{\lambda}^{-1}.
\end{equation}
Apply \eqref{3l1} to conclude \eqref{3l3}.
\end{proof}

The following lemma shows that using the above framework, resolvent estimates for $\Ac$ give energy decay results for the semigroup.  
\begin{lemma}[Stability plus unique continuation equals decay]\label{3t7}
Let $K(\lambda):[0,\infty)\ra(0,\infty)$ be a function with positive increase. That is, $K$ is a continuous increasing function for which there are $\alpha, c, \lambda_0>0$, such that $K(s\lambda)\ge cs^\alpha K(\lambda)$ uniformly for all $s\ge1$ and $\lambda\ge \lambda_0$. Then $K$ is bijective from $[0,\infty)$ to its range, and the following are equivalent:
\begin{enumerate}[wide]
\item The unique continuation $\operatorname{UCP}_{P,Q}$ holds, and there exists $C,\lambda_0>0$ such that
\begin{equation}
\|P_{\lambda}^{-1} \|_{\mathcal{L}(H)}\le C  K(\abs{\lambda})/\abs{\lambda}
\end{equation}
for all $\lambda\in\mathbb{R}$ with $\abs{\lambda}\ge \lambda_0$.
\item The unique continuation $\operatorname{UCP}_{P,Q}$ holds, and there exists $C,\lambda_0>0$ such that
\begin{equation}\label{3l28}
\|(\mathcal{A}+i\lambda)^{-1}\|_{\mathcal{L}(\mathcal{H})}\le CK(\abs{\lambda})
\end{equation}
for all $\lambda\in\mathbb{R}$ with $\abs{\lambda}\ge \lambda_0$.
\item There exists $C>0$ such that for all $t\ge 0$ we have
\begin{equation}\label{3l27}
\|e^{t\mathcal{A}}\|_{\mathcal{D}\rightarrow\Hcd} \le \frac{C}{K^{-1}(t)}. 
\end{equation}
In the limit case when $K(\abs{\lambda})=C$, the equivalence still holds with \eqref{3l27} replaced with
\begin{equation}
\|e^{t\mathcal{A}}\|_{\mathcal{H}\rightarrow\Hcd} \le Ce^{-t/C}.
\end{equation}
\end{enumerate}
\end{lemma}
Note that when $K(\lambda)=\lambda^{\alpha}$, $K^{-1}(t) = t^{1/\alpha}$, which recovers the rate of \cite{bt10}. Similarly when $K(\lambda)=e^{c\lambda}$, $K^{-1}(t) = c^{-1}\log(t)$, which recovers the rate of \cite{bur98}.

In semigroup theory, it is standard to utilise stability theorems under the assumption that $\spec(\Acd)\cap i\mathbb{R}=\emptyset$. Here we instead use the unique continuation assumption, since the two are equivalent.
\begin{lemma}\label{3t8}
$\operatorname{UCP}_{P,Q}$ holds if and only if $\spec(\Acd)\cap i\mathbb{R}=\emptyset$. In such case, for every $\lambda_0>0$ we have $C>0,$ such that $\|(\Acd+i\lambda)^{-1}\|_{\mathcal{L}(\mathcal{H})}\le C$ uniformly for $\lambda\in[-\lambda_0, \lambda_0]$. 
\end{lemma}
\begin{remark}
It is somehow surprising that, this equivalence holds given that $Q^*Q$ is not relatively compact and the spectrum of $\Acd$ has not been shown discrete. Despite this, the unique continuation principle is enough to ensure there is no spectrum on the imaginary axis.
\end{remark}
We first prove Lemma \ref{3t7} assuming Lemma \ref{3t8} and then prove Lemma \ref{3t8}.
\begin{proof}[Proof of Lemma \ref{3t7} via Lemma \ref{3t8}]
1. (1)$\Leftrightarrow$(2) follows from Lemma \ref{3t5}. 

2. (2)$\Leftrightarrow$(3): from \eqref{3l29}, we know \eqref{3l27} is equivalent to
\begin{equation}\label{3l30}
\|e^{t\Acd}\|_{\Dcd\rightarrow\Hcd} \le \frac{C}{K^{-1}(t)}. 
\end{equation}
By Lemma \ref{semigrouplemma} part (iv) and the abstract semigroup decay results \cite[Theorem 1.1]{rss19} (see also \cite{bur98}): \eqref{3l28} is equivalent to \eqref{3l30} (hence \eqref{3l27}), if $\spec(\Acd)\cap i\mathbb{R}=\emptyset$, that is, if $\operatorname{UCP}_{P,Q}$ holds by Lemma \ref{3t8}. This observation implies (2)$\Rightarrow$(3). Now assume $\operatorname{UCP}_{P,Q}$ does not hold, such that there is $u\in H_1$, real $\lambda\neq 0$ such that
\begin{equation}
(P-\lambda^2)u=0, \ Qu=0. 
\end{equation}
Then we have
\begin{equation}
\partial_t e^{t\Ac}(u,i\lambda u)=\Ac (u,i\lambda u)=i\lambda (u,i\lambda u), \ e^{t\Ac}(u,i\lambda u)=e^{i\lambda t}(u,i\lambda u). 
\end{equation}
Notably, 
\begin{equation}
\|e^{t\Ac}(u,i\lambda u)\|^2_{\Hcd}=\|P^{\frac{1}{2}}u\|_H^2+\lambda^2\|u\|_{H}^2
\end{equation}
is constant in time. This contradicts the assumption that the right hand side of \eqref{3l27} decays to $0$ in time. Thus we have (2)$\Leftarrow$(3) as well. 

3. In the limit case $K(\lambda)=C$, we use \cite{gea78,pru84,hua85} instead of \cite{rss19}, noting that the required boundedness of $\|(\Acd+i\lambda)^{-1}\|_{\Lc(\Hc)}$ for $\abs{\lambda} \leq \lambda_0$ is guaranteed by Lemma \ref{3t8}.
\end{proof}

\begin{proof}[Proof of Lemma \ref{3t8}]
1. By Lemma \ref{semigrouplemma}, we know $\Acd$ is invertible, and $\spec(\Acd)=\spec(\Ac)\setminus\{0\}$. By Lemma \ref{semigrouplemma}(1), $\spec(\Ac)\cap i\mathbb{R}\setminus\{0\}=\{\lambda\in \mathbb{R}_{\neq 0}: P_\lambda \text{ is not invertible}\}$. Thus
\begin{equation}
\spec(\Acd)\cap i\mathbb{R}=\{\lambda\in \mathbb{R}_{\neq 0}: P_\lambda \text{ is not invertible}\}.
\end{equation}
It suffices to show $\operatorname{UCP}_{P,Q}$ holds if and only if $P_\lambda$ is invertible for all $\lambda\neq 0$.

2. Assume $P_\lambda$ is invertible for all $\lambda\neq 0$. Assume there is nontrivial $u\in H_{1/2}$ and $\lambda>0$ such that
\begin{equation}
(P-\lambda^2)u=0, \ Qu=0.
\end{equation}
Then $P_\lambda u=(P-i\lambda Q^*Q -\lambda^2)u=0$ and $u$ has to be $0$. Thus $\operatorname{UCP}_{P,Q}$ holds. 

3. Assume $\operatorname{UCP}_{P,Q}$ holds. For each $\lambda>0$, we claim the quantitative estimate
\begin{equation}\label{3l31}
\|u\|_{\frac{1}{2}}\le C_\lambda(\|(P-\lambda^2)u\|_{-\frac{1}{2}}+C_\lambda\|Qu\|_{Y})
\end{equation}
holds with $C_\lambda>0$. Assume not, there are $u_n\in H_{1/2}$ with $(P-\lambda^2)u_n=f_n$ such that
\begin{equation}
\|u_n\|_{\frac{1}{2}}\equiv 1, \ \|f_n\|_{-\frac{1}{2}}=o(1), \ \|Q u_n\|_{Y}=o(1). 
\end{equation}
Note $P-i$ is invertible and apply its inverse to $f_n$ to observe
\begin{equation}
(P-i)^{-1}(P-\lambda^2)u_n=(P-i)^{-1} f_n,
\end{equation}
which after simplification becomes
\begin{equation}
u_n=w_n+(P-i)^{-1}f_n, \ w_n=(\lambda^2-i)(P-i)^{-1}u_n. 
\end{equation}
Note $w_n$ are bounded in $H_{3/2}$ and hence converges in $H_{1/2}$ to some $w\in H_{3/2}$. Also $(P-i)^{-1}f_n\rightarrow 0$ in $H_{1/2}$. Thus $u_n\rightarrow w$ in $H_{1/2}$. This implies
\begin{equation}
(P-\lambda^2)w=0, \ Qw=0,
\end{equation}
which contradicts $\operatorname{UCP}_{P,Q}$. Thus we have \eqref{3l31}. 

4. We now show $P_\lambda: H_{1/2}\rightarrow H_{-1/2}$ is invertible for $\lambda\neq 0$. We firstly show it is injective. Assume not, then there are $u_n\in H_{1/2}$, $P_\lambda u_n=(P-i\lambda Q^*Q-\lambda^2)u_n=f_n$ such that
\begin{equation}
\|u_n\|_{\frac{1}{2}}\equiv 1, \ \|f_n\|_{-\frac{1}{2}}=o(1).
\end{equation}
Consider
\begin{equation}
o(1)=\cim\langle f_n, u_n\rangle=\lambda \|Q u_n\|_Y^2 
\end{equation}
implies $\|Qu_n\|_Y\rightarrow 0$. Then $i\lambda Q^*Q u_n\rightarrow 0$ in $H_{-1/2}$ and $(P-\lambda^2) u_n\rightarrow 0$ in $H_{-1/2}$. This contradicts \eqref{3l31}. Thus $P_\lambda$ is injective for $\lambda\neq 0$. But this implies $P_\lambda^*=P_{-\lambda}$ is also injective, and thus $P_\lambda$ is invertible. 

5. Fix $\lambda_0>0$, assume $\|(\Acd+i\lambda)^{-1}\|_{\mathcal{L}(\Hc)}$ is not bounded on $\lambda\in [-\lambda_0, \lambda_0]$. Since $[-\lambda_0, \lambda_0]$ is compact, there are $\lambda_n\rightarrow \lambda$, $\abs{\lambda_n-\lambda}\le 1/n$ such that $\|(\Acd+i\lambda_n)^{-1}\|_{\mathcal{L}(\Hc)}\ge 1/n$. Now by the resolvent identity
\begin{equation}
(\Acd+i\lambda_n)^{-1}=(\Acd+i\lambda)^{-1}+i(\lambda-\lambda_n)(\Acd+i\lambda)^{-1}(\Acd+i\lambda_n)^{-1},
\end{equation}
Note $\|(\Acd+i\lambda)^{-1}\|_{\mathcal{L}(\Hc)}\le C$ and we have
\begin{equation}
\|(\Acd+i\lambda_n)^{-1}\|_{\mathcal{L}(\Hc)}\le C+\frac{C}{n}\|(\Acd+i\lambda_n)^{-1}\|_{\mathcal{L}(\Hc)}.
\end{equation}
For large $n$, we have $\|(\Acd+i\lambda_n)^{-1}\|_{\mathcal{L}(\Hc)}\le 2C$ uniformly. This contradicts the unbounded hypothesis and we have the uniform boundedness. 
\end{proof}

\subsection{Theorem \ref{thmcontrol}: control and stability}\label{s3-4}
By Lemma \ref{3t7}, to understand the energy decay of $e^{t\Ac}$, it suffices to find high-frequency resolvent estimates for $P_{\lambda}^{-1}$. In \S \ref{s3-4}, we will prove Propositions \ref{3t1}, \ref{3t9}, turning control estimates \eqref{1l1} into resolvent estimates and vice versa, and hence prove Theorem \ref{thmcontrol}. 
\begin{proposition}[Control to resolvent]\label{3t1}
Let $Q\in \mathcal{L}(H_{\gamma}, Y)$ for $\gamma\in [0,\frac{1}{2}]$. Fix $\mu\ge 0$. Assume the control estimate \eqref{1l1} holds, that is there exist $\lambda_0\ge 0$, $\delta\in[0,\frac{1}{2}]$ such that 
\begin{equation}\label{3l5}
\|u\|_{H}\le M(\lambda)\langle \lambda\rangle^{-1}\|(P-\lambda^2)\Lambda^{\mu-\gamma} u\|+m(\lambda)\|Q\Lambda^{\mu} u\|_{Y}+o(\lambda^{2\delta})\|u\|_{-\delta}
\end{equation}
uniformly for any $u\in H_{1+\mu-\gamma}$ and all $\lambda\ge \lambda_0$. Assume $M(\lambda)\ge C$, $m(\lambda)\ge C$ for all $\lambda\ge\lambda_0$. Then there are $C,\lambda_0'>0$ such that
\begin{equation}\label{3l9}
\|P_{\lambda}^{-1}\|_{\mathcal{L}(H)}\le C(M^2+m^2)\abs{\lambda}^{-1+4\mu}.
\end{equation} 
uniformly for all $\abs{\lambda}\ge \lambda_0'$. 
\end{proposition}

\begin{remark}\label{3t3}
When $\delta>\frac{1}{2}$, \eqref{3l9} still holds with $M^2$ replaced by $(M+\lambda^{\lceil 2\delta \rceil-1})^2$. 
\end{remark}
\begin{proof}
The idea of the proof is the following: considering the $H$-unboundedness of $Q^*\in \mathcal{L}(Y, H_{-\gamma})$, we could only apply the Minkowski inequality at the $H_{-\gamma}$-regularity to $P-\lambda^2=P_{\lambda} + i \lambda Q^* Q$, instead of the $H$-regularity. We use the interpolation inequality to remove the $o(\lambda^{2\delta})$-size error and restore the $H$-regularity at the cost of some $\lambda$.

1. Let $h=\lambda^{-1}$. Apply \eqref{3l5} to $v=\Lambda^{-\mu} u$ to observe
\begin{equation}\label{3l7}
\|u\|_{-\mu}\le M h^{-1}\|(h^2P-1)u\|_{-\gamma}+m\|Qu\|_Y+o(h^{-2\delta})\|u\|_{-\mu-\delta},
\end{equation}
uniformly for any $u\in H_{1-\gamma}$. 

2. We first get rid of the $o(h^{-2\delta})$-term in \eqref{3l7}. For any $s\in \mathbb{R}$, pair $(h^2P-1)u$ with $u$ to see
\begin{equation}
h^2\|P^{\frac{1}{2}}u\|_s^2-\|u\|_s^2=\langle (h^2P-1)u, u\rangle_s.
\end{equation}
It further reduces to 
\begin{equation}\label{3l20}
\|u\|_{s}\le C\|(h^2P-1)u\|_{s}+Ch\|u\|_{s+\frac{1}{2}}.
\end{equation}
Iterative use of \eqref{3l20} gives
\begin{equation}
\|u\|_{s}\le C\|(h^2P-1)u\|_{s+\frac{k-1}{2}}+Ch^k\|u\|_{s+\frac{k}{2}}.
\end{equation}
Let $k=\lceil 2\delta \rceil$ be the least nonnegative integer that bounds $2\delta$ from above. Taking $s$ as $-\mu-\frac{k-1}{2}$ and $-\mu-\frac{k}{2}$ we have
\begin{gather}
\|u\|_{-\mu-\frac{k-1}{2}}\le C\left(h^{-1}\|(h^2P-1)u\|_{-\mu-\frac{1}{2}}+h^{k-1}\|u\|_{-\mu}\right),\\
\|u\|_{-\mu-\frac{k}{2}}\le Ch\left(h^{-1}\|(h^2P-1)u\|_{-\mu-\frac{1}{2}}+h^{k-1}\|u\|_{-\mu}\right).
\end{gather}
Note $-\mu-\frac{k}{2}\le-\mu-\delta<-\mu-\frac{k-1}{2}$ and apply the interpolation inequality \eqref{3l8} with $\alpha=h^{2\delta-k+1}$ to obtain
\begin{equation}
\|u\|_{-\mu-\delta}\le h^{2\delta-\lceil 2\delta \rceil}\|(h^2P-1)u\|_{-\mu-\frac{1}{2}}+h^{2\delta}\|u\|_{-\mu}.
\end{equation}
Since $M\ge C^{-1}h$, $\delta \le \frac{1}{2}$ and $\mu\ge 0$, \eqref{3l7} reduces to
\begin{equation}\label{3l18}
\|u\|_{-\mu}\le CM h^{-1}\|(h^2P-1)u\|_{-\gamma}+m\|Qu\|_Y.
\end{equation}
When $\delta>\frac{1}{2}$, Remark \ref{3t3} follows from replacing $ M $ by $M+h^{1-\lceil 2\delta \rceil}$ in \eqref{3l18}. 

3. We then reinforce the regularity on the left of \eqref{3l8} to $H_0$. Note for any $s\in\mathbb{R}$, 
\begin{equation}
\|u\|_{s+1}=\|(P+1)u\|_{s}\le C h^{-2}\left(\|(h^2P-1)u\|_{s}+\|u\|_{s}\right), \\
\end{equation}
uniformly for $h$ small. For each $k\in\mathbb{N}\cup\{0\}$ such that $-\mu+k\le -\gamma$, iteratively we have from \eqref{3l18} that
\begin{equation}
\|u\|_{-\mu+k+1}\le Ch^{-2k-2}\left(M h^{-1}\|(h^2P-1)u\|_{-\gamma}+m\|Qu\|_Y\right).
\end{equation}
Let $k$ be smallest integer such that $-\mu\le 0<-\mu+k+1$. Apply the interpolation inequality $\eqref{3l8}$ with $\alpha=h^{2(-\mu+k+1)}$ to see
\begin{equation}\label{3l19}
\|u\|_0\le Ch^{-2\mu}\left(M h^{-1}\|(h^2P-1)u\|_{-\gamma}+m\|Qu\|_Y\right),
\end{equation}
uniformly for any $u\in H_{1-\gamma}$.

4. Let $f\in H$. There exists a unique $u\in H$ such that $h^2P_\lambda u=(h^2P-ihQ^*Q-1)u=f$: see further details in Proposition 2.5 of \cite{kw22} and Section 2.2 of \cite{cpsst19}. Pair $f$ with $u$ in $H_{0}$ to observe
\begin{equation}\label{3l6}
h^2\|P^{\frac{1}{2}}u\|_{0}^2-\|u\|_{0}^2-ih\|Qu\|_Y^2 =\langle f, u\rangle_{0}. 
\end{equation}
The imaginary part of \eqref{3l6} implies
\begin{equation}
\|Qu\|_{Y}^2=h^{-1}\cim\langle f, u\rangle<\infty. 
\end{equation}
This implies that $h^2P u=f+u+ihQ^*Q u\in H_{-\gamma}$, and further implies $u\in H_{1-\gamma}$. Note further
\begin{equation}
\|(h^2P-1)u\|_{-\gamma}\le\|f\|_{-\gamma}+ih\|Q^*Qu\|_{-\gamma}\le\|f\|_{-\gamma}+ih\|Qu\|_{Y}. 
\end{equation}
Apply \eqref{3l19} to see
\begin{equation}
\|u\|\le C\left(Mh^{-1-2\mu}\|f\|_{-\gamma}+(M+m)h^{-2\mu}\|Q u\|_Y\right).
\end{equation}
The imaginary part of \eqref{3l6} further gives
\begin{equation}
\|Qu\|_{Y}^2\le h^{-1}\abs{\langle f, u\rangle}\le C\epsilon^{-1}h^{-2}(M+m)^2h^{-4\mu}\|f\|^{2}+\epsilon(M+m)^{-2}h^{4\mu}\|u\|^2
\end{equation}
and thus
\begin{equation}
\|u\|\le CMh^{-1-2\mu}\|f\|_{-\gamma}+C\epsilon^{-1/2}h^{-1}(M+m)^2h^{-4\mu}\|f\|+\epsilon^{1/2}\|u\|.
\end{equation}
After absorption of the $\epsilon$-term we have
\begin{equation}
\|u\|\le C(M+m)^2h^{-1-4\mu}\|f\|.
\end{equation}
Recall $h^2 P_{\lambda} u = f$ and $\lambda=h^{-1}$: for large $\abs{\lambda}$ this implies
\begin{equation}
\|P_\lambda^{-1}\|\le C(M^2+m^2)\abs{\lambda}^{-1+4\mu},
\end{equation}
as desired.
\end{proof}

\begin{proposition}[Resolvent to control]\label{3t9}
Let $Q\in \mathcal{L}(H_{\gamma}, Y)$ for $\gamma\in [0,\frac{1}{2}]$. 
Assume there exists $\lambda_0\ge0,$ such that the resolvent estimates
\begin{equation}
\|P_{\lambda}^{-1}\|_{\mathcal{L}(H)}\le K(\abs{\lambda})\langle \lambda\rangle^{-1},
\end{equation} 
hold for all $\abs{\lambda}\ge \lambda_0$, where $K(\abs{\lambda})\ge C'$ for some $C'>0$. Then for any $\mu\in[0, \frac{1}{2}+\gamma]$, the control estimate \eqref{1l1} holds for $P$ and $Q$, that there is $C\ge0$ such that 
\begin{equation}
\|u\|_{H}\le C\langle \lambda\rangle^{2(\gamma-\mu)}K(\abs{\lambda}) \left(\langle \lambda\rangle^{-1}\|(P-\lambda^2)\Lambda^{\mu-\gamma}u\|_{H}+ \|Q\Lambda^{\mu} u\|_Y\right),
\end{equation}
uniformly for all $\lambda\ge \lambda_0$ and $u\in H_{\frac{1}{2}+\mu}$. 
\end{proposition}
\begin{proof}
Apply \eqref{3l1} to see 
\begin{equation}
\|P_\lambda^{-1}\|_{H_{-\gamma}\rightarrow H}=\|P_{-\lambda}^{-1}\|_{H\rightarrow H_{\gamma}}\le C\langle \lambda\rangle^{2\gamma}\|P_{\lambda}^{-1}\|_{\mathcal{L}(H)}=C\langle \lambda\rangle^{2\gamma-1}K(\lambda).
\end{equation}
For any $v\in H_\frac{1}{2}$, we have
\begin{equation}\label{3l23}
\|v\|_H\le C\langle \lambda\rangle^{2\gamma-1}K \|(P-i\lambda Q^*Q-\lambda^2)v\|_{-\gamma}\le C\langle \lambda\rangle^{2\gamma}K \left(\langle \lambda\rangle^{-1}\|(P-\lambda^2)v\|_{-\gamma}+\|Qv\|_Y\right). 
\end{equation}
For each $s\le -\frac{1}{2}$, pair $(P-\lambda^2)v$ with $v$ in $H_s$ to obtain
\begin{equation}
\|P^\frac{1}{2}v\|_{s}^2-\lambda^2 \|v\|_{s}^2=\langle (P-\lambda^2)v, v\rangle_s, 
\end{equation}
and estimate
\begin{equation}\label{3l21}
\|v\|_{s}\le C\langle \lambda\rangle^{-1}\left(\langle \lambda\rangle^{-1}\|(P-\lambda^2)v\|_{s}+\|v\|_{s+\frac{1}{2}}\right).
\end{equation}
Iteratively we have 
\begin{equation}
\|v\|_{-\frac{k}{2}}\le C\langle \lambda\rangle^{-2}\|(P-\lambda^2)v\|_{-\frac{1}{2}}+C\langle \lambda\rangle^{-k}\|v\|_{H}.
\end{equation}
for each positive integer $k$. We fix $k$ to be the largest integer such that $k\le 1+2\gamma$; note that in general we can take $k$ even larger, as long as $\langle \lambda\rangle^{2\gamma-k+2}K(\lambda)$ is asymptotically bounded from below. Then \eqref{3l23} implies
\begin{equation}
\|v\|_{-\frac{k}{2}}\le C \langle \lambda\rangle^{2\gamma-k}K \left(\langle \lambda\rangle^{-1}\|(P-\lambda^2)v\|_{-\gamma}+\|Qv\|_Y\right).
\end{equation}
Use the interpolation inequality \eqref{3l8} with $\|v\|_0$ and $\alpha=\langle \lambda\rangle^{-2\mu}$ gives
\begin{equation}\label{3l22}
\|v\|_{-\mu}\le C\langle \lambda\rangle^{2(\gamma-\mu)}K \left(\langle \lambda\rangle^{-1}\|(P-\lambda^2)v\|_{-\gamma}+\|Qv\|_Y\right),
\end{equation}
for each $\mu\in[0, \frac{k}{2}]$. For each $\mu\in[\frac{k}{2}, \frac{1}{2}+\gamma]$, apply \eqref{3l21} to obtain \eqref{3l22}. Thus \eqref{3l22} holds for all $\mu\in[0, \frac{1}{2}+\gamma]$. For each $u\in H_{\frac{1}{2}+\mu}$, let $v=\Lambda^{\mu}u$ and we have
\begin{equation}
\|u\|_{H}\le C\langle \lambda\rangle^{2(\gamma-\mu)}K \left(\langle \lambda\rangle^{-1}\|(P-\lambda^2)\Lambda^{\mu-\gamma}u\|_{H}+\|Q\Lambda^{\mu}u\|_Y\right),
\end{equation}
as we desired. 
\end{proof}
\begin{proof}[Proof of Theorem \ref{thmcontrol}]
It follows from Propositions \ref{3t1}, \ref{3t9} and Lemma \ref{3t5}. 
\end{proof}

\subsection{Theorem \ref{thmbackward}: optimal backward uniqueness}\label{s3-5}
To show the semigroup $e^{t\Ac}$ is backward unique, we need to show along a pair of radial rays symmetric about the imaginary axis that goes into the lower half plane, one can obtain some uniform estimates over the resolvent. This is characterised in the next proposition. 
\begin{proposition}[Rays into deep spectra]
Assume $Q\in \mathcal{L}(H_{1/4}, Y)$. Then there exists $\theta_0\in (0, \pi)$ small such that along any ray $\lambda=\abs{\lambda}e^{i\theta}$ with $\abs{\theta-\frac{3}{2}\pi}\in (0,\theta_0)$, there are $\lambda_0, C>0,$ such that for any $\lambda\in\mathbb{R}$ with $\abs{\lambda}\ge \lambda_0$
\begin{equation}
\|P^{\frac{1}{2}} u\|^{2}+\abs{\lambda}^2\|u\|^2\le C\abs{\lambda}^{-2}\|(P-i\lambda Q^*Q-\lambda^2)u\|^2,
\end{equation}
which implies that along those rays, 
\begin{equation}\label{3l16}
\|P_{\lambda}^{-1}\|_{\mathcal{L}(H)}\le C\abs{\lambda}^{-1}.
\end{equation} 

\end{proposition}
\begin{remark}
In \cite{kw22}, the authors showed that for $Q\in\mathcal{L}(H_{\gamma}, Y)$ for $\gamma<\frac{1}{2}$, along arbitrary pair of radial rays symmetric about the imaginary axis that goes into the lower half plane, one can obtain estimates \eqref{3l16}. In the limit case $\gamma=\frac{1}{2}$, we here show that only along pairs of radial rays that are descending fast enough in $\cim\lambda$ (compared to arbitrary ones when $\gamma=\frac{1}{2}$), such estimates are obtainable. 
\end{remark}
\begin{proof}
Fix $\theta$ such that $\abs{\theta-\frac{3}{2}\pi}\in (0,\theta_0)$, where $\theta_0>0$ will be determined later. Parametrise $\lambda=\abs{\lambda}e^{i\theta}$ by $\lambda=h^{-1}(\delta-i)$, where $h=\langle \delta\rangle\abs{\lambda}^{-1}\rightarrow 0$. Here $\delta=\cot\theta\in (-\delta_0, \delta_0)$, where $\delta_0=-\cot(\frac{3}{2}\pi+\theta_0)$. Let
\begin{equation}
h^2P_\lambda u=(h^2P+(1-\delta^2)-hQ^*Q+i(2-h\delta Q^*Q))u=f. 
\end{equation}
Pair it with $u$ in $H$ to observe
\begin{equation}\label{3l13}
h^2\|P^{\frac{1}{2}}u\|^2+(1-\delta^2)\|u\|^2-h\|Q u\|^2+i(2\|u\|^2-h\delta \|Q u\|^2)=\langle f, u\rangle. 
\end{equation}
Now interpolate
\begin{equation}
\|Qu\|\le C\|u\|_{\frac{1}{4}}\le \epsilon \abs{\delta}^{-\frac{1}{2}}h^{-\frac{1}{2}}\|u\|+C_\epsilon\abs{\delta}^{\frac{1}{2}}h^{\frac{1}{2}}\|u\|_{\frac{1}{2}}. 
\end{equation}
The imaginary part of \eqref{3l13} implies
\begin{multline}
2\|u\|^2\le h\abs{\delta}\|Q u\|^2+\epsilon \|u\|^2+C\epsilon^{-1}\|f\|^2\le 2\epsilon\|u\|^2+C_\epsilon\left(\delta^{2}h^2 \|u\|_{\frac{1}{2}}^2+\|f\|^2\right)\\
\le (2\epsilon+C_\epsilon\delta^2 h^2) \|u\|^2+C_\epsilon\left(\delta^2\|hP^{\frac{1}{2}}u\|^2+\|f\|^2\right),
\end{multline}
where we used the fact that $\|\Lambda^{\frac{1}{2}}u\|\le \|u\|+\|P^{\frac{1}{2}}u\|$. Now without loss of generality we assume $\abs{\delta}<1$, and we fix $\epsilon<\frac{1}{8}$ and there exists $h_1>0$ such that $2\epsilon+C_\epsilon \delta^2h_1^2<\frac{1}{2}$. Then for any $h<h_1$, we can absorb the $\|u\|^2$-term on the right by the left, and we arrive at
\begin{equation}
\|u\|^2\le C\delta^2h^2\|P^{\frac{1}{2}}u\|^2+C\|f\|^2. 
\end{equation}
The imaginary part of \eqref{3l13} also implies
\begin{equation}
h\|Q u\|^2\le C\abs{\delta}h^2\|P^{\frac{1}{2}}u\|^2+C\abs{\delta}^{-1}\|f\|^2.
\end{equation}
Substitute the term back into the real part of \eqref{3l13} to observe
\begin{equation}\label{3l15}
h^2\|P^{\frac{1}{2}}u\|^2+\|u\|^2\le C\|f\|^2+(1+\delta^2)\|u\|^2+h\|Q u\|^2\le C(\abs{\delta}+\delta^2)h^2\|P^{\frac{1}{2}}u\|^2+C(1+\abs{\delta}^{-1})\|f\|^2. 
\end{equation}
Note that the constants here do not depend on $\delta$. Hence we can choose $\delta_0<1$ small such that $C(\delta_0+\delta_0^2)\le \frac{1}{2}$. Then for any $\abs{\delta}\in(0, \delta_0)$, the first term on the right of \eqref{3l15} can be absorbed by the left, and gives
\begin{equation}
h^2\|P^{\frac{1}{2}}u\|^2+\|u\|^2\le C(1+\abs{\delta}^{-1})\|f\|^2. 
\end{equation}
Since $h^2P_{\lambda}u =f$,
\begin{equation}
\|u\|^2+\langle \delta\rangle^2\abs{\lambda}^{-2}\|P^{\frac{1}{2}} u\|^2\le C\langle \delta\rangle^4\abs{\lambda}^{-4}(1+\abs{\delta}^{-1})\| P_{\lambda} u \|^2,
\end{equation}
which further reduces to
\begin{equation}
\|u\|^2+\abs{\lambda}^{-2}\|P^{\frac{1}{2}} u\|^2\le C\abs{\lambda}^{-4}(1+\abs{\tan\theta})\|P_{\lambda} u\|^2,
\end{equation}
for any $\theta$ with $\abs{\theta-\frac{3}{2}\pi}\in (0, \theta_0)$, where $\theta_0=\frac{\pi}{2}-\cot^{-1}(\delta_0)$. This is what we want. 
\end{proof}
\begin{proof}[Proof of Theorem \ref{thmbackward}]
From \eqref{3l3} and \eqref{3l16} we know $\|(\Ac+i\lambda)^{-1}\|_{\mathcal{L}(\mathcal{H})}\le C$ uniformly along along any ray $\lambda=\abs{\lambda}e^{i\theta}$ with $\abs{\theta-\frac{3}{2}\pi}\in (0,\theta_0)$ and $\abs{\lambda}\ge \lambda_0$. Apply Proposition 2.16 of \cite{kw22} (or originally Theorem 3.1 of \cite{lrt01}) to conclude the backward uniqueness. 
\end{proof}

\subsection{Theorem \ref{thmdilate}: dilation and contraction}\label{s3-6}
We now prove a stronger version of Theorem \ref{thmdilate}. 
\begin{proposition}[Dilation and contraction for functional calculus]\label{3t4}
Let $\gamma\in[0, \frac{1}{2}]$, $\mu\ge 0$. Assume the control estimate \eqref{1l1} holds for $P$ and $Q$, that there is $\delta>0$ and for $\lambda>\lambda_0$, 
\begin{equation}
\|u\|_{H}\le M(\lambda)\lambda^{-1}\|(P-\lambda^2)\Lambda^{\mu-\gamma} u\|_{H}+m(\lambda)\|Q\Lambda^{\mu} u\|_{Y}+e(\lambda)\|\Lambda^{-\delta} u\|_{H},
\end{equation}
uniformly for $u\in H_{1+\mu-\gamma}$. Fix positive $\fpow\ge 2\gamma$. Let $f$ be a differentiable function bijectively mapping from $[0,\infty)$ to a subset of $[0,\infty)$, and assume there exist $K, \rho_0>0$ such that
\begin{equation}
f'(s)\ge K^{-1}s^{\alpha-1}
\end{equation}
 for all $s\ge \rho_0^2$. Then there exists $C, \lambda_0'$ such that the control estimate holds for $f(P)$ and $Q$. That is for $\lambda>\lambda_0'$, 
\begin{multline}
\|u\|\le C(M(\tilde\lambda)\tilde\lambda^{1-2\fpow}+m(\tilde\lambda)\tilde\lambda^{-2+2(2\gamma+1-\fpow)_+})\|(f(P)-\lambda^2)\Lambda^{\mu-\gamma} u\|_H\\
+Cm(\tilde\lambda)\|Q\Lambda^{\mu}u\|_Y+Ce(\tilde\lambda)\|\Lambda^{-\delta}u\|_H,
\end{multline}
uniformly for all $u\in H_{\alpha+\mu-\gamma}$, where $(2\gamma+1-\fpow)_+=\max\{2\gamma+1-\fpow, 0\}$. Here $\tilde\lambda=(f^{-1}(\lambda^2))^{\frac{1}{2}}$, $M(\tilde\lambda)=M({(f^{-1}(\lambda^2))^{\frac{1}{2}}})$ and $m(\tilde\lambda)$ is understood similarly. Here we define
\begin{equation}
f(P)=\int_0^\infty f(\rho^{2})\ dE_\rho.
\end{equation}
\end{proposition}
\begin{remark}
The proof works for a more general class of functions $f$, where $f$ is a continuous function bijectively mapping from $[0,\infty)$ to a subset of $[0,\infty)$, with the assumption that there exist $K, \rho_0>0$ and $\alpha\ge 2\gamma$ such that
\begin{equation}
\abs{f(s)-f(t)}\ge K^{-1}\min\{s^{\alpha-1}, t^{\alpha-1}\}\abs{s-t}. 
\end{equation}
for all $\rho_0^2\le s<t$, where $K$ does not depend on $s,t$. 
\end{remark}
\begin{figure}
\includegraphics[page=1]{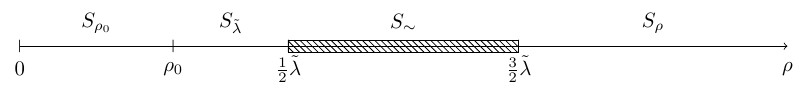}
\caption{Illustration of the spectral decomposition in Proof of Proposition \ref{3t4}. We will estimate $S_{\rho_0}$, $S_{\tilde\lambda}$, $S_\rho$, and apply the control estimate \ref{1l1} to the shaded regime $S_\sim$ only.}\label{f1}
\end{figure}
\begin{proof}
1. Without loss of generality, let $\lambda_0'\ge 2\rho_0$. Consider the spectral decomposition
\begin{equation}
[0, \infty)_\rho=[0, \rho_0)\cup \left[\rho_0, \frac{1}{2}\tilde\lambda\right)\cup\left[\frac{1}{2}\tilde\lambda, \frac{3}{2}\tilde\lambda\right]\cup \left(\frac{3}{2}\tilde\lambda,\infty\right)=S_{\rho_0}\cup S_{\tilde\lambda}\cup S_\sim\cup S_{\rho}.
\end{equation}
The spectral regimes $S_{\tilde\lambda}$ and $S_{\rho}$ indicate where $\tilde\lambda$ and $\rho$ are larger than one another, respectively; the regime $S_\sim$ indicates they are comparable. Let the spectral projector onto one of these sets $S_*$ be
\begin{equation}
\Pi_*=\int_{S_*}~dE_\rho(u), \ *\in \{\rho_0, \tilde\lambda, \sim, \rho\}.
\end{equation}
See Figure \ref{f1} for illustration of the spectral decomposition. 

2. Regime $S_{\rho_0}$: $0\le \rho<\rho_0$. Here we mainly use the ellipticity of $f(\rho^2)-\lambda^2$. In this step we will be able to fix $\lambda_0'$. We first estimate $\Pi_{\rho_0} \Lambda^{\gamma+\mu}u$. From the assumption we know $f$ maps $[0,\infty)$ to $[f(0), \infty)$ bijectively. Without loss of generality assume $\rho_0\ge f(0)$. Note
\begin{equation}
\|\Pi_{\rho_0} \Lambda^{\gamma+\mu} u\|^2=\int \mathbbm{1}_{[0,\rho_0)}(1+\rho^2)^{2(\gamma+\mu)} d\langle E_\rho u, u\rangle. 
\end{equation}
On $[0, \rho_0)$, note that $c<(1+\rho^2) < C$ so
\begin{equation}
(1+\rho^2)^{2(\gamma+\mu)}\le C_{\rho_0}\le C(1+\rho^2)^{2 (\mu-\gamma)}.
\end{equation}
Furthermore $f(\rho^2)$ is bounded on $[0, \rho_0)$, so there exists $\lambda_0'\ge\rho_0$ such that for $\rho\in[0,\rho_0)$ and all $\lambda>\lambda_0'$
\begin{equation}
(1+\rho^2)^{2(\gamma+\mu)}\le C_{\rho_0}\le C(1+\rho^2)^{2 (\mu-\gamma)} \lambda^{-4}\abs{f(\rho^2)-\lambda^2}^2.
\end{equation} 
Thus
\begin{equation}\label{3l33}
\|\Pi_{\rho_0} \Lambda^{\gamma+\mu} u\|\le C\lambda^{-2}\|(f(P)-\lambda^2) \Lambda^{\mu-\gamma} u\|
\end{equation}
uniformly for all $\lambda>\lambda_0'$. 

Now for $s'>\rho_0$ we have
\begin{equation}
f(s')=\int_{\rho_0}^{s'}f'(s)~ds+f(\rho_0)\ge \frac{K}{\alpha}(s')^{\alpha}-C_{\rho_0}.
\end{equation}
By picking a large $\lambda_0'$, we have uniformly for $\lambda>\lambda_0'$ that
\begin{equation}
\lambda^2=f(\tilde\lambda^2)\ge \frac{K}{2\alpha}\tilde\lambda^{2\alpha}.
\end{equation}
Then \eqref{3l33} reduces to, uniformly for $\lambda>\lambda_0'$, 
\begin{equation}
\|\Pi_{\rho_0} \Lambda^{\gamma+\mu} u\|\le C\tilde\lambda^{-2\alpha}\|(f(P)-\lambda^2) \Lambda^{\mu-\gamma} u\|, \label{Szero}
\end{equation}
where $C$ does not depend on $\rho,\tilde\lambda$. 

3. Observe that the mean value theorem implies for any $\rho, \tilde\lambda\in [\rho_0, \infty)$ we have
\begin{equation}
f(\rho^2)-\lambda^2=f'(\tilde\eta^2)(\rho^2-f^{-1}(\lambda^2))=f'(\tilde\eta^2)(\rho^2-\tilde\lambda^2)
\end{equation}
where $\tilde\eta^2$ is in between $\rho^2$ and $\tilde\lambda^2$. Since $f'(s) \ge Ks^{(\fpow-1)}$, for any $\rho, \tilde\lambda\in[\rho_0, \infty)$ we have
\begin{equation}\label{MVT}
K\tilde\eta^{2(\fpow-1)}\abs{\rho^2-\tilde\lambda^2} \leq \abs{f(\rho^2)-\lambda^2}, 
\end{equation}
where $K$ does not depend on $\rho$ or $\lambda$. 

4. Regime $S_{\rho}$: $\rho\ge \frac{3}{2}\tilde\lambda$. For $\rho \in S_{\rho},$ we have $\frac{3 \tilde \lambda}{2}\leq \rho$ and $\tilde \lambda \leq \tilde\eta \leq \rho$. Thus 
\begin{equation}
\frac{5}{9} \rho^2 \leq |\rho^2 -\tilde \lambda^2|. 
\end{equation} 
and so by \eqref{MVT}
\begin{equation}\label{MVTrho}
\tilde\eta^{2(\fpow-1)}  \rho^2 \leq C\tilde\eta^{2(\fpow-1)} \abs{\rho^2 -\tilde \lambda^2}\leq  CK^{-1} \abs{f(\rho^2)-\lambda^2}.
\end{equation}
The argument is slightly different depending on the sign of $\fpow-1$.

4a. When $\fpow-1<0$, we have $\rho^{2(\fpow-1)} \leq \tilde\eta^{2(\fpow-1)}$. So by \eqref{MVTrho} we have
\begin{equation}\label{mvtrho1}
\rho^{2\fpow} \leq \tilde\eta^{2(\alpha-1)} \rho^2 \leq C \abs{f(\rho^2)-\lambda}^2.
\end{equation}
To estimate the left, recall $2 \gamma \leq \fpow$, so $2 \fpow-4 \gamma  \geq 0$. Then $\tilde \lambda^{2\fpow-4\gamma} \leq \rho^{2\fpow-4\gamma}$ and
\begin{equation}
(1+\rho^2)^{\gamma + \mu} (1+\rho^2)^{-\mu+\gamma} \tilde{\lambda}^{2\fpow-4\gamma}= (1+\rho^2)^{2\gamma} \tilde \lambda^{2\fpow-4\gamma}\le C\rho^{4\gamma} \tilde \lambda^{2\fpow-4\gamma} \le C\rho^{2\fpow}.
\end{equation}
Combining this with \eqref{mvtrho1}, we arrive at
\begin{equation}\label{dilrho1}
(1+\rho^2)^{\gamma+\mu} (1+\rho^2)^{-\mu+\gamma} \tilde{\lambda}^{2\fpow-4\gamma}\leq C\rho^{2\fpow} \le C\abs{f(\rho^2)-\lambda^2}.
\end{equation}

4b.  When $\fpow -1 \geq 0$, we have $\tilde{\lambda}^{2(\fpow-1)} \leq \tilde\eta^{2(\fpow-1)}$. So by \eqref{MVTrho} we have
\begin{equation}\label{mvtrho2}
\tilde \lambda^{2(\fpow-1)} \rho^2 \leq \tilde\eta^{2(\alpha-1)} \rho^2 \leq C \abs{f(\rho^2)-\lambda}^2.
\end{equation}
To estimate the left, recall $2\gamma \leq 1$, so $2-4\gamma \geq 0$. Then $\tilde \lambda^{2-4\gamma}\le \rho^{2-4\gamma}$ and 
\begin{equation}
(1+\rho^2)^{\gamma+\mu} (1+\rho^2)^{-\mu+\gamma} \tilde \lambda^{2-4\gamma} \le (1+\rho^2)^{2\gamma} \tilde \lambda^{2-4\gamma}\le C\rho^{4 \gamma} \tilde \lambda^{2-4\gamma} \leq C\rho^{2}.
\end{equation}
Multiplying by $\tilde\lambda^{2(\alpha-1)}$ and applying \eqref{mvtrho2} gets us
\begin{equation}\label{dilrho2}
(1+\rho^2)^{\gamma+\mu} (1+\rho^2)^{-\mu+\gamma} \tilde{\lambda}^{2\fpow-4\gamma} \leq C\tilde{\lambda}^{2(\fpow-1)} \rho^2\le C \abs{f(\rho^2)-\lambda^2}.
\end{equation}

4c. Regardless of the sign of $\fpow-1$, by \eqref{dilrho1} and \eqref{dilrho2}, on $S_{\rho}$ we have
\begin{equation}
(1+\rho^2)^{\gamma+\mu} \leq C \tilde \lambda^{4\gamma-2\fpow} (1+\rho^2)^{\mu-\gamma} \abs{f(\rho^2)-\lambda^2}.
\end{equation}
Therefore we have
\begin{equation}
\|\Pi_{\rho} \Lambda^{\gamma+\mu} u\| \leq C \tilde \lambda^{2(2\gamma-\fpow)} \|\Pi_{\rho} (f(P)-\lambda^2)  \Lambda^{\mu-\gamma} u\| \le C \tilde \lambda^{2(2\gamma-\fpow)} \|(f(P)-\lambda^2) \Lambda^{\mu-\gamma} u\|, \label{Srho}
\end{equation}
where $C$ does not depend on $\rho,\tilde\lambda$.

5. Regime $S_{\tilde\lambda}$: $\rho_0\le \rho< \frac{1}{2}\tilde\lambda$. For $\rho \in S_{\tilde \lambda}$ we have $\rho < \frac{1}{2}\tilde \lambda$ and $\rho < \tilde\eta < \tilde \lambda$. Thus 
\begin{equation}
\frac{3}{4}\tilde\lambda^2 \leq \abs{\rho^2-\tilde\lambda^2}.
\end{equation}
So by \eqref{MVT} we have
\begin{equation}
\tilde\eta^{2(\fpow-1)} \tilde \lambda^2
\leq C\tilde\eta^{2(\fpow-1)} \abs{\rho^2 -\tilde \lambda^2} \leq CK^{-1}
\abs{f(\rho^2)-\lambda^2}.
\end{equation}
This implies
\begin{equation}\label{dillambda}
(1+\rho^2)^{\gamma+\mu} \leq C (1+\rho^2)^{\gamma+\mu}\tilde\eta^{-2(\fpow-1)}\tilde \lambda^{-2}\tilde\eta^{2(\fpow-1)} \tilde \lambda^2\le  C(1+\rho^2)^{\gamma+\mu} \tilde\eta^{-2(\alpha-1)} \tilde \lambda^{-2} \abs{f(\rho^2) -\lambda^2}.
\end{equation}
Again there are separate cases depending on the sign of $\fpow-1$. 

5a. When $\fpow-1 < 0$, we have $\tilde\eta^{-2(\fpow-1)} \leq \tilde{\lambda}^{-2(\fpow-1)}$. Hence \eqref{dillambda} becomes
\begin{equation}
(1+\rho^2)^{\gamma+\mu} \leq C(1+\rho^2)^{\gamma+\mu} \tilde{\lambda}^{2(1-\fpow)} \tilde{\lambda}^{-2} \abs{f(\rho^2)-\lambda^2}.
\end{equation}
We then obtain
\begin{align}
(1+\rho^2)^{\gamma+\mu} &\leq C(1+\rho^2)^{\mu-\gamma} \tilde{\lambda}^{2(2\gamma-\fpow)} \abs{f(\rho^2) -\lambda^2} \label{dillambda2},
\end{align}
by noting $\rho^{4\gamma}\le \tilde\lambda^{4\gamma}$.

5b. When $\fpow-1\ge 0$, we have $\tilde\eta^{-2(\fpow-1)} < \rho^{-2(\fpow-1)}$. Hence \eqref{dillambda} becomes
\begin{equation}
(1+\rho^2)^{\gamma+\mu} \leq C(1+\rho^2)^{\mu-\gamma}\rho^{4\gamma-2(\fpow-1)} \tilde \lambda^{-2} \abs{f(\rho^2) - \lambda^2}.
\end{equation}
Here we can freely give up on $\rho$ having a negative power to make use of $\rho< \tilde \lambda$:
\begin{equation}
\rho^{4\gamma-2(\alpha-1)} \leq C\rho^{2(2\gamma+1-\fpow)_+} \leq C\tilde{\lambda}^{2(2\gamma+1-\fpow)_+}. 
\end{equation}
Hence we obtain
\begin{equation}\label{dillambda1}
(1+\rho^2)^{\gamma+\mu} \leq C(1+\rho^2)^{\mu-\gamma} \tilde \lambda^{-2+2(2\gamma-\fpow+1)_+} \abs{f(\rho^2) -\lambda^2}.
\end{equation}

5c. Regardless of the sign of $\fpow-1$, by \eqref{dillambda2} and \eqref{dillambda1}, we have on $S_{\tilde \lambda}$
\begin{equation}
(1+\rho^2)^{\gamma+\mu} \leq C \tilde \lambda^{-2 + 2(2\gamma+1-\fpow)_+} (1+\rho^2)^{\mu-\gamma} \abs{f(\rho^2) -\lambda^2}.
\end{equation}
Therefore
\begin{align}
\|\Pi_{\tilde\lambda}\Lambda^{\gamma+\mu}u\|&\le C\tilde\lambda^{-2+2(2\gamma+1-\fpow)_+}\| \Pi_{\tilde\lambda} (f(P)-\lambda^2)\Lambda^{\mu-\gamma} u\|\\
&\le C\tilde\lambda^{-2+2(2\gamma+1-\fpow)_+}\|(f(P)-\lambda^2) \Lambda^{\mu-\gamma} u\|. \label{Slambda}
\end{align}

6. Regime $S_\sim$: $\frac{1}{2}\tilde\lambda\le \rho\le \frac{3}{2}\tilde\lambda$. Here we mainly use the control estimate. Combining \eqref{Szero}, \eqref{Srho}, \eqref{Slambda} gives us
\begin{equation}\label{3l24}
\|(\id-\Pi_\sim)\Lambda^{\gamma+\mu}u\|\le C\tilde\lambda^{-2+2(2\gamma+1-\fpow)_+}\|(f(P)-\lambda^2)\Lambda^{\mu-\gamma} u\|,
\end{equation}
uniformly for all $u\in H_{\alpha+\mu-\gamma}$. For $\rho\in S_\sim$, $\rho^2\in [\frac{1}{4}\tilde\lambda^2, \frac{9}{4}\tilde\lambda^2]$, $\tilde\eta^2$ is comparable to $\tilde\lambda^2$ and
\begin{equation}\label{MVTsim}
\rho^2-\tilde\lambda^2\le K^{-1}\tilde\lambda^{2(1-\alpha)}(f(\rho^2)-\lambda^2).
\end{equation}
Here we used \eqref{MVT}. This implies
\begin{align}\label{3l32}
\|(P-\tilde{\lambda}^2 )&\Lambda^{\mu-\gamma}\Pi_{\sim}u\|^2=\int_{\frac{1}{2}\tilde\lambda}^{\frac{3}{2}\tilde\lambda}(\rho^2-\tilde{\lambda}^2)^2(1+\rho^2)^{2(\mu-\gamma)}\ d\langle E_\rho u, u\rangle\\
&\le K^{-2}\tilde\lambda^{4(1-\fpow)}\int_{\frac{1}{2} \tilde{\lambda}}^{\frac{3}{2} \tilde{\lambda}} (1+ \rho^2)^{2(\mu-\gamma)}(f(\rho^2)-\lambda^2)^2\ d\langle E_\rho u, u\rangle\\
&=K^{-2}\tilde\lambda^{4(1-\fpow)}\|(f(P)-\lambda^2)\Lambda^{\mu-\gamma}\Pi_{\sim}u\|^2.
\end{align}
Now we claim that $\Pi_\sim u\in H_s$ for any $s>0$: indeed, 
\begin{equation}
\|\Lambda^s \Pi_\sim u\|^2=\int_{0}^{\infty}(1+\rho^2)^{2s}\mathbbm{1}_{[\frac{1}{2}\tilde\lambda, \frac{3}{2}\tilde\lambda]}(\rho)~d\langle E_\rho u, u\rangle\le C_{\tilde\lambda}\|u\|^2. 
\end{equation}
That $\Pi_\sim u\in H_{1+\mu-\gamma}$ allows us to apply the control estimate \eqref{1l1}. Thus for any $\lambda\ge \lambda_0$,
\begin{equation}\label{dilhautus}
\|\Pi_{\sim}u\|\le M(\tilde\lambda)\tilde\lambda^{-1}\|(P-\tilde{\lambda}^2)\Lambda^{\mu-\gamma} \Pi_{\sim}u\|+m(\tilde\lambda)\|Q\Lambda^\mu\Pi_{\sim}u\|_Y+e(\tilde\lambda)\|\Lambda^{-\delta}\Pi_{\sim}u\|,
\end{equation}
uniformly for any $u\in H_{\alpha+\mu-\gamma}$, where $C$ does not depend on $\lambda$.  Since $Q\in \mathcal{L}(H_{\gamma}, Y)$
\begin{equation}
\|Q\Lambda^\mu(1-\Pi_{\sim})u\|_Y \le \|(1-\Pi_{\sim})\Lambda^{\gamma+\mu}u\| \le C\tilde\lambda^{-2+2(2\gamma+1-\fpow)_+}\|(f(P)-\lambda^2) \Lambda^{\mu-\gamma} u\|,
\end{equation}
uniformly for all $u\in H_{\alpha+\mu-\gamma}$, where we used \eqref{3l24}. We estimate
\begin{multline}
\|Q\Lambda^\mu\Pi_{\sim}u\|_Y \le \|Q\Lambda^\mu u\|_Y +\|Q\Lambda^\mu(1-\Pi_{\sim})u\|_Y \\
\le \|Q\Lambda^\mu u\|_Y+C\tilde\lambda^{-2+2(2\gamma+1-\fpow)_+}\|(f(P)-\lambda^2)\Lambda^{\mu-\gamma} u\|.
\end{multline}
Combining it with \eqref{3l32}, \eqref{dilhautus} reduces to 
\begin{align}
\|\Pi_{\sim}u\|&\le M(\tilde\lambda)\tilde\lambda^{-1}\|(P-\tilde \lambda^2)\Lambda^{\mu-\gamma} u\|+m(\tilde\lambda)\|Q\Lambda^\mu u\|_Y+C\|\Lambda^{-\delta}u\|
\\ \le & C(M(\tilde\lambda)\tilde\lambda^{1-2\fpow}+m(\tilde\lambda)\tilde\lambda^{-2+2(2\gamma+1-\fpow)_+})\|(f(P)-\lambda^2)\Lambda^{\mu-\gamma} u\|\\
&+ m(\tilde\lambda)\|Q\Lambda^\mu u\|_Y+ e(\tilde\lambda)\|\Lambda^{-\delta}u\|. 
\end{align}
Bring in \eqref{3l24} to observe
\begin{align}
\|u\|\le &C(M(\tilde\lambda)\tilde\lambda^{1-2\fpow}+m(\tilde\lambda)\tilde\lambda^{-2+2(2\gamma+1-\fpow)_+})\|(f(P)-\lambda^2)\Lambda^{\mu-\gamma} u\|\\
&+m(\tilde\lambda)\|Q\Lambda^\mu u\|_Y+e(\tilde\lambda)\|\Lambda^{-\delta}u\|,
\end{align}
as desired. 
\end{proof}

\begin{proof}[Proof of Theorem \ref{thmdilate}]
Let $f(s)=s^{\alpha}$, then $\tilde\lambda=\lambda^{1/\alpha}$. By Proposition \ref{3t4} we have
\begin{align}
\|v\|\le& C(M(\tilde\lambda)\tilde\lambda^{1-2\fpow}+m(\tilde\lambda)\tilde\lambda^{-2+2(2\gamma+1-\fpow)_+})\|(P^\alpha-\lambda^2)\Lambda^{\mu-\gamma} v\|_H\\
&+m(\tilde\lambda)\|Q\Lambda^{\mu}v\|_Y+e(\tilde\lambda)\|\Lambda^{-\delta}v\|_H,
\end{align}
for any $v\in H_{1+\mu-\gamma}$. Let $v=\Lambda^{-\mu}\Lambda^{\mu/\alpha}_{(\alpha)} u$. We have
\begin{align}\label{3l25}
\|\Lambda^{-\mu}\Lambda^{\mu/\alpha}_{(\alpha)} u\|_H\le& C(M(\tilde\lambda)\tilde\lambda^{1-2\fpow}+m(\tilde\lambda)\tilde\lambda^{-2+2(2\gamma+1-\fpow)_+})\|(P^\alpha-\lambda^2)\Lambda^{-\gamma}\Lambda^{\mu/\alpha}_{(\alpha)} u\|_H\\
&+m(\tilde\lambda)\|Q\Lambda_{(\alpha)}^{\mu/\alpha} u\|_Y+e(\tilde\lambda)\|\Lambda^{-\delta}\Lambda^{-\mu}\Lambda^{\mu/\alpha}_{(\alpha)} u\|_H,
\end{align}
all terms under the $H$-norm of which we will further simplify. Now note the algebraic inequality
\begin{equation}
C^{-1}(1+\rho^{2\alpha})^{\frac{1}{\alpha}}\le (1+\rho^2)\le C(1+\rho^{2\alpha})^{\frac{1}{\alpha}}
\end{equation}
implies
\begin{equation}
C^{-1}\le (1+\rho^2)^{-\mu}(1+\rho^{2\alpha})^{\frac{\mu}{\alpha}}\le C. 
\end{equation}
We then have
\begin{equation}
\|u\|_H\le C\|\Lambda^{-\mu}\Lambda^{\mu/\alpha}_{(\alpha)} u\|_H, \ \|\Lambda^{-\delta}\Lambda^{-\mu}\Lambda^{\mu/\alpha}_{(\alpha)} u\|_H\le C\|\Lambda_{(\alpha)}^{-\delta/\alpha}u\|_H.
\end{equation}
On the other hand, we can estimate
\begin{multline}
\|(P^\alpha-\lambda^2)\Lambda^{-\gamma}\Lambda^{\mu/\alpha}_{(\alpha)} u\|_H^2=\int_0^{\infty}(\rho^{2\alpha}-\lambda^2)^2(1+\rho^2)^{-2\gamma}(1+\rho^{2\alpha})^{\frac{2\mu}{\alpha}} d\langle E_\rho u, u\rangle\\
\le \int_0^{\infty}(\rho^{2\alpha}-\lambda^2)^2(1+\rho^{2\alpha})^{2(\mu-\gamma)/\alpha} d\langle E_\rho u, u\rangle\le C\|(P^\alpha-\lambda^2)\Lambda^{(\mu-\gamma)/\alpha}_{(\alpha)} u\|_H^2.
\end{multline}
Bring those estimates back into \eqref{3l25} to obtain
\begin{multline}
\|u\|_H\le C(M(\tilde\lambda)\tilde\lambda^{1-2\fpow}+m(\tilde\lambda)\tilde\lambda^{-2+2(2\gamma+1-\fpow)_+})\|(P^\alpha-\lambda^2)\Lambda^{(\mu-\gamma)/\alpha}_{(\alpha)} u\|_H\\
+Cm(\tilde\lambda)\|Q\Lambda_{(\alpha)}^{\mu/\alpha} u\|_Y+Ce(\tilde\lambda)\|\Lambda_{(\alpha)}^{-\delta/\alpha}u\|_H.
\end{multline}
Note $\tilde\lambda=\lambda^{1/\alpha}$ to conclude the proof. 
\end{proof}

\subsection{Theorem \ref{t4}: non-uniform Hautus test}\label{s3-7}
We will give a complementary proof of Theorem \ref{t4} proposed in \S \ref{s2-1}. To prove it, one can indeed follow our remarks in \cite[Remark 2.29]{kw22} to apply the theorems \cite{cpsst19} to obtain $\gamma$-dependent rates. But to make the paper self-contained, we here give a semiclassical proof via only looking at the  resolvent estimates for the semiclassical damped wave operator $P_h=h^2P_\lambda=h^2-ihQ^*Q-1$, where we semiclassicalise the system using $h=\abs{\lambda}^{-1}$. The proof mostly follows the ideas using the wavepacket decomposition as in \cite{cpsst19}. 

From now we abuse the notation and denote by $ M $ either $M(\abs{\lambda})$ or $M(h^{-1})$, and similarly by $m$ either $m(\abs{\lambda})$ or $m(h^{-1})$. Fix $h_0,\epsilon>0$ small. The assumption \eqref{2l1} now reads
\begin{equation}\label{al2}
\|u\|_{H}\le M h^{-1}\|(h^2P-1)u\|_H+m\|Qu\|_Y
\end{equation}
uniformly for $h\in (0, h_0)$. The major goal of this section is to prove a semiclassical resolvent estimate:
\begin{proposition}[Semiclassical resolvent estimate]\label{at1}
Let $Q\in\mathcal{L}(H_\gamma, Y)$ for $\gamma\in[0,\frac{1}{2}]$. Assume the control estimate \eqref{al2} holds for all $h\in (0, h_0)$. Then 
\begin{equation}
\|P_h^{-1}\|_{\mathcal{L}(L^2)}\le CM^2m^2h^{-1-8\gamma},
\end{equation}
uniformly for $h\in (0, h_0)$. 
\end{proposition}
\begin{proof}[Proof of Theorem \ref{t4} via Proposition \ref{at1}]
Proposition \ref{at1} implies
\begin{equation}
\|P_\lambda^{-1}\|\le CM^2m^2\lambda^{-1+8\gamma}, 
\end{equation}
for real $\lambda$ with $\abs{\lambda}$ large. Invoke Lemma \ref{3t5} to conclude the proof.
\end{proof}

To prove Proposition \ref{at1}, we will look at the wavepackets of $u\in H$. For $h\in(0, h_0)$, we introduce the associated semiclassical spectral projectors
\begin{equation}\label{al1}
\Pi_h=\int^{h^{-1}+\epsilon /M}_{h^{-1}-\epsilon /M} ~dE_\rho, \ \Pi_h^{\perp}=\id-\Pi_h. 
\end{equation}
Spectral projectors $\Pi_h, \Pi_h^{\perp}$ commute with $P$. The projector $\Pi_h$ spectrally localises an function into $[h^{-1}-\epsilon /M, h^{-1}+\epsilon /M]$, a semiclassically small spectral window. We call those localised functions $\Pi_hu$ wavepackets of $u\in H$. They are the part of $u$ that oscillates at frequencies close to $h^{-1}$. 
\begin{lemma}[Wavepackets are quasimodes]
For any $u\in H$, we have
\begin{gather}\label{2l3}
\|(h^2P-1)\Pi_hu\|\le 3M^{-1}\epsilon h\|\Pi_h u\|,\\
\label{2l4}
\|(h^2P-1)\Pi_h^{\perp}u\|\ge M^{-1}\epsilon h\|\Pi_h^{\perp} u\|,
\end{gather}
uniformly for $h\in(0, h_0)$.
\end{lemma}
\begin{proof}
Note when $h$ and $\epsilon$ are small, if $\rho\in [h^{-1}-\epsilon /M, h^{-1}+\epsilon /M]$, then $\abs{\rho^2-h^{-2}}\le \frac{3\epsilon}{Mh}$, and thus
\begin{equation}
\|\Pi_h(h^2P-1)u\|=h^2\left\|\int^{h^{-1}+\epsilon/M}_{h^{-1}-\epsilon/M} (\rho^2-h^{-2})dE_\rho(u)\right\|\le 3M^{-1}\epsilon h\|\Pi_h u\|.
\end{equation}
On the other hand, since on $\{\abs{\rho-h^{-1}}\ge\epsilon /M\}$ we have $\abs{\rho^2-h^{-2}}\ge \frac{\epsilon}{Mh}$, 
\begin{equation}
\|\Pi_h^{\perp}(h^2P-1)u\|=h^2\left\|\int_{\abs{\rho-h^{-1}}\ge \epsilon/M} (\rho^2-h^{-2})dE_\rho(u)\right\|\ge M^{-1}\epsilon h\|\Pi_h^{\perp} u\|,
\end{equation}
as desired. 
\end{proof}

We show that wave packets are observable by $Q$ under the control estimate \eqref{al2}. 
\begin{lemma}[Wavepacket control]
Assume the control estimate \eqref{al2} holds. Then for small $\epsilon>0$ in \eqref{al1} we have
\begin{equation}\label{al3}
\|\Pi_h u\|\le Cm \|Q\Pi_h u\|,
\end{equation}
uniformly for $h\in(0, h_0)$ and all $u\in H$. 
\end{lemma}
\begin{proof}
Apply \eqref{al2} to $\Pi_h u$ to see
\begin{equation}
\|\Pi_h u\|_{H}\le Mh^{-1}\|(h^2P-1)\Pi_h u\|_{H}+m\|Q \Pi_h u\|_{Y}\le 3\epsilon \|\Pi_h u\|+m\|Q \Pi_h u\|_{Y},
\end{equation}
where we used \eqref{2l3}. Choosing small $\epsilon$ gives the desired estimate. 
\end{proof}
From now on we fix $\epsilon>0$ small such that \eqref{al3} holds. 
\begin{proof}[Proof of Proposition \ref{at1}]
1. Let $P_h u=(h^2P-ihQ^*Q-1)u=f$. Consider the semiclassically elliptic operator $\Lambda_h=h^2 P+1$. Its inverse $\Lambda_h^{-1}$ is a semiclassically bounded operator on $H$. Observe $\Lambda_h=P_h+2+ihQ^*Q$, and thus
\begin{equation}\label{2l2}
 \id=\Lambda_h^{-1}P_h+2\Lambda_h^{-1}+ih \Lambda_h^{-1}Q^*Q. 
\end{equation} 
Note $h^2P-1$, $\Lambda_h$ and $\Pi_h^{\perp}$ commute mutually. Thus by $P_h=h^2P-1-ihQ^*Q$, \eqref{2l2} gives
\begin{equation}\label{2l5}
\Pi_h^{\perp}=\Lambda_h^{-1}\Pi_h^{\perp} P_h+2\Lambda_h^{-1}\Pi_h^{\perp}+ih \Pi_h^{\perp}\Lambda_h^{-1}Q^*Q. 
\end{equation}

2. We want to estimate $\Pi_h^{\perp} u$. Use \eqref{2l4} to estimate
\begin{equation}
\| \Lambda_h^{-1} \Pi_h^{\perp} u\|\le C M h^{-1}\|(h^2P-1)\Lambda_h^{-1}\Pi_h^{\perp} u\|.
\end{equation}
Thus apply \eqref{2l5} to $u$ gives
\begin{equation}\label{2l6}
\|\Pi_h^{\perp}u\| \le C\|f\|+C M h^{-1}\|(h^2P-1)\Lambda_h^{-1}\Pi_h^{\perp} u\| +h \|\Lambda_h^{-1}Q^*\|\|Q u\|,
\end{equation}
where we used the boundedness of $\Pi_h^{\perp}$ and $\Lambda_h^{-1}$ on $H$. Note
\begin{equation}
(h^2P-1)\Pi_h^{\perp}=\Pi_h^{\perp} P_h+ih \Pi_h^{\perp} Q^*Q
\end{equation}
and we can estimate
\begin{equation}
\|(h^2P-1)\Lambda_h^{-1}\Pi_h^{\perp} u\| = \|\Pi_h^{\perp} \Lambda_h^{-1}P_hu+ih \Pi_h^{\perp} \Lambda_h^{-1}Q^*Qu\|\le \|f\|+h\|\Lambda_h^{-1}Q^*\|\|Qu\|.
\end{equation}
Hence \eqref{2l6} further reduces to
\begin{equation}\label{al4}
\|\Pi_h^{\perp}u\| \le C M h^{-1}\|f\| +iM \|\Lambda_h^{-1}Q^*\|\|Q u\|. 
\end{equation}

3. Motivated by \eqref{al3}, we now want to estimate $\Pi_h u$ via estimating $Q\Pi_h^{\perp} u$. Note \eqref{2l5} gives 
\begin{equation}\label{2l8}
Q\Pi_h^{\perp}=Q\Lambda_h^{-1}\Pi_h^{\perp}P_h+2Q\Lambda_h^{-1}\Pi_h^{\perp}+ih Q\Pi_h^{\perp} \Lambda_h^{-1}Q^*Q.
\end{equation}
The last term on the right can be written as
\begin{equation}\label{2l7}
Q\Pi_h^{\perp}\Lambda_h^{-1}Q^*Q=Q\Lambda_h^{-1}Q^*Q-Q\Pi_h\Lambda_h^{-1}Q^*Q=Q\Lambda_h^{-1}Q^*Q-Q\Lambda_h^{-1}(\Lambda_h\Pi_h)\Lambda_h^{-1}Q^*Q,
\end{equation}
where inside the last term we have the bound
\begin{equation}
\|\Lambda_h\Pi_h\|\le \|(h^2P-1)\Pi_h\| +2\|\Pi_h\|\le CM^{-1}h+2\le C,
\end{equation}
due to \eqref{2l3}. Note that the adjoint of $\Lambda_h^{-1}Q^*\in \mathcal{L}(H)$ is $Q\Lambda_h^{-1}\in \mathcal{L}(H)$. This implies $\|\Lambda_h^{-1}Q^*\|=\|Q\Lambda_h^{-1}\|$. We can further simplify \eqref{2l7} to
\begin{equation}
\|Q\Pi_h^{\perp}\Lambda_h^{-1}Q^*Qu\|\le \left(\|Q\Lambda_h^{-1}Q^*\| +\|Q\Lambda_h^{-1}\|^{2}\right)\|Qu\|.
\end{equation}
Bring in \eqref{al4} and revisit \eqref{2l8} to observe
\begin{multline}
\|Q\Pi_h^{\perp} u\|\le C\|Q\Lambda_h^{-1}\|(\|f\|+\|\Pi_h^{\perp} u\|)+\left(\|Q\Lambda_h^{-1}Q^*\| +\|Q\Lambda_h^{-1}\|^{2}\right)\|Qu\|\\
\le CMh^{-1}\|Q\Lambda_h^{-1}\|\|f\|+CM\|Q\Lambda_h^{-1}\|^2\|Qu\|+C\|Q\Lambda_h^{-1}Q^*\|\|Qu\|
\end{multline}
Now apply \eqref{al3} to see
\begin{multline}
\|\Pi_h u\|\le Cm \|Q\Pi_h u\|\le Cm\|Qu\|+Cm\|Q\Pi_h^{\perp} u\|\\
\le CMmh^{-1}\|Q\Lambda_h^{-1}\|\|f\|+Cm(1+M\|Q\Lambda_h^{-1}\|^2+\|Q\Lambda_h^{-1}Q^*\|)\|Qu\|
\end{multline}
and thus from \eqref{al4} we have
\begin{multline}\label{al5}
\|u\|\le \|\Pi_h u\|+\|\Pi_h^{\perp} u\|\le CMh^{-1}(1+m\|Q\Lambda_h^{-1}\|)\|f\|\\+C(M \|Q\Lambda_h^{-1}\|+m+Mm\|Q\Lambda_h^{-1}\|^2+m\|Q\Lambda_h^{-1}Q^*\|)\|Qu\|.
\end{multline}
Note $\langle \rho\rangle^\gamma\langle h\rho\rangle^{-1}\le Ch^{-\gamma}$ implies
\begin{equation}
\|\Lambda^{\gamma}\Lambda_h^{-1}u\|^2=\int_{0}^\infty\langle \rho\rangle^{2\gamma}\langle h\rho\rangle^{-2}~d\langle E_\rho(u), u\rangle \le h^{-2\gamma}\|u\|^2. 
\end{equation}
Thus we have $\|\Lambda_h^{-1}\|_{H\rightarrow H_{\gamma}}\le Ch^{-2\gamma}$, and similarly $\|\Lambda_h^{-1}\|_{H_{-\gamma}\rightarrow H_{\gamma}}\le Ch^{-4\gamma}$. We have
\begin{equation}
\|Q\Lambda_h^{-1}\|_{\mathcal{L}(H)}\le Ch^{-2\gamma}, \ \|Q\Lambda_h^{-1}Q^*\|_{\mathcal{L}(H)}\le Ch^{-4\gamma}. 
\end{equation}
Now we can estimate \eqref{al5}:
\begin{equation}\label{al6}
\|u\|\le CMmh^{-1-2\gamma}\|f\|+CMmh^{-4\gamma}\|Qu\|.
\end{equation}
Pair $P_h u$ with $u$ and take the imaginary part to observe
\begin{equation}
\|Qu\|^2\le h^{-1}\abs{\langle f,u\rangle}\le C\epsilon^{-1}M^2m^2h^{-2-8\gamma}\|f\|^2+\epsilon M^{-2}m^{-2}h^{8\gamma}\|u\|^2.
\end{equation}
Bring into \eqref{al6} and absorb the $\epsilon$-term to see
\begin{equation}
\|u\|\le CM^2m^2h^{-1-8\gamma}\|f\|
\end{equation}
and thus $\|P_h^{-1}\|\le CM^2m^2h^{-1-8\gamma}$ as desired. 
\end{proof}

\bibliographystyle{alpha}
\bibliography{Robib}

\end{document}